\documentclass[11pt]{amsart}
\usepackage[margin=1.1in]{geometry}
\usepackage{setspace}
\usepackage{enumerate}
\usepackage{amssymb}
\usepackage{tikz-cd}
\usepackage{mathrsfs}
\usepackage{extarrows}
\usepackage{url}
\numberwithin{equation}{subsection}
\newtheorem{theorem}[subsubsection]{Theorem}
\newtheorem{lemma}[subsubsection]{Lemma}

\newtheorem{proposition}[subsubsection]{Proposition}
\newtheorem{corollary}[subsubsection]{Corollary}
\newtheorem{innercustomthm}{Theorem}
\newenvironment{customtheorem}[1]
  {\renewcommand\theinnercustomthm{#1}\innercustomthm}
  {\endinnercustomthm}

\newtheorem{innercustomconj}{Conjecture}
\newenvironment{customconjecture}[1]
  {\renewcommand\theinnercustomconj{#1}\innercustomconj}
  {\endinnercustomconj}  
\theoremstyle{definition}
\newtheorem{definition}[subsubsection]{Definition}

\theoremstyle{remark}
\newtheorem{remark}[subsubsection]{Remark}

\let\realequation\equation
\def\equation{\setcounter{equation}{\arabic{subsubsection}}%
   \refstepcounter{subsubsection}%
   \realequation}
   
\makeatletter
\newcommand*\notwithin[2]{%
  \@removefromreset{#1}{#2}%
}
\makeatother

\newcommand{\abs}[1]{\lvert#1\rvert}
\newcommand{\Isog}{\mathrm{Isog}}

\DeclareMathOperator{\defect}{def}
\newcommand{\Gr}{\mathrm{Gr}}
\newcommand{\Sh}{\mathrm{Sh}}
\newcommand{\ord}{\mathrm{ord}}
\newcommand{\Frob}{\mathrm{Frob}}
\newcommand{\sA}{\mathscr{A}}
\newcommand{\sG}{\mathscr{G}}
\newcommand{\sO}{\mathscr{O}}
\newcommand{\sS}{\mathscr{S}}

\newcommand{\bbD}{\mathbb{D}}
\newcommand{\bbF}{\mathbb{F}}

\newcommand{\bbR}{\mathbb{R}}
\newcommand{\bbQ}{\mathbb{Q}}
\newcommand{\bbX}{\mathbb{X}}
\newcommand{\bbZ}{\mathbb{Z}}
\newcommand{\bbMV}{\mathbb{MV}}
\newcommand{\calA}{\mathcal{A}}
\newcommand{\calC}{\mathcal{C}}
\newcommand{\calD}{\mathcal{D}}
\newcommand{\calE}{\mathcal{E}}
\newcommand{\calF}{\mathcal{F}}

\newcommand{\calH}{\mathcal{H}}

\newcommand{\calS}{\mathcal{S}}
\newcommand{\calO}{\mathcal{O}}
\newcommand{\calX}{\mathcal{X}}
\newcommand{\calY}{\mathcal{Y}}
\newcommand{\calZ}{\mathcal{Z}}

\newcommand{\frakD}{\mathfrak{D}}
\newcommand{\frakE}{\mathfrak{E}}
\newcommand{\frakM}{\mathfrak{M}}

\newcommand{\frakS}{\mathfrak{S}}

\newcommand{\GSp}{\mathrm{GSp}}
\newcommand{\RZ}{\mathrm{RZ}}
\usepackage{hyperref}
\begin{document}

\title{Eichler-Shimura Relations for Shimura Varieties of Hodge Type}

\author{Si Ying Lee}
\address{Department of Mathematics, Harvard University, Cambridge, Massachusetts 02138}
\email{sylee@math.harvard.edu}

\begin{abstract}
We show Eicher-Shimura relations for some Shimura varieties of Hodge type, proving a conjecture of Blasius and Rogawski in this case. We show this using a parabolic reduction strategy on the Hecke action on irreducible components of affine Deligne-Lusztig varieties.
\end{abstract}

\maketitle
\setcounter{tocdepth}{1}
\tableofcontents
\section{Introduction}
The classical Eichler-Shimura relation for modular curves relates the Frobenius $(\Frob)$ and Verschiebung $(\mathrm{Ver})$ correspondences with the Hecke correspondence $T_p$ over the special fiber of the modular curve, in the following well known way:
\begin{equation*}
    T_p=\Frob+\mathrm{Ver}.
\end{equation*}

Rearranging terms, we see that this implies that $\Frob$ is a root of the polynomial
\begin{equation*}
    x^2-T_px+p\langle p\rangle.
\end{equation*}
In \cite{BR1994}, Blasius and Rogawski conjectured a generalization of this result for arbitrary Shimura varieties $\Sh(G,X)$ with good reduction. Let $G$ be a connected reductive group over $\mathbb{Q}$, and let $(G,X)$ be a Shimura datum. Fix a prime $p>2$. Let $E$ be the reflex field of the Shimura datum, and $v$ a prime of $E$ above $p$. Let $O_{E,v}$ be the ring of integers of the completion $E_v$ of $E$. We assume the Shimura datum has good reduction at $p$, so $G$ has a reductive model $G_{\mathbb{Z}_p}$ over $\mathbb{Z}_p$. We let $K_p=G_{\mathbb{Z}_p}(\mathbb{Z}_p)$ be the hyperspecial subgroup of $G_{\mathbb{Q}_p}$. We choose some sufficiently small level structure $K^p$, and let $K:=K_pK^p$. $X$ defines a conjugacy class of cocharacters over $E_v$, and we choose a representative $\mu$ which is dominant with respect to some choice of Borel subgroup and maximal torus.

Blasius and Rogawski defined the Hecke polynomial $H_{G,X}$, as a renormalized characteristic polynomial of the irreducible representation of $\hat{G}$ with highest weight $\hat{\mu}$. This is a polynomial with coefficients in the local Hecke algebra $\mathcal{H}(G(\mathbb{Q}_p)//K_p,\mathbb{Q})$. Blasius and Rogawski then conjectured the following generalization of the Eichler-Shimura relation.

\begin{customconjecture}{1.1}
\label{BRconj}
  Consider the $\mathrm{Gal}(\overline{\mathbb{Q}}/E)$-module $V=H^i_{\acute{e}t}(\Sh_K(G,X)_{\bar{E}},\mathbb{Q}_l)$ for some prime $l\neq p$. Let $\Frob$ be the geometric Frobenius at $v$. Then we have 
  \begin{equation*}
    H_{G,X}(\Frob|_V)=0.
  \end{equation*}
\end{customconjecture}

We remark here that Blasius and Rogawski actually conjectured this result for the intersection cohomology of the Baily-Borel compactification of $\Sh_K(G,X)$, but we use this formulation in line with earlier work (\cite{FC1990},\cite{B2002},\cite{Kos2014}, etc.) on the conjecture. In this paper, we prove the congruence relation on the level of algebraic correspondences, which also implies the analogous theorem for cohomology with compact support and intersection cohomology (see Section \ref{subsubsection:7.2} for more details).

The main theorem of this paper is the following.
\begin{customtheorem}{1.2}
  \label{mainthm}
  Conjecture \ref{BRconj} holds if $(G,X)$ is of Hodge type, and satisfies either of the two conditions:
  \begin{enumerate}
    \item The Shimura variety is a Hilbert-Blumenthal modular variety, and $p$ is inert, or
     \item For every unramified $\sigma$-conjugacy class $[b]$ in $B(G,\upsilon)$, the pair $([b],\upsilon)$ is Hodge-Newton decomposable for $M_{[b]}$.
  \end{enumerate}
\end{customtheorem}

We define Hodge-Newton decomposability in \eqref{defnHN}. We state some situations in which Condition (2) holds. For example, if the group $G$ is split over $\bbQ_p$, then this condition holds, since in this case there is only one unramified $\sigma$-conjugacy class. Another situation in which Condition (2) holds is when the Shimura variety is fully Hodge-Newton decomposable (see \cite{GHN19} for definition), in which case there are always two $\sigma$-conjugacy classes, namely the $\mu$-ordinary one and the basic one. 

In order to show Theorem \ref{mainthm}, we generalize a construction of Faltings-Chai \cite{FC1990} for Shimura varieties of Hodge type. In particular, following the work of Kisin (c.f. \cite{K2010}), for the Shimura variety $\Sh_K(G,X)$ over $E$ we have a canonical integral model $\mathscr{S}_K(G,X)$ over $O_{E,(v)}$. Thus, we can define a moduli stack $p-\Isog$ of $p$-quasi-isogenies (with extra conditions) between points on $\mathscr{S}_K(G,X)$.

Let $\kappa$ denote the residue field of $E_v$, and $n$ be the degree $[\kappa:\bbF_p]$. For a scheme $S$ over $O_{E,(v)}$, let $\mathbb{Q}[p-\Isog\otimes S]$ be the algebra of top dimensional irreducible components of the moduli stack  $p-\Isog\otimes S$, with multiplication given by composition of isogenies. Generalizing the constructions of Chai and Faltings, we have a map
\begin{equation*}
  h:\mathcal{H}(G(\mathbb{Q}_p)//K_p,\mathbb{Q})\rightarrow\mathbb{Q}[p-\Isog\otimes E].
\end{equation*}
Composing with the specialisation of cycles map gives a map, which we also denote by $h$,
\begin{equation*}
  h:\mathcal{H}(G(\mathbb{Q}_p)//K_p,\mathbb{Q})\rightarrow\mathbb{Q}[p-\Isog\otimes \kappa].
\end{equation*}

In this paper, we prove the following result, which implies Theorem \ref{mainthm}. 

\begin{customtheorem}{1.3}
  \label{thm:algcycles}  
  Let $(G,X)$ be a Shimura datum of Hodge type, with the same assumptions as in Theorem 1.2. Consider the polynomial $H_{G,X}$, viewed as a polynomial with coefficients inside $\mathbb{Q}[p-\operatorname{Isog}\otimes \kappa]$ via the morphism $h$. Then the element $\Frob$ lies in the center of this ring and the following relation holds in $\mathbb{Q}[p-\operatorname{Isog}\otimes \kappa]$:
  \begin{equation}
  \label{eqn:BR}
    H_{G,X}(\Frob)=0,
  \end{equation}
  where $\Frob\in\mathbb{Q}[p-\operatorname{Isog}\otimes \kappa]$ is the Frobenius section, i.e. the formal sum of irreducible components of $p-\operatorname{Isog}\otimes \kappa$ consisting of isogenies which are the relative Frobenius morphism.
\end{customtheorem}

\subsection*{Previously known results}
Other than the case of Siegel modular varieties, where the conjecture was proved by Chai and Faltings, this conjecture was previously known only under more restrictive conditions. Bu\"{e}ltel proved this conjecture when the Shimura variety satisfies a no vertical components (NVC) condition, while Wedhorn \cite{W2000} verified this conjecture in the case where the Shimura variety is of PEL type and the group $G$ is split over $\mathbb{Q}_p$. All of these results focused on the ordinary locus, and in particular require that the ordinary locus $p-\Isog^{ord}\otimes\kappa$ is dense in $p-\Isog\otimes\kappa$, a condition which does not hold in general. Moreover, the study of the Hecke action on the ordinary locus is considerably simpler, because every ordinary point admits a canonical lift, and in fact every quasi-isogeny mod $p$ lifts to characteristic 0. Such a property does not hold for other Newton strata.

Beyond the aforementioned results, two more cases were previously known, namely the case of the unitary Shimura variety with signature $(1,n-1)$ was verified in \cite{BW2006} (even $n$) and \cite{Kos2014} (odd $n$). More recently, \cite{Li18} provided a proof for the Shimura variety attached to the spinor similitude group $\mathrm{GSpin}(2,n)$ for any $n$. Both these cases relied on an explicit description of the basic locus. We remark here that all these cases are covered by condition (2).

\subsection*{Method of proof}
We first introduce a stratification on $p-\Isog\otimes\kappa$ indexed by $\sigma$-conjugacy classes $B(G)$, similar to the Newton stratification of the special fiber $\sS_\kappa$ of $\sS_K(G,X)$. For any top-dimensional irreducible component $C$ of $p-\Isog\otimes\kappa$, let $C^{[b]}$ be the subscheme of points corresponding to isogenies between $p$-divisible groups with $\sigma$-conjugacy class $[b]$. If $C^{[b]}$ is dense in $C$, we say $C$ is \emph{$[b]$-dense}.

We consider now the set of all $[b]\in B(G)$ such that there exists some top-dimensional irreducible component $C$ of $p-\Isog\otimes\kappa$ which is $[b]$-dense. It turns out that such $[b]$ must be \emph{unramified}, a condition we define in Section \ref{section:unramifiedclasses}. Observe that if the restriction $H_{G,X}(\Frob)^{[b]}$ is known to be zero after restricting to the $[b]$-strata for every unramified $[b]$, then by taking the closure we see that $H_{G,X}(\Frob)$ is zero. 

We fix for now an unramified $[b]$. Our first observation is that a union of irreducible components $D\subset p-\Isog\otimes\kappa$ defines a cohomological correspondence from $R\Gamma_c(\sS_\kappa,\bbQ_l)$ to itself. Choosing a base-point, we have a map $\pi_\infty:\RZ(G,b,\upsilon)^{red}\rightarrow \sS_\kappa$, and we may pullback the cohomological correspondence on $\sS_\kappa$ to the Rapoport-Zink space $\RZ(G,b,\upsilon)$ via $\pi_\infty$. We denote this cohomological correspondence by $u_D$. Let $r$ be the dimension of $\RZ(G,b,\upsilon)^{red}$, and let $M:=M_{[b]}$.

The key theorem we have is the following:
\begin{customtheorem}{1.4}
Let $(G,X)$ be a Shimura datum of Hodge type. For any unramified $[b]\in B(G,\upsilon)$, let $C$ be a $[b]$-dense irreducible component of $p-\Isog\otimes\kappa$. There exists a map $\bar{h}$ (depending on $C$)
\begin{equation*}
    \bar{h}:\calH(M_{[b]}(\bbQ_p)//M_{[b]}(\bbZ_p),\bbQ)\rightarrow \bbQ[p-\Isog\otimes\kappa]
\end{equation*}
such that for any element $f\in\mathcal{H}(G(\mathbb{Q}_p)//K_p,\mathbb{Q})$, as cohomological correspondences acting on $H^{2r}_c(\RZ(G,b,\upsilon)^{red})$, we have
\begin{equation}
\label{eqn:1equal}
    u_{C\cdot h(f)}=u_{\bar{h}(\dot{\calS}^G_M(f))},
\end{equation}
where $\dot{\calS}^G_M$ is the twisted Satake homomorphism. Moreover, for any $f_1,f_2\in\mathcal{H}(G(\mathbb{Q}_p)//K_p,\mathbb{Q})$, we have 
\begin{equation*}
    u_{C\cdot h(f_1\cdot f_2)}=u_{\bar{h}(\dot{\calS}^G_M(f_1\cdot f_2))}.
\end{equation*}

If in addition we also assume that the Shimura datumn satisfies the conditions in Theorem 1.2, then we have an equality in $\bbQ[p-\Isog\otimes\kappa]$
\begin{equation}
\label{eqn:2equal}
    C\cdot h(f)=\bar{h}(\dot{\calS}^G_M(f)).
\end{equation}
\end{customtheorem}
To construct the map $\bar{h}$ in Theorem 1.4, we observe that the action of $\mathcal{H}(G(\mathbb{Q}_p)//K_p,\mathbb{Q})$ on the cohomology of $\RZ(G,b,\upsilon)$ extends to an action of $\mathcal{H}(M(\mathbb{Q}_p)//M(\bbZ_p),\mathbb{Q})$. Indeed, we have the following theorem.
\begin{customtheorem}{1.5}
In the Grothedieck group of representations of $\calH(G(\bbQ_p)//G(\bbZ_p),\bbQ)\times J_b(\bbQ_p)\times W_E$, we have the equality
\begin{equation}
\label{eqn:1.6}
    H^{\bullet}_c(\RZ(G,b,\upsilon)^{rig})=\sum_{\upsilon'\in I}H^{\bullet}_c(\RZ(M,b,\upsilon')^{rig}),
\end{equation}
where $H^\bullet_c$ denotes the alternating sum of cohomology,
\begin{equation*}
    I=\{\upsilon': \upsilon' \text{ is a dominant cocharacter of $M$ conjugate to }\upsilon \text{ in }G, \text{with }[b]\in B(M,\upsilon')\},
\end{equation*}
and on the right-hand side $\calH(G(\bbQ_p)//G(\bbZ_p),\bbQ)$ acts via the twisted Satake homomorphism $\dot{S}^G_M$.
\end{customtheorem}
Note that the compactly supported cohomology of the rigid analytic generic fiber $\RZ(G,b,\upsilon)^{rig}$ and the compactly supported cohomology of the special fiber $\RZ(G,b,\upsilon)^{red}$ are closely related; one is the dual and shift of the other.

For any $f_M\in \calH(M_{[b]}(\bbQ_p)//M_{[b]}(\bbZ_p),\bbQ)$, we construct the map in Theorem 1.4 by defining $\bar{h}(f_M)$ to be such that the cohomological correspondence $u_{\bar{h}(f_M)}$ is equal to $u_{C}\circ u_{f_M}$, where $u_{f_M}$ is the cohomological correspondence on $H^{2r}_c(\RZ(G,b,\upsilon)^{red})$, via the equality \eqref{eqn:1.6}.

Returning to Theorem 1.4, if the Shimura datumn is one of the cases in Theorem 1.2, we show in \eqref{section:nonzerob} that for any irreducible component $D\subset p-\Isog\otimes\kappa$, $u_D$ acts non-trivially on $H^{2r}_c(\RZ(G,b,\upsilon)^{red})$. For any irreducible component $[X_1]\in H^{2r}_c(\RZ(G,b,\upsilon)^{red})$,  let $\tilde{X_1}$ be the closed subvariety in $\sS_\kappa$ give by the image of $X_1$ under $\pi_\infty$. The cycle class map gives a cohomology class $\tilde{[X_1]}\in H_c^{2d-2r}(\sS_\kappa)$, where $d$ is the dimension of the Shimura variety. If we denote the cohomological correspondence associated to $D$ on $H_c^{2d-2r}(\sS_\kappa)$ by $\tilde{u}_D$, to show that $u_D([X_1])$ is non-zero, it suffices to show that $\tilde{u}_D(\tilde{[X_1]})$ is non-zero. $D$ determines an algebraic correspondence in $\sS_\kappa\times \sS_\kappa$, and acts on the Chow group $A_{r}(\sS_\kappa)$ of dimension $r$-cycles in $\sS_\kappa$. We show explicitly calculate this action of $D$ using intersection theory, and show that it is non-trivial. The key ingredient needed for this calculation is the (non-zero) intersection numbers of certain cycles. In case (1) we apply the results of Tian-Xiao \cite{TX19}. In the situation of case (2), the Hodge-Newton decomposable condition implies an isomorphism
\begin{equation*}
    \RZ(G,b,\upsilon)^{red}\simeq \RZ(M_{[b]},b,\upsilon)^{red},
\end{equation*}
and we hence can do this calculation on another Hodge-type Shimura variety $\Sh(M',X')$ whose basic locus gives rise to the Rapoport-Zink space $\RZ(M_{[b]},b,\upsilon)$. Here the intersection numbers are provided by the main theorem of Xiao-Zhu \cite{XZ17}.

Theorem 1.2 follows from Theorem 1.4 and an induction on the set of unramified elements in $B(G,\upsilon)$. Let $\mathrm{res}^{[b]}$ (resp. $\mathrm{res}^{\succeq[b]}$) be the maps on $\bbQ[p-\Isog\otimes\kappa]$ given by restricting to the irreducible components which are $[b]$-dense (resp. $[b']$-dense for $[b']\succeq [b]$). Observe that it suffices to show that for all unramified $[b]$ we have
\begin{equation}
\label{eqn:succeqb}
    \mathrm{res}^{\succeq[b]}(H_{G,X}(\Frob))=0.
\end{equation}
We will show this inductively, starting from the $\mu$-ordinary element, where the corresponding result is known \eqref{prop:muordinary}. If we know \eqref{eqn:succeqb} holds for all unramified $[b']\succ [b]$, then to show \eqref{eqn:succeqb} for $[b]$, it suffices to show that   
\begin{equation}
\label{eqn:b-congruence}
        \mathrm{res}^{[b]}(H_{G,X}(\Frob))=0.
\end{equation}
We will show \eqref{eqn:b-congruence} in two steps. Firstly, $\mathrm{res}^{[b]}(H_{G,X}(\Frob))$ can be factorized as 
\begin{equation}
\label{eqn:1.9}
    \mathrm{res}^{[b]}(H_{G,X}(\Frob))=H_{[b]}(\Frob)\cdot P(\Frob)
\end{equation}
for some polynomials $H_{[b]}(x),P(x)\in \bbQ[p-\Isog\otimes\kappa][x]$, and such that $P(\Frob)$ is $[b]$-dense. $H_{[b]}(x)$ is the image via $h$ of $\tilde{H}_{[b]}(x)$, a polynomial with coefficients in $\calH(G(\bbQ_p)//G(\bbZ_p),\bbQ)$. $\tilde{H}_{[b]}(x)$ has the property that via $\dot{S}^G_M$ it has a factor $H'_{[b]}(x)$ is of the form
\begin{equation}
\label{eqn:1.8}
    H'_{[b]}(x)=x^i-p^j1_{s\nu_{[b]}(p)M_c}
\end{equation}
for some positive integers $i,j$ and $s$ (all depending on $[b]$). We then apply the construction of $\bar{h}$ in Theorem 1.4, and observe that for any top-dimensional irreducible component $C\subset p-\Isog\otimes\kappa$ such that $C$ is $[b]$-dense, we have
\begin{equation}
\label{eqn:Cdot1}
    C\cdot \Frob^i=p^j\bar{h}(1_{s\nu_{[b]}(p)M_c}).
\end{equation}
To see this, we show that for a fixed $p$-divisible group $\sG/\bar{\bbF}_p$, the $i$-th power of Frobenius and the mod $p$ reduction of the quasi-isogeny given by $s\nu_{[b]}(p)$ lie in the same irreducible component of the associated affine Deligne-Lusztig variety (ADLV). We prove this by utilizing the explicit description of the set of irreducible components of ADLVs associated to unramified $[b]$, as shown by Xiao and Zhu (c.f. \cite[\S4]{XZ17}).

Finally, \eqref{eqn:Cdot1} shows that
\begin{equation*}
    C\cdot H_{[b]}(\Frob)=0,
\end{equation*}
since $C\cdot \Frob^i-p^j\bar{h}(1_{s\nu_{[b]}(p)M_c})$ will be a factor of the left-hand side by the compatibility with $\dot{\calS}^G_M$ described in Theorem 1.4. We observe that \eqref{eqn:1.9} is zero by applying the above result to all irreducible components $C$ which lie in the support of $P(\Frob)$. 

\subsection*{Structure of article}
We now describe the structure of this article. Section 2 contains all the necessary results about the Hecke polynomial, including the factors which we need to construct $H'_{[b]}(\Frob)$ and $H_{[b]}(\Frob)$. Section 3 contains the main technical results about $p$-divisible groups with $G$-structure, and we recall the construction of Rapoport-Zink spaces with $P$-structure for a parabolic subgroup $P$ of $G$. In Section 4, we recall results of \cite{HV2018} and \cite{XZ17} about the irreducible components of ADLVs, and reduction of isogenies mod $p$. We also show that the irreducible components of the ADLV are invariant under $\sigma^m$ for large enough $m$. In Section 5, we recall the results of \cite{K2010} on the construction of integral models of Hodge type Shimura varieties, as well as results about Newton strata. We also recall the results in \cite{K2017} about mod $p$ isogenies, and special point liftings (up to isogeny) of points in $\mathscr{S}_{K_p}(G,X)(\overline{\mathbb{F}}_p)$. Subsequently, in Section 6, we define the moduli space of $p$-quasi-isogenies $p-\Isog$, define the stratification of the special fiber $p-\Isog\otimes\kappa$ by $[b]\in B(G,\upsilon)$, and prove results about the dimensions of various strata, as well as an almost product structure on each strata. In Section 7, we describe how to pullback Hecke correspondences to the associated Rapoport-Zink spaces, and show a parabolic reduction result on the cohomology of Rapoport-Zink spaces (\ref{thm:pr1},\ref{thm:pr2}), and use these to prove Theorem 1.4 in this section. Finally, in Section 8, we complete the proof of Theorem \ref{mainthm}.

\subsection*{Acknowledgements}
I would like to thank my advisor, Mark Kisin, for suggesting this problem to me, and for his constant encouragement and advice. Many thanks also to Oliver B\"{u}ltel, Christophe Cornut, Pol van Hoften, Liang Xiao and Rong Zhou for helpful conversations about this project and comments on an earlier draft.

\section{The Hecke Polynomial $H_{G,X}(x)$}
\label{section:Heckepoly}
In this section, we recall results of Blasius and Rogawski \cite{BR1994} and Wedhorn \cite{W2000} about the construction and properties of the Hecke polynomial.
\subsection{Preliminaries and Definition}
Let $G$ be a connected algebraic group that is unramified over $\mathbb{Q}_p$, i.e. quasi-split over $\mathbb{Q}_p$ and split over an unramified extension. Equivalently, the base change $G_{\mathbb{Q}_p}$ has a reductive model $G_{\mathbb{Z}_p}$ over $\mathbb{Z}_p$. Let $K_p= G_{\mathbb{Z}_p}(\mathbb{Z}_p)$; this is a hyperspecial subgroup of $\mathbb{G}(\mathbb{Q}_p)$. 

Fix a pair $(T,B)$ where $T$ is a maximal torus of $G$ and $B$ is a Borel subgroup of $G$ containing $T$. Let $\Delta$ be the choice of simple roots of $G$ defined by $(T,B)$. $T$ contains a unique maximal subtorus which is split over $\mathbb{Q}_p$, which we denote by $S$. Let $\Omega$ denote the Weyl group $N_G(T)/T$, where $N_G(T)$ is the normalizer in $G$ of $T$. Let $\rho$ be the half-sum of positive roots of $G$.

Let $\hat{G}$ denote the dual of $G$, which is an algebraic group defined over $\mathbb{C}$. Let $\hat{T}$ denote the dual torus of $T$, and fix a Borel subgroup $\hat{B}$ of $\hat{G}$ containing $\hat{T}$. Moreover, we fix a splitting $(\hat{T}, \hat{B}, (X_{\hat{\alpha}})_{\hat{\alpha}})$ of $\hat{G}$, where for each simple root $\hat{\alpha}$ of $\hat{G}$, $X_{\hat{\alpha}}$ is an $\hat{\alpha}$-root vector in $\operatorname{Lie}(\hat{G})$. 

Let $\Gamma^{nr}=\text{Gal}(\mathbb{Q}_p^{ur}/\mathbb{Q}_p)$, where $\mathbb{Q}_p^{ur}$ is the maximal unramified extension of $\mathbb{Q}_p$. We let $\sigma$ denote the geometric Frobenius in $\Gamma^{nr}$. $\Gamma^{nr}$ acts on the root datum of $G$ (and hence the root datum of $\hat{G}$). This action fixes the splitting $(\hat{T}, \hat{B}, (X_{\hat{\alpha}})_{\hat{\alpha}})$, and hence induces an action of $\Gamma^{nr}$ on $\hat{G}$.

Define the $L$-group of $G$, denoted by ${^L}G$, as the semidirect product of $\Gamma^{nr}$ and $\hat{G}$, such that $\Gamma^{nr}$ acts on $\hat{G}$ via the action described above. Note that ${^L}G$ depends on the choice of splitting, but the $L$-groups obtained from different splittings are isomorphic.
\subsubsection{}
For any groups $H_1\subset H_2$, and any $\mathbb{Z}$-algebra $A$, let $\mathcal{H}(H_2//H_1,A)$ denote the Hecke algebra of locally constant $H_1$ bi-invariant $A$-valued functions. If $A=\bbQ$, for notational simplicity, we will simply write $\calH(H_2//H_1)$ for $\calH(H_2//H_1,\bbQ)$.

Recall that we have the Cartan decomposition
\begin{equation*}
G(\mathbb{Q}_p)=\coprod_{\lambda\in X_*(S)^+} K_p p^\lambda K_p,
\end{equation*}
where $X_*(S)^+$ consists of cocharacters of $S$ which are dominant with respect to $(T,B)$. Hence, $\mathcal{H}(G(\mathbb{Q}_p)//K,\mathbb{C})$ is generated by indicator functions of the form $1_{K_pgK_p}$, for some $g\in G(\mathbb{Q}_p)$.

\subsubsection{}

\label{subsection:hecke-poly-def}
Let $X$ be any $G(\mathbb{R})$-conjugacy class of morphisms $h:\mathbb{S}\rightarrow G_{\mathbb{C}}$ such that $(G,X)$ forms a Shimura datum, with reflex field $E$. In this section, we do not assume the Shimura datum is of Hodge type unless specifically mentioned. 

Any such morphism $h$ defines a cocharacter $\mu_h$, defined over complex points as follows:
\begin{align*}
  \mu_{h,\mathbb{C}}: \mathbb{C}^\times&\hookrightarrow\mathbb{C}^\times\times\mathbb{C}^\times\xrightarrow{h} G(\mathbb{C})\\
  z&\mapsto (z,1).
\end{align*} 

Note that the conjugacy class of $\mu_h$ is defined over $E_v$, where $v$ is a prime of $E$ above $p$, and in particular we can choose a representative cocharacter $\mu:\mathbb{G}_m\rightarrow G$ that is defined over $E_v$, and dominant with respect to the choice of $(T,B)$. Note that $E_v$ must be an unramified extension of $\mathbb{Q}_p$, since $G$ is unramified over $\mathbb{Q}_p$. Moreover, $\mu$ must be minuscule. 

Let $\lambda$ be the unique Weyl conjugate of $\mu^{-1}$ which is dominant with respect to $(T,B)$. There exists a unique representation 
\begin{equation*}
  r: {^L} G\rightarrow GL(V)
\end{equation*} such that 
\begin{enumerate}[(a)]
	\item The restriction of $r$ to $\hat{G}$ is irreducible with highest weight $\hat{\lambda}$
	\item The subgroup $\Gamma^{nr}$ acts trivially on the highest weight space.
\end{enumerate}

Define the polynomial
\begin{equation*}
	H_{G,X}(x)=\det(x-p^{dn}r(\sigma\ltimes \hat{g})^n),
\end{equation*}
where  $\hat{g}$ is any element in $\hat{G}(\mathbb{C})$, $d=\langle \rho,\lambda\rangle$, $\sigma\in\Gamma^{nr}$ is the geometric Frobenius and $[E_v:\mathbb{Q}_p]=n$.

\begin{remark}
  Note that this definition is different from the one given in \cite{BR1994}, where the Hecke polynomial is defined to be the characteristic polynomial of the highest weight representation associated to $\mu$, instead of $\mu^{-1}$. We believe that this is the correct polynomial to take, because it is consistent with the reciprocity law on points. See also the remarks in \cite[A4]{Nek2018} correcting \cite[Sec. 2]{W2000}.
\end{remark}

\subsection{Coefficients of $H_{G,X}(x)$}\label{subsection:hecke-poly-coeff} 
Write 
\begin{equation*}
  H_{G,X}(x)=\sum A_ix^i.
\end{equation*}
As defined, all the $A_i$'s are functions on $\hat{G}(\mathbb{C})$. Observe that each $A_i$ is invariant under $\sigma$-conjugacy because $\sigma$-conjugation on $\hat{G}$ is normal conjugation by $\hat{G}(\mathbb{C})$ in $^LG$, hence $\sigma$-conjugating $\hat{g}$ by any element of $\hat{G}(\mathbb{C})$ preserves the determinant. Let $\Phi_{nr}(G)$ be the set of semisimple $\sigma$-conjugacy classes of elements of $\hat{G}(\mathbb{C})$. These $A_i$'s then descend to functions on $\Phi_{nr}(G)$, which we also denote by $A_i$.

Let $\hat{S}$ denote the subtorus of $\hat{T}$ which is dual to $S$. From \cite[6.4]{borel1979} we see that we have a bijection  between $\hat{S}(\mathbb{C})/\Omega(\mathbb{Q}_p)$ and $\Phi_{nr}(G)$, hence we can identify functions on $\Phi_{nr}(G)$ with functions on $\hat{S}(\mathbb{C})/\Omega(\mathbb{Q}_p)$. Such functions are given by elements of $\mathbb{C}[X_*(S)]^{\Omega(\mathbb{Q}_p)}$. Moreover, we also have an isomorphism between $\mathbb{C}[X_*(S)]$ and the Hecke algebra $\mathcal{H}(T(\mathbb{Q}_p)//T_c,\mathbb{C})$ where $T_c=T(\mathbb{Q}_p)\cap K_p$ is a maximal compact subgroup of $T(\mathbb{Q}_p)$. This isomorphism is given via the map 
\begin{equation*}
  \nu\mapsto h_{\nu}:=1_{p^\nu T_c},
\end{equation*}
where $1_{p^\nu T_c}$ is the indicator function supported on the subset $p^\nu T_c$.

Hence, we see that $A_i$'s are elements of $\mathcal{H}(T(\mathbb{Q}_p)//T_c,\mathbb{C})$ that are fixed under the action of $\Omega(\mathbb{Q}_p)$. By the Satake isomorphism, we see that $A_i$'s are elements of $\mathcal{H}(G(\mathbb{Q}_p)//K_p,\mathbb{C})$. 

In fact, the following result of Wedhorn (c.f. \cite[2.4]{W2000}) shows that the $A_i$'s are $\mathbb{Q}$-linear combinations of the functions $h_\lambda$.
	
\begin{theorem}
	$H_{G,X}(x)$ is a polynomial with coefficients in $\mathcal{H}(G(\mathbb{Q}_p)//K_p,\mathbb{Q})$.
\end{theorem}

\subsubsection{}
\label{subsub:twisted satake}Let $P = MN$ be any standard parabolic subgroup of $G$ with respect to $(T,B)$, where $M$ is a standard Levi subgroup, and $N$ is the unipotent radical of $P$. Let $M_c=M(\mathbb{Q}_p)\cap K_p$; this is a hyperspecial subgroup of $M(\mathbb{Q}_p)$. We can define a map
\begin{equation*}
    \dot{S}^G_M:\mathcal{H}(G(\mathbb{Q}_p)//K_p,\mathbb{C}) \rightarrow \mathcal{H}(M(\mathbb{Q}_p)//M_c,\mathbb{C})
\end{equation*} by mapping $f\in \mathcal{H}(G(\mathbb{Q}_p)//K_p,\mathbb{C})$ to the function on $M(\mathbb{Q}_p)$ taking
\begin{equation*}
  m\mapsto\int_{N(\mathbb{Q}_p)} f(mn) dn.
\end{equation*}
We call this map the \emph{twisted Satake homomorphism}. This is a homomorphism of $\mathbb{C}$-algebras because the Iwasawa decomposition $G = PK_p$ induces a decomposition of measures $dg = dp dk$. Moreover, since the image of a $\bbQ$-valued function is also clearly $\bbQ$-valued, we see that $\dot{S}^G_M$ maps $\mathcal{H}(G(\mathbb{Q}_p)//K_p,\mathbb{Q})$ to $\mathcal{H}(M(\mathbb{Q}_p)//M_c,\mathbb{Q})$.

We can relate the classical Satake isomorphism with the map $\dot{S}^G_T$. In fact, let $S^G_T$ denote the usual Satake isomorphism. Then $S^G_T=\alpha \circ \dot{S}^G_T$, where $\alpha : \mathbb{C}[X_*(S)] \rightarrow \mathbb{C}[X_*(S)]$ is given by 
\begin{equation*}
  \alpha(\nu)=p^{-2\langle \rho,\nu\rangle}\nu
\end{equation*} for all $\nu$ in $X_*(S)$.

Similar to the usual Satake isomorphism, there is an isomorphism between $\mathcal{H}(G(\mathbb{Q}_p)//K_p,\mathbb{C})$ and the subalgebra $\mathcal{H}(T(\mathbb{Q}_p)//T_c,\mathbb{C})^{(\Omega(\bbQ_p),\bullet)}$ induced by $\dot{S}^G_T$. Here, the Weyl group acts via the ``dot action", instead of the usual action, give by 
\begin{equation*}
  (w\bullet\phi)(t)=\delta(t)^{1/2}\delta(w^{-1}t)^{-1/2}\phi(w^{-1}t),
\end{equation*}
where $\delta^{1/2}$ is given by $\delta^{1/2}=\abs{\rho}$. In particular, we see that $\dot{S}^G_M$ is injective, since $\dot{S}^G_T=\dot{S}^M_T\circ\dot{S}^G_M$ is injective.

\subsection{Factors of $H_{G,X}(x)$} We want to understand the factors of the Hecke polynomial. Let
\begin{equation*}
\tilde{\lambda}:=\mathrm{Nm}_{E_v/\mathbb{Q}_p}\lambda=\lambda\sigma(\lambda)\dots\sigma^{n-1}(\lambda).
\end{equation*}

This is a cocharacter defined over $\mathbb{Q}_p$. Let $M_{\tilde{\lambda}}$ denote the centralizer of $\tilde{\lambda}$, then $M_{\tilde{\lambda}}$ is a standard Levi subgroup of $G$, and the results of the previous section apply. Via the homomorphism $\dot{S}^G_{M_{\tilde{\lambda}}}$, we have the following proposition, which is a restatement of a result of B\"{u}ltel (c.f. \cite[3.4]{B2002})
\begin{proposition}
\label{prop:bultelresult}
  Viewed as a polynomial with coefficients in $\mathcal{H}(M_{\tilde{\lambda}}(\mathbb{Q}_p)//M_{\tilde{\lambda},c},\mathbb{Q})$ via $\dot{S}^G_{M_{\tilde{\lambda}}}$, $H_{G,X}(x)$ has a factor of the form 
  \begin{equation}
    \label{eqn:bultelresult}
    x-1_{p^{\tilde{\lambda}}M_{\tilde{\lambda},c}}.
  \end{equation}
\end{proposition}

This result can be generalized, and we want to determine other factors of $H_{G,X}(x)$. Recall that from the discussion in Section \ref{subsection:hecke-poly-coeff} we can also view $H_{G,X}(x)$ as a polynomial with coefficients in $\mathbb{Q}[X_*(S)]$. In order to determine these coefficients, we need to understand how the representation $r$ defined in \eqref{subsection:hecke-poly-def} acts on $\hat{T}$. 

Consider ${^L}T= \Gamma^{nr}\ltimes\hat{T}$ as a subgroup of ${^L}G=\Gamma^{nr}\ltimes\hat{G}$. The restriction of the representation $r$ to $\hat{T}$ defines a grading $V=\oplus_{\lambda\in X_*(\hat{T})}V_\lambda$. We define $r_T:{^L}T\rightarrow GL(V)$, the twisted restriction of $r$, as
\begin{equation}
  r_T(\sigma\ltimes 1)=r(\sigma\ltimes 1)
\end{equation}
\begin{equation}
    r_T(1\ltimes \hat{t})x_\nu= p^{-\langle\rho,\nu\rangle}\nu(\hat{t})x_\nu
    \label{eq:r_T description2}
\end{equation}
for $x_\nu\in V_\nu,\hat{t}\in\hat{T}$. Consider the characteristic polynomial
\begin{equation*}
  Z_{r_T}(x)=\det(x-p^{dn}r_T(\sigma\ltimes \hat{t})^n).
\end{equation*}
Since $\sigma$-conjugation of $\hat{t}$ preserves $Z_{r_T}(x)$, we see that a consideration similar to Section \ref{subsection:hecke-poly-coeff} implies that $Z_{r_T}(x)$ has coefficients lying in $\mathcal{H}(T(\mathbb{Q}_p)//T_c,\mathbb{C})$. Moreover, we also have
\begin{equation*}
    \dot{S}^G_T(H_{G,X}(x))=Z_{r_T}(x).
\end{equation*}

The following lemma \cite[Prop 2.7]{W2000} gives an explicit description of $r_T$.
\begin{lemma}
    The twisted restriction $r_T$ of $r$ to ${^L}T$ is isomorphic to a direct sum
    \begin{equation*}
        \bigoplus_{\nu\in \Omega(\overline{\bbQ}_p)\mu^{-1}}V_\nu
    \end{equation*}
    where $V_\nu$ is one-dimensional with generators $e_\nu$ such that 
    \begin{equation}
    \label{eqn:r_T description}
        r(\sigma^n\ltimes \hat{t})e_{\sigma^{-n}(\nu)}=p^{-n\langle \rho,\nu\rangle}\hat{\nu}(\hat{t})e_\nu.
    \end{equation}
\end{lemma}
\begin{proof}
  Note that since $r$ is the highest weight representation associated to a minuscule character, every weight space $V_\nu$ has dimension one. Let $Z$ be an $\langle\sigma^n\rangle$-orbit of $\Omega(\bar{\bbQ}_p)\mu^{-1}$. Let $m$ be the number of elements of $Z$, and $\nu$ be any element of $Z$. $\nu$ is hence defined over an unramified degree $m$ extension of $E_v$. Choose a nonzero element $e_\nu\in V_\nu$ and define $e_{\sigma^{nm}\nu}=r(\sigma^{nm}\ltimes 1)(e_\nu)$ for $n=0,\dots m-1$. The relation \eqref{eqn:r_T description} then follows from \eqref{eq:r_T description2}.
\end{proof}
\subsubsection{}
We also observe that $V$ admits a direct sum decomposition as
\begin{equation*}
  V=\bigoplus_{\langle \sigma^n \rangle-\text{ orbit } Z_i}V_{Z_i}
\end{equation*}
where each $V_{Z_i}$ is invariant under the action of $r(\sigma^n\ltimes \hat{t})$, for any $\hat{t}\in\hat{T}$. Let $\nu_i$ be an element of this orbit. With respect to the basis $e_{\nu_i}$, $ e_{\sigma^n(\nu_i)}$, $\dots$, $e_{\sigma^{n(m_i-1)}(\nu_i)}$ of $V_{Z_i}$, $r_T(\sigma^n\ltimes\hat{t})$ is given by the matrix
\begin{equation*}
  p^{-n\langle\rho,\nu_i\rangle}\begin{pmatrix}
   & & & \sigma^{n(m_i-1)}(\hat{\nu_i})(\hat{t})\\
  \sigma^{n}(\hat{\nu_i})(\hat{t}) &  &  &\\
  &\sigma^{2n}(\hat{\nu_i})(\hat{t})& &\\
  & &\ddots &
\end{pmatrix}.
\end{equation*}
The characteristic polynomial of this matrix is given by 
\begin{equation*}
    H_i(x):=x^{m_i}-p^{nm_i\langle\rho,\lambda-\nu_i\rangle}\tilde{\nu_i},
\end{equation*}
where $\tilde{\nu_i}=\prod_{j=1}^{nm_i}\sigma^j(\nu_i)$ is a cocharacter defined over $\mathbb{Q}_p$. This is a factor of the Hecke polynomial $H_{G,X}(x)$, viewed as a polynomial with coefficients in $\mathbb{Q}[X_*(S)]$. 
Let $M_{\tilde{\nu_i}}$ be the Levi subgroup of $G$ which is the centralizer of $\tilde{\nu_i}$. Via the twisted Satake homomorphism $\dot{S}^{M_{\tilde{\nu_i}}}_G$, $H_{G,X}(x)$ can be viewed as a polynomial with coefficients in $\mathcal{H}(M_{\tilde{\nu_i}}(\mathbb{Q}_p),M_{\tilde{\nu_i},c})$. Observe that since $\tilde{\nu_i}$ is central in $M_{\tilde{\nu_i}}$, $1_{\tilde{\nu_i}(p)M_{\tilde{\nu_i},c}}$ is an element of $\mathcal{H}(M_{\tilde{\nu_i}}(\mathbb{Q}_p),M_{\tilde{\nu_i},c})$. Hence, $H_i(x)$ is invariant under the action of the relative Weyl group of $M_{\tilde{\nu_i}}$, and thus it lies in $\mathcal{H}(M_{\tilde{\nu_i}}(\mathbb{Q}_p),M_{\tilde{\nu_i},c})[x]$. If we let $H_{G,X}(x)=H_i(x)R(x)$ as polynomials in $\mathbb{Q}(X_*(S))[x]$, then we see that both factors $R(x),H_i(x)$ are defined in $\mathcal{H}(M_{\tilde{\nu_i}}(\mathbb{Q}_p),M_{\tilde{\nu_i},c})[x]$. Finally, we note that $\langle\rho,\lambda\rangle=\langle\rho,\mu\rangle$, so we can replace $\lambda$ with $\mu$ in the formula. We summarize the above discussion in the following proposition.
\begin{proposition}
\label{prop:factorHecke}
  Let $\nu_i$ be any element in $\Omega(\bar{\mathbb{Q}}_p)\mu^{-1}$, defined over a degree $m_i$ extension of $E_v$. If we define $\tilde{\nu_i}=\prod_{j=1}^{nm}\sigma^j(\nu_i)$, and let $M_{\tilde{\nu_i}}$ be the centralizer of $\tilde{\nu_i}$, then, viewed as an element of $\mathcal{H}(M_{\tilde{\nu_i}}(\mathbb{Q}_p),M_{\tilde{\nu_i},c})[x]$ via the twisted Satake homomorphism $\dot{S}^{M_{\tilde{\nu_i}}}_G$, $H_{G,X}(x)$ factors as 
  \begin{equation*}
    H_{G,X}(x)=H_i(x)R(x),
  \end{equation*}
  where 
  \begin{equation*}
    H_i(x)=x^{m_i}-p^{nm_i\langle\rho,\mu-\nu_i\rangle}\tilde{\nu_i}.
  \end{equation*}
\end{proposition}
\subsubsection{}
\label{prop:Hbwithweylconjugate}
While the factors $H_i(x)$ are not necessarily defined in $\calH(G(\bbQ_p)//K_p)[x]$, for each $\langle\sigma^n\rangle$-orbit $Z_i$, we can consider the (not-necessarily distinct) $\langle\sigma^n\rangle$-orbits of the elements $\omega(\nu_i)$, where $\omega\in \Omega(\bbQ_p)$. Let $Y_i$ be the $\langle\sigma^n\rangle$-orbits of $\omega(\nu_i)$ for some $\omega$. Observe that $\omega(\nu_i)$ is also defined over a degree $m_i$ extension of $E_v$, and hence from the proposition above we see that $H_{G,X}(x)$ has another factor of the form 
\begin{equation*}
    P'_i(x)=x^{m_i}-p^{nm_i\langle\rho,\mu-\omega(\nu_i)\rangle}\omega(\tilde{\nu_i}).
  \end{equation*}
Thus, we see that if we consider
\begin{equation*}
    H'_i(x):=\prod P'_i(x)
\end{equation*}
where the product is taken over all possible $\langle\sigma^n\rangle$-orbits of the elements $\omega(\nu_i)$ (without multiplicity), we see that $H'_i(x)$ is a factor of $H_{G,X}(x)$ which is defined in $\calH(G(\bbQ_p)//K_p)[x]$.

\begin{remark}
  The set of $\langle\sigma^n\rangle$-orbits of the elements $\omega(\nu_i)$ is closely related to the $\sigma$-conjugacy classes of unramified elements, a connection we elaborate upon in Section \ref{section:constructionHb}.
\end{remark}

\section{$p$-divisible groups}
For the entirety of this section, we fix a prime $p>2$ and an unramified reductive group $G$ over $\bbQ_p$. We fix also a Borel pair $T\subset B\subset G$. 
\subsection{$p$-divisible groups with $G$-structure}
We recall the work of Kisin \cite{K2010,K2017} on $p$-divisible groups with a collection of tensors in the Dieudonne module, both over fields of characteristic $p$ and over mixed characteristic discrete valuation rings.
\subsubsection{}
Let $k$ be a perfect field of characteristic $p$, $W=W(k)$ its ring of Witt vectors and $K_0=W(k)[1/p]$. Let $K$ be a finite totally ramified extension of $K_0$, and $O_K$ its ring of integers. Let the absolute Galois group $\operatorname{Gal}(\bar{K}/K)$ be $G_K$.

Let $\mathrm{Rep}^{\mathrm{cris}}_{G_K}$ be the the category of crystalline $G_K$-representations, and $\mathrm{Rep}^{\mathrm{cris}\circ}_{G_K}$ the category of $G_K$-stable $\mathbb{Z}_p$-lattices spanning a representation in $\mathrm{Rep}^{\mathrm{cris}}_{G_K}$. Given $V\in \mathrm{Rep}^{\mathrm{cris}}_{G_K}$, we let $D_{\mathrm{dR}}(V)$ and $D_{\mathrm{cris}}(V)$ denote the image of $V$ under Fontaine's de Rham and crystalline functors.

Let $E(u)\in O_{K_0}[u]$ be the Eisenstein polynomial of a fixed uniformiser $\pi$ in $K$. We set $\mathfrak{S}=W[[u]]$. We extend the Frobenius on $W$ to the Frobenius map $\phi$ on $\frakS$ which maps $u$ to $u^p$.

Let $\text{Mod}^\varphi_{/\mathfrak{S}}$ denote the category of finite free $\mathfrak{S}$-modules $\mathfrak{M}$ equipped with a Frobenius semi-linear isomorphism
\begin{equation*}
    1\otimes\varphi:\varphi^*(\mathfrak{M})[1/E(u)]\xrightarrow{\sim}\mathfrak{M}[1/E(u)],
\end{equation*}
and let $\text{BT}^\varphi_{/\mathfrak{S}}$ denote the full sub-category of $\text{Mod}^\varphi_{/\mathfrak{S}}$ consisting of those modules such that $1\otimes\varphi$ maps $\varphi^*(\mathfrak{M})$ into $\mathfrak{M}$, and the image of this map contains $E(u)\mathfrak{M}$.

We equip $\varphi^*(\mathfrak{M})$ with the following filtration: for $i\in\mathbb{Z}$ 
\begin{equation*}
\operatorname{Fil}^i\varphi^*(\mathfrak{M}) = (1\otimes\varphi)^{-1}(E(u)^i\mathfrak{M})\cap \varphi^*(\mathfrak{M}).
\end{equation*}

The following theorem is \cite[Thm 1.2.1]{K2010}.
\begin{theorem}
\label{thm:frakM}
  There exists a fully faithful tensor functor
  \begin{equation*}
      \mathfrak{M}(\cdot) : \mathrm{Rep}^{\mathrm{cris}\circ}_{G_K}\rightarrow \mathrm{Mod}^{\varphi}_{/\mathfrak{S}},
  \end{equation*}
  which is compatible with formation of symmetric and exterior powers. If $L$ is in $\mathrm{Rep}^{\mathrm{cris}\circ}_{G_K}$, let $V=L\otimes _{\mathbb{Z}_p}\mathbb{Q}_p$, and $\mathfrak{M}=\mathfrak{M}(L)$, then
  \begin{enumerate}
      \item There are canonical isomorphisms 
      \begin{equation*}
          D_{\mathrm{cris}}(V)\xlongrightarrow{\sim}\mathfrak{M}/u\mathfrak{M}[1/p]\qquad \mathrm{and}\qquad D_{\mathrm{dR}}(V)\xlongrightarrow{\sim}\varphi^*(\mathfrak{M})\otimes_\mathfrak{S}K,
      \end{equation*}
      where the map $\mathfrak{S}\rightarrow K$ is given by $u\mapsto \pi$. The first isomorphism is compatible with Frobenius and the second maps $\mathrm{Fil}^i\varphi^*(\mathfrak{M})\otimes_{\frakS} W_{K_0}$ onto $\mathrm{Fil}^i D_{\mathrm{dR}}(V)$ for $i\in\mathbb{Z}$.
      \item There is a canonical isomorphism
      \begin{equation*}
          O_{\widehat{\calE_{ur}}}\otimes_{\mathbb{Z}_p} L\xlongrightarrow{\sim}O_{\widehat{\calE_{ur}}}\otimes_{\mathfrak{S}} \mathfrak{M}
      \end{equation*}
      where $O_{\widehat{\calE_{ur}}}$ is a certain faithfully flat, and formally  \'{e}tale $O_{\calE}$-algebra, and $O_{\calE}$ is the $p$-adic completion of $\mathfrak{S}_{(p)}$.
  \end{enumerate}
\end{theorem}
\subsubsection{} For a ring $R$ and a finite free $R$-module $M$, let $M^\otimes$ be the direct sum of all the $R$-modules obtained from $M$ by taking duals, tensor products, symmetric and exterior powers.

Let $L\in \mathrm{Rep}^{\mathrm{cris}\circ}_{G_K}$. Let $G\subset GL(L)$ be a reductive group defined by a finite collection of $G_K$-invariant tensors $(s_\alpha)\subset L^\otimes$. We may view the tensors $s_\alpha$ as morphisms $s_\alpha:1\rightarrow L^\otimes$ in $\mathrm{Rep}^\mathrm{cris\circ}G_K$. Applying the functor $\frakM$ of the theorem, we obtain morphisms $\tilde{s}_\alpha:1\rightarrow \frakM(L)^\otimes$ in $\mathrm{Mod}^\varphi_{/\frakS}$.

We have the following results \cite[1.3.4,1.3.6]{K2010}.

\begin{theorem}
\label{thm:isomorphismfrakM}
  We have the following:
  \begin{enumerate}
      \item The tensors $(\tilde{s}_\alpha)$ define a reductive subgroup in $GL(\frakM)$
      \item If $k$ is separably closed or $G$ is connected and $k$ is finite, then there exists an isomorphism 
      \begin{equation*}
          L\otimes_{\bbZ_p}\frakS\xlongrightarrow{\sim}\frakM
      \end{equation*}
      aking $s_\alpha$ to $\tilde{s}_\alpha$.
  \end{enumerate}
\end{theorem}

We also recall the following proposition \cite[1.1.7]{K2017} which characterises reductions mod $p$ of a $p$-divisible group with $G$ structure over $O_K$.
\begin{proposition}
\label{prop:Gred}
  Let $\mathscr{G}$ be a $p$-divisible group over $O_K$, and $(s_\alpha)\subset T_p\mathscr{G}^\otimes$ a collection of $G_K$-invariant tensors defining a reductive subgroup $G\subset GL(T_p\mathscr{G})$. Suppose that either $k$ is separably closed or that $G_{\mathbb{Z}_p}$ is connected and $k$ is finite. Let $(s_{\alpha,0})\subset\mathbb{D}(\mathscr{G})(W)^\otimes \otimes_{W} K_0$ denote the image of $(s_\alpha)$ under the $p$-adic comparison isomorphism. Then
  \begin{enumerate}
    \item $(s_{\alpha,0})\subset \mathbb{D}(\mathscr{G})(W)^\otimes$.
    \item There is an isomorphism $T_p\mathscr{G}^*(-1)\otimes_{\mathbb{Z}_p} W\xlongrightarrow{\sim}\mathbb{D}(\mathscr{G})$ taking $s_\alpha$ to $s_{\alpha,0}$.
    \item The filtration on $\mathbb{D}(\mathscr{G})(k)$ is induced by a $G_W$-valued cocharacter.
  \end{enumerate}
\end{proposition}
\subsubsection{}
Let $\sG$ be a $p$-divisible group over $k$. We write $\bbD(\sG)$ for $\bbD(\sG)(W)$. 
\begin{definition}
  A $p$-divisible group with $G$ structure over $k$ consists of a $p$-divisible group $\sG/k$ and a collection of $\varphi$-invariant tensors $(s_{\alpha,0})$ which define a reductive subgroup of $GL(\bbD(\sG))$ such that there exists a finite free $\mathbb{Z}_p$-module $U$ and an isomorphism
  \begin{equation}
    U\otimes_{\mathbb{Z}_p}W\xlongrightarrow{\sim}\bbD(\sG) \label{eqn:Uisomorphism}
\end{equation}
  such that under this isomorphism $(s_{\alpha,0})$ correspond to tensors $(s_\alpha)\subset U^\otimes$. Moreover, these $s_{\alpha}$ define the reductive subgroup $G_{\bbZ_p}\subset GL(U)$.
\end{definition}
Given any $p$-divisible group over $k$ with $G$-structure, since $(s_{\alpha,0})\subset \bbD(\sG)^\otimes$ are $\varphi$-invariant, via the above isomorphism, $\varphi$ on $\bbD(\sG)(W)\otimes_W W[1/p]$ has the form $b\sigma$ for some $b\in G(W[1/p])$. The following lemma is \cite[1.1.12]{K2017}.
\begin{lemma}
    \label{lemma:filtration}
    The filtration on $\bbD(\sG)(k)$ is given by a $G_W$-valued cocharacter $\mu^{-1}_0$, and $b\in G_W(W)p^{\upsilon_0} G_W(W)$ where $\upsilon_0=\sigma(\mu^{-1}_0)$.
\end{lemma}
\subsubsection{}Let $S$ be the $p$-adic completion of the divided power envelope of $W(k)[u]$ with respect to the kernel of the map $W(k)[u]\xrightarrow{u\mapsto\pi} O_K$. We equip $S$ with a Frobenius given by the usual Frobenius on $W(k)$ and sending $u$ to $u^p$. We view $S$ as a $\mathfrak{S}$-algebra by the map sending $u$ to $u$, and we view $W$ as an $S$-algebra by the map sending $u$ to 0. The following definition appears in \cite[1.1.8]{K2017}.
\begin{definition}
  Let $\tilde{\mathscr{G}}$ be a deformation of $\mathscr{G}$ to $O_K$. We say that $\tilde{
  \mathscr{G}}$ is $G_{W}$-adapted if there exist Frobenius invariant tensors $(\tilde{s}_\alpha)\subset\mathbb{D}(\tilde{\mathscr{G}})(S)^\otimes$ lifting $s_{\alpha,0}$, which define a reductive subgroup $G_S$ and such that the images of the $\tilde{s}_\alpha$ in $\mathbb{D}(\mathscr{G})(O_K)^\otimes$ are in $\text{Fil}^0(\mathbb{D}(\mathscr{G})(O_K)^\otimes)$
\end{definition}

Let $\mu^{-1}_0$ be the $G_W$-valued cocharacter inducing the filtration on $\bbD(\sG)(k)$, as in Lemma \ref{lemma:filtration}. Let $U$ be the unipotent subgroup of $GL(U)$ associated to $\mu_0$, and $U_G$ be the unipotent subgroup of $G$ associated to $\mu_0$. Let $R$ (respectively $R_G$) be the completion of $U$ (respectively $U_G$) at the identity; we thus have an inclusion $R_G\subset R$ of  power series rings over $W$. Following the results of \cite[\S7]{Fal99}, we know that the deformation space of $\sG$ is given by $\operatorname{Spf}R$. We have the following result of \cite[1.5.8]{K2010}
\begin{proposition}
\label{prop:3RG}
  For any field $K$ as above, a map of $W$-algebras $\varpi:R\rightarrow O_K$ factors through $R_G$ if and only if the $p$-divisible group $\tilde{\sG}$ induced by $\varpi$ is $G_W$-adapted.
\end{proposition}
\subsubsection{}
\label{section:Gadaptedlift}
We now show that every $p$-divisible group $\sG$ over $k$ with $G$-structure admits a $G_W$-adapted lifting. Let $\mathfrak{M}=\sigma^{-1 *}(\mathbb{D}(\mathscr{G})(W))$, so that $\varphi^*(\mathfrak{M})\xrightarrow{\sim}\mathbb{D}(\mathscr{G})(W)\otimes_W\mathfrak{S}$. 

Let $\mu_0$ be a $G_W$-valued cocharacter such that $\mu^{-1}_0$ induces the filtration on $\bbD(\sG)(k)$. Since the $s_{\alpha,0}$ are in $U^\otimes$, we may think of these as tensors in $\mathfrak{M}^\otimes$ and $\varphi^*(\mathfrak{M})^\otimes$. Consider now a cocharacter $\mu$, lifting $\mu_0$, valued in $G_\mathfrak{S}\subset GL(\varphi^*(\mathfrak{M}))$, and consider the map
\begin{equation*}
  \varphi^*(\mathfrak{M})\xrightarrow{c\cdot \mu^{-1}(E(u))}\varphi^*(\mathfrak{M})\xrightarrow{\sim} \mathfrak{M},
\end{equation*}
where $c=\sigma^{-1}(b)\mu_0(E(0))$ and the final map is induced by the identity on $U\subset \varphi^*(\mathfrak{M})$. Then
\begin{equation*}
  \sigma(c) =b\upsilon_0(E(0))^{-1}= (b\upsilon_0(p)^{-1})(\upsilon_0(p/E(0)))\in GL(\mathbb{D}(\mathscr{G})(W)).
\end{equation*}
Hence $c\in GL(\mathbb{D}(\mathscr{G})(W))$ and the map above gives $\mathfrak{M}$ the structure of an object of $\mathrm{BT}^\varphi_{/\frakS}$, and hence corresponds to a deformation $\tilde{\mathscr{G}}$ of $\mathscr{G}$. The filtration on $\mathcal{M}(\mathfrak{M}) =S\otimes_{\frakS}\varphi^*(\mathfrak{M})$ is induced by the cocharacter $\mu^{-1}$. Hence thinking of the $s_{\alpha,0}$ in $\mathcal{M}(\mathfrak{M}) =\mathbb{D}(\tilde{\mathscr{G}}(S)$, shows that $\tilde{\mathscr{G}}$ is a $G_W$-adapted lifting.

\subsubsection{}
\label{subsubsection:farguesfiltrationextracondition}
We want to understand when the $G$-adapted lift of a $p$-divisible group $\sG$ over $k=\bar{\bbF}_p$ also has a lift of the slope filtration.

Recall that to any $p$-divisible group $\sG$ over $\bar{\bbF}_p$ with $G$-structure we have attached an element $b\in G(K_0)$, defined up to $\sigma$-conjugacy by $G(O_{K_0})$, such that the Frobenius on $\bbD(\sG)(O_{K_0})$ is given by $b\sigma$. Let $\nu_b$ be the cocharacter inducing the slope graduation on the isocrystal $D:=\bbD(\sG)(O_{K_0})[1/p]$.

Following the main result of \cite{zink2001slope}, we know that $\sG$ admits a slope filtration, so let
\begin{equation}
    0=\sG_0\subset \sG_1\subset\dots\subset \sG_n=\sG
    \label{eqn:slopefiltration}
\end{equation}
denote the slope filtration. This induces a filtration on the Dieudonn\'{e} modules $M_i:=\bbD(\sG_i)(O_L)$, so we have
\begin{equation*}
    0=M_0\subset M_1\subset\dots\subset M_n=\bbD(\sG)(O_L).
\end{equation*}

We will now show that $\sG$ admits a $G$-adapted lift $\tilde{\sG}$ over $O_L$ which admits a lift of the slope filtration. Note that we can always choose a representative $b\in M_{[b]}$, and in this case the associated $p$-divisible group $\sG$ has $M_{[b]}$ structure. We would like to generalize this slightly.

\begin{proposition}
There exists a $p$-divisible group $\tilde{\sG}$ over $O_L$, lifting $\sG$, and which admits a filtration
\begin{equation*}
    0=\tilde{\sG}_0\subset \tilde{\sG}_1\subset\dots\subset \tilde{\sG}_n=\tilde{\sG}
\end{equation*}
which lifts \eqref{eqn:slopefiltration}.
\end{proposition}
\begin{proof}
Let $b':=g^{-1}b\sigma(g)$.

We first observe that by the Iwasawa decomposition, we may assume that $g\in P(L)$, and hence that the parabolic subgroup associated to the Newton cocharacter $\nu_{b'}=g\nu_b g^{-1}$ is equal to $P$.

By Grothendieck-Messing, it suffices to exhibit a $P_W$-valued cocharacter $\mu'$ lifting the Hodge filtration on $\bbD(\sG_g)(k)$. To see this, first consider the parabolic subgroup $P^H_k\subset G_k$ associated to the Hodge filtration on $\bbD(\sG_g)(k)$; with the set up as above it is the parabolic subgroup associated to $\mu^{-1}\otimes k$. Let $P^N_k\subset G_k$ be the special fiber of the parabolic $P$ associated to the slope filtration on $\sG$, since $P$ is defined over $W$.

By the Bruhat decomposition, there is some maximal torus $T'$, defined over $k$, such that $T'\subset P^N_k\cap P^H_k$. We can lift $T'$ to a torus $T'_W$ over $W$ contained in $P_W$. Then the Hodge filtration on $\bbD(\sG_g)(k)$ is given by some $T'$-valued cocharacter ${\mu'}^{-1}$, and lifting this to a $T'_W$-valued cocharacter ${\mu'_0}^{-1}$ gives us the desired cocharacter.
\end{proof}

In particular, we have a $P_W$-valued cocharacter $\mu'_0$, whose inverse lifts the Hodge filtration, so we can recall the previous results about the construction of the universal deformation ring $\operatorname{Spf}R_G$.

Let $U_G$ be the unipotent subgroup of the parabolic subgroup of $G$ corresponding to $\mu'$. We have an isomorphism
\begin{equation*}
    U_G=\prod_{\alpha\in S}U_\alpha,
\end{equation*}
where $S=\{\alpha\in\Delta:\langle \alpha,\mu'_0\rangle>0\}$.

Let $U'=\prod_{\alpha\in S'}U_\alpha\subset U_G$ be the product of the roots subspaces, where 
\begin{equation*}
    S'=\{\alpha\in S:\langle \alpha,\nu\rangle\geq 0\}.
\end{equation*}

Let $R'$ denote the completion of $U'$ at the identity. From the construction, we see that $R'$ is a quotient of $R_G$. We now observe the following.

\begin{proposition}
\label{prop:Plocalring}
Let $h: \operatorname{Spec}R\rightarrow \operatorname{Spf}(R_G)$ be a deformation $\tilde{\sG}$ of $\sG$ over a ring $R$ where $R$ is a $p$-nilpotent local ring. Then $h$ factors through $\operatorname{Spf}R'$ if and only if $\tilde{\sG}$ admits a filtration 
\begin{equation*}
    \tilde{\sG}=\tilde{\sG}_0\supset \tilde{\sG}_1\supset \dots\supset \tilde{\sG}_n=0
\end{equation*}
deforming the slope filtration.
\end{proposition}
\begin{proof}
The `only if' part is clear, because there is a filtration on the universal deformation over $\operatorname{Spf}R'$. To show the `if' part, since $R$ is a $p$-nilpotent local ring, we may apply Grothendieck-Messing theory, from which the result follows from the definition of $R'$, since $\operatorname{Spf}R'$ is the subspace where the lifting of the Hodge filtration factors through $P$. 
\end{proof}
\subsubsection{}
Let $E:=D(\sG_{\operatorname{Spf} R_G})[\frac{1}{p}]$, which is a vector bundle over $\operatorname{Spf} R_G$ equipped with a Hodge filtration $\mathrm{Fil}^1\sG_{\operatorname{Spf} R_G}\subset E$. The relative position of the Hodge filtration induces a map
\begin{equation*}
    \pi:(\operatorname{Spf} R_G)^{\mathrm{rig}}\rightarrow \mathrm{Fl}_{G,{\mu}}
\end{equation*}
to the flag variety. (This is the restriction to $\operatorname{Spf} R_G$ of the period map on Rapoport-Zink spaces.)

Let us consider the image of $\pi\vert_{(\operatorname{Spf} R')^{\mathrm{rig}}}$. By the definition of $R'$, $(\operatorname{Spf} R')^{\mathrm{rig}}$ is exactly the preimage of the Schubert cell $S=P/P'\subset \mathrm{Fl}_{G,{\mu}}$, since $\operatorname{Spf} R'$ corresponds to the subspace of $\operatorname{Spf}R_G$ where the Hodge filtration of the associated $p$-divisible group factors through $P$.

\begin{proposition}
\label{prop:subspacelifts}
Let $\tilde{\sG}$ be a $G$-adapted lift, associated to a  map $\varpi:R_G\rightarrow O_K$. Then $\tilde{\sG}$ admits a filtration lifting the slope filtration on $\sG$ if and only if $\varpi$ factors through $R'$.
\end{proposition}
\begin{proof}
Observe that, for rigid points, we can restate the condition that the Hodge filtration factors through $P$ in terms of the relative positions of the Hodge and slope filtration. In particular, for $x\in(\operatorname{Spf} R_G)^{\mathrm{rig}}$, $\pi(x)$ lies in $S$ exactly when (the inverse of) the slope filtration of the the associated weakly admissible $F$-isocrystal $V_x$ (which is a priori only a filtration by $F$-isocrystals) is a filtration by weakly admissible $F$-isocrystals (i.e. the restriction of the Hodge filtration to each filtrand of the slope filtration is weakly admissible). This condition is exactly equivalent to $\tilde{\sG}$ admits a filtration lifting the slope filtration on $\sG$. 
\end{proof}
\subsubsection{}
Let $\tilde{\sG}$ be a $G$-adapted lifting of $\sG$, with a filtration lifting the slope filtration. Then we have a filtration on $L=T_p\sG$, which we denote by $(L^\bullet)$. More generally, we consider $L$ in $\mathrm{Rep}^\mathrm{cris\circ}_{G_K}$ with a $G_K$-stable filtration $L^\bullet$.

Consider the filtration on $\frakM=\frakM(L)$ given by
\begin{equation*}
    \frakM^\bullet:=\frakM(L^\bullet).
\end{equation*}

The following proposition refines Theorem \ref{thm:isomorphismfrakM}.

\begin{proposition}
\label{prop:ptorsor}
  Let $L$ be in $\mathrm{Rep}^\mathrm{cris\circ}_{G_K}$, equipped with $G_K$-invariant tensors $(s_\alpha)$ whose stabilizer is $G$. If $\frakM=\frakM(L)$, then $\frakM(L^\bullet)$ is a $G$-filtration on $\frakM$, which we denote by $\calF_F(\frakM)$. If $k$ is separably closed, then there is an $\frakS$-linear isomorphism.
  \begin{equation*}
      \frakM\xlongrightarrow{\sim}L\otimes_{\bbZ_p}\frakS
  \end{equation*}
  which takes the tensors $\tilde{s}_\alpha$ to $s_\alpha$, and which induces isomorphisms on the filtered pieces
  \begin{equation*}
      \calF_F^\gamma(\frakM)\xlongrightarrow{\sim}\calF_F^\gamma(L)\otimes_{\bbZ_p}\frakS.
  \end{equation*}
\end{proposition}
\begin{proof}
It suffices to prove the proposition when $k$ is separably closed. Now suppose that $k$ is separably closed, set $\frakM'=L\otimes_{\bbZ_p}\frakS$, and let $A\subset \underline{\mathrm{Hom}}_{\frakS}(\frakM,\frakM')$ be the subscheme of isomorphisms between $\frakM$ and $\frakM'$ which take $\tilde{s}_\alpha$ to $s_\alpha$, and induces an isomorphism between the filtered pieces $\calF_F^\gamma(\frakM)$ and $\calF_F^\gamma(L)\otimes_{\bbQ_p}\frakS$. We claim that $A$ is a $P$-torsor over $\frakS$. The claim implies the proposition, since if we let $U$ be the unipotent radical of $P$, then any $P/U$-torsor over $\frakS$ is trivial, since $P/U$ is reductive, and any $U$-torsor over an affine scheme is trivial.

We first observe that $A_{\frakS_{(p)}}$ is a trivial $P$-torsor. Recall that we have a canonical isomorphism over $O_{\widehat{\calE_{ur}}}$
  \begin{equation}
    \label{eqn:canonicalisomtensorsgraded}
      \frakM\otimes O_{\widehat{\calE_{ur}}}\xlongrightarrow{\sim}L\otimes O_{\widehat{\calE_{ur}}}
  \end{equation}
 defined by Fontaine (c.f. \cite[A1]{Fon90}). Since this isomorphism is functorial, the isomorphism above takes $\calF_F^\gamma(\frakM)\otimes \widehat{O_{\calE_{ur}}}$ to  $\calF_F^\gamma(L)\otimes_{\bbZ_p}\widehat{O_{\calE_{ur}}}$. Thus, we see that the filtration on $\frakM\otimes \widehat{O_{\calE_{ur}}}$ given by $\calF_F^\bullet(\frakM)\otimes \widehat{O_{\calE_{ur}}}$ is a $G$-filtration, and moreover is exactly given by $\calF_F(L)$. Thus, we have a (trivial) $P$-torsor over $O_{\widehat{\calE_{ur}}}$, which, by faithfully flat descent, gives us a $P$-torsor over $\frakS_{(p)}$.
 
Moreover, since $\frakM[1/p]$ is an exact faithful tensor functor, $A_{\frakS[1/p]}$ is a $P$-torsor. 

We now view $L$ (resp. $\frakM$) as a tensor functor from $\mathrm{Rep}_{\bbZ_p}G$ to $\mathrm{Rep}^{\mathrm{cris}\circ}_{G_K}$ (resp. $\mathrm{Mod}_{/\frakS}$), denoted by $|L|$ (resp. $|\frakM|). $ In order to show that $\calF_F^\gamma(\frakM)$ is a $G$-filtration, it suffices to show that for every $\gamma\in\bbQ$, the fiber functor $\calF_F^\gamma(|\frakM|):\mathrm{Rep}_{\bbZ_p}G\rightarrow \mathrm{Mod}_{/\frakS}$ given for every $\tau\in \mathrm{Rep}_G$ by $\calF_F^\gamma(|\frakM|)(\tau)=\calF_F^\gamma(|\frakM|(\tau))=\frakM(\calF_F^\gamma(|L|(\tau)))$ is exact. The above arguments show that the restriction of  $\calF_F^\gamma(|\frakM|)$ to $\frakS[1/p]$ and $\frakS_{(p)}$ are exact functors, and thus we have an exact tensor functor over $U=\operatorname{Spec}\frakS\backslash\{\mathfrak{m}\}$. Thus, if we consider the scheme of isomorphisms
\begin{equation*}
    B=\mathrm{Isom}(\calF_F^\gamma(|L|(\tau))\otimes\frakS,\calF_F^\gamma(|\frakM|)),
\end{equation*}
then $B_U$ is a $G$-torsor, and by \cite[Thm 6.13]{CS79}, we can extend $B_U$ to a torsor over $\frakS$, which is necessarily equal to $B$. Thus $\calF_F^\gamma(|\frakM|)$ is exact, so $\calF_F^\bullet(|\frakM|))$ is a $G$-filtration, and thus $A$ is a $P$-torsor. Since $\frakS$ is a strictly henselian local ring, $A$ is necessarily trivial. We hence choose a section $f\in A(\frakS)$, and as described above, this proves the claim.
\end{proof}
\subsection{Rapoport-Zink Spaces}
We now define the Rapoport-Zink space for a triple $(G,b,\mu)$ of Hodge type, following \cite{kim_2018}, where $G$ is an unramified, connected reductive group over $\bbQ_p$, $b\in G(L)$ and $\mu\in X_*(T)$. Let $\bbX$ be a $p$-divisible group over $k$, such that the Frobenius $\varphi$ on $\bbD(\bbX)$ is given by $b\sigma$.

\begin{definition}
  Let $\RZ(G,b,\mu)$ be the functor which assigns to any $p$-locally nilpotent smooth $W$-algebra $R$ the set of isomorphism classes $((X,\rho,t_\alpha)$ such that
  \begin{enumerate}
      \item $(X,t_\alpha)$ is a $p$-divisible group over $R$ with tensors $t_\alpha\subset\bbD(X)^\otimes$, where $(t_\alpha)$ consists of morphisms of crystals $t_\alpha : 1 \rightarrow \bbD(X)^\otimes$ over $\mathrm{Spec}(R)$ such that
      \begin{equation*}
          t_\alpha : 1[1/p] \rightarrow \bbD(X)^\otimes[1/p] 
      \end{equation*}
    is Frobenius equivariant;
      \item $\rho:\mathbb{X}_{R/p}\rightarrow X_{R/p}$ is a quasi isogeny;
      \item For some nilpotent ideal $J \subset R$ containing $(p)$, the pull-back of $t_\alpha$ over $\mathrm{Spec}(R/J)$ is identified
    with $s_{\alpha}$ under the isomorphism of isocrystals induced by $\rho$:
    \begin{equation*}
        \bbD(X_{R/J})[1/p] \xrightarrow{\sim}\bbD(\bbX_{R/J})[1/p].
    \end{equation*}
    \item For some (any) formally smooth $p$-adic $W$-lift  $\tilde{R}$ of $R$, endowed with the standard PD-structure on $\mathrm{ker}(\tilde{R}\rightarrow R) = p^m \tilde{R}$ for some $m$, let $(t_\alpha(\tilde{R}))$ denote the  $\tilde{R}$-section of $(t_\alpha)$. Then the $\tilde{R}$-scheme
\begin{equation*}
    P(\tilde{R}) := \mathrm{Isom}_{\tilde{R}}([\bbD(X)_{\tilde{R}}),(t_\alpha(\tilde{R}))],[\tilde{R}\otimes_{\bbZ_p} {\Lambda^*},(1\otimes s_\alpha)]),
\end{equation*}
classifying isomorphisms matching $(t_\alpha(\tilde{R}))$ and $(1\otimes s_\alpha)$, is a $G_W$-torsor.
    \item The Hodge filtration $\mathrm{Fil}^1(X) \subset \bbD(X)(R)$ is a $\{\mu\}$-filtration with respect to $(t_\alpha(R)) \subset \bbD(X)(R)^\otimes$, where $\{\mu\}$ is the unique $G(W)$-conjugacy class of cocharacters such that $b\in G(W)p^\upsilon G(W)$. 
  \end{enumerate}
\end{definition}

Now, fix $[b]\in B(G,\upsilon)$, and suppose that $P$ is the standard parabolic subgroup of $G$ given as the stabilizer of $\nu_{[b]}$. By $\sigma$-conjugating $b$, we may choose a representative $\bbX$ which is fully slope-decomposable, i.e. we have a decomposition
\begin{equation*}
    \bbX = \bbX'_1 \oplus\dots \oplus \bbX'_m 
\end{equation*}
such that $\bbX_i$ is isoclinic. Such a representative always exists, see for example \cite[\S2]{Ham2017}. In particular, we have a slope filtration on $\bbX$ as
\begin{equation*}
    0 \subset \bbX_1 \subset\dots \subset \bbX_t = \bbX,
\end{equation*}
for $\bbX_j = \oplus_{1\leq i\leq j}\bbX'_i$.
\begin{definition}
  Let $\RZ(P,b,\mu)$ be the subfunctor of $\RZ(G,b,\mu)$ which to a $W$-algebra $R$ associates the set of isomorphisms classes of triples $(H,H^{\bullet},\rho, t_\alpha)$, where
  \begin{enumerate}
        \item $(H,\rho,t_\alpha)\in \RZ(G,b,\mu)(R)$,
      \item $H^\bullet$ is an increasing filtration of $H$ by $p$-divisible groups over $\mathrm{Spec}R$, with $p$-divisible subquotients,
      \item $\rho : \bbX_{R/p} \rightarrow  H_{R/p}$ is a quasi-isogeny compatible with the filtration, i.e. 
      \begin{equation*}
          \rho(\bbX_{j,R/p}) \subseteq H_{j,R/p}\text{ for any }j = 1,\dots,t;
      \end{equation*}
      such that the restrictions of $\rho$ to the Barsotti-Tate subgroups defining the filtration
      \begin{equation*}
        \rho_j : \bbX_{j,R/p}\rightarrow H_{j,R/p}
      \end{equation*}
        are quasi-isogenies.
        \item For some (any) formally smooth $p$-adic $W$-lift  $\tilde{R}$ of $R$, endowed with the standard PD-structure on $\mathrm{ker}(\tilde{R}\rightarrow R) = p^m \tilde{R}$ for some $m$, let $(t_\alpha(\tilde{R}))$ denote the  $\tilde{R}$-section of $(t_\alpha)$. Then the $\tilde{R}$-scheme
\begin{equation*}
    P(\tilde{R}) := \mathrm{Isom}_{\tilde{R}}([\bbD(X)^\bullet_{\tilde{R}}),(t_\alpha(\tilde{R}))],[(\tilde{R}\otimes_{\bbZ_p} {\Lambda^*})^\bullet,(1\otimes s_\alpha)]),
\end{equation*}
classifying isomorphisms matching $(t_\alpha(\tilde{R}))$ and $(1\otimes s_\alpha)$, and the filtration, is a $P_W$-torsor.
  \end{enumerate}
\end{definition}
The same proof as in the PEL type case (see \cite[Prop 4.3]{Man08}) shows that $\RZ(P,b,\mu)$ is represented by a formal scheme over $O_E$, formally
locally of finite type. Moreover, the formal scheme $\RZ(P,b,\mu)$ is formally smooth, since the complete local ring at any point $x$ is, from Prop \ref{prop:Plocalring}, exactly $\mathrm{Spf}R'$. 

We now recall a key result \cite[Prop 5.1]{Man08} concerning the liftings of filtrations of $p$-divisible groups. Let $S$ be a connected scheme of characteristic $p$.
\begin{proposition}
\label{prop:strata}
Let $\bbX = \oplus_i\bbX'_i$ be a decomposition of a fixed $p$-divisible group over $k$ and write $\bbX_j = \oplus_{0\leq i\leq j}\bbX_i$ (j = 1,...,t) for its filtration 

Let $\sG$ be a $p$-divisible group over a noetherian $k$-scheme $S$, together with a quasi-isogeny $\rho : \bbX_S \rightarrow G$. Then there exists a stratification of $S$ by closed subschemes $\{S_i\}$
\begin{equation*}
    S = S_0 \supset S_1 \supset\dots \supset S_r \supset S_{r+1} = \null;\qquad S^\circ_s = S_s-S_{s+1}
\end{equation*}
such that
\begin{enumerate}
    \item for each $s$, the restriction of $\sG$ to the locally closed subscheme $S^\circ_s$ admits an increasing filtration $\sG^{\bullet}$ satisfying the condition: for each $j = 1,...,t$, the restriction of $\rho$ to the subgroup $\bbX_j$ induces a quasi-isogeny $\rho_j : \bbX_j \rightarrow \sG_j$;
    \item for any connected scheme $Z$ and any morphism $f:Z \rightarrow S$, $f^*\sG$ admits a filtration by $p$-divisible subgroups with the above property if and only if the morphism $f$ factors via the inclusion $S^\circ_s \hookrightarrow S$, for some $s\geq 0$.
\end{enumerate}
\end{proposition}

\section{Affine Deligne-Lusztig varieties and isogeny classes mod $p$}
In this section, we will recall results about isogeny classes of $p$-divisible groups with $G$-structure, and the irreducible components of affine Deligne-Lusztig varities of unramified elements. 
\subsection{Affine Deligne-Lusztig varieties}
\subsubsection{}Let $L=W(\overline{\bbF})p)[1/p]$ and $O_L=W(\overline{\bbF}_p)$. Denote by $\Gamma := \operatorname{Gal}(\overline{\bbF}_p/\bbF_p)$ the absolute Galois group. Fix a connected reductive group $G$ over $\bbZ_p$. We denote by $G_{\bbQ_p}$ the generic fibre of $G$. Let $T\subset G$ be the centralizer of a maximal split torus, $B\supset T$ a Borel subgroup of $G$.

For $b\in G(L)$, let $[b]$ denote the $\sigma$-conjugacy class of $b$, and $B(G)$ denote the set of $\sigma$-conjugacy classes in $G(L)$.

Recall that since $G$ is unramified, it splits over $L$, so we have the Cartan decomposition
\begin{equation*}
    G(L)=\coprod_{\mu\in X_*(T)^+}G(O_L)p^\mu G(O_L).
\end{equation*}
where $X_*(T)^+$ is the set of dominant cocharacters. For any $b\in G(L)$, let $\mu_b$ denote the dominant cocharacter such that $b$ lies in the double coset of $p^{\mu_b}$ under the Cartan decomposition.

Let $\pi_1(G)$ denote the quotient of $X_*(T)$ by the space of coroots of $G$. Let $\tilde{\kappa}_G$ denote the composition
\begin{equation*}
    \tilde{\kappa}_G:G(L)\xrightarrow{b\mapsto \mu_b} X_*(T)/\Omega\rightarrow \pi_1(G)
\end{equation*}
Composing with the projection $\pi_1(G)\twoheadrightarrow \pi_1(G)_{\Gamma}$, we obtain a map $\kappa$, which is known as the Kottwitz homomorphism. Observe that $\kappa$ only depends on the $\sigma$-conjugacy class $[b]$.

For $\mu\in X_*(T)$, we denote by $\mu^\sharp$ the image of $\mu$ in $\pi_1(G)_\Gamma$.
\subsubsection{}
Recall the Newton map $\nu_b:\bbD(\bbQ)\rightarrow G$ defined by Kottwitz \cite[4.2]{Kott85}. For $G=GL(V)$, $\nu_b$ is the cocharacter which induces the slope decomposition of the isocrystal $(V,\varphi=b\sigma)$. In general, it is the cocharacter inducing the slope $\bbQ$-graduation of the trivial $G$-isocrystal with Frobenius $\varphi=b\sigma$. 

Observe that for any $g\in G(L)$, we have $\nu_{gb\sigma(g)^{-1}} =g\nu_bg^{-1}$. Thus, the slope homomorphism on the $\sigma$-conjugacy class $[b]$ is defined up to $G(L)$-conjugacy.  We set $\nu([b])\in (X_*(T)\otimes \mathbb{Q})^+$ to be the unique dominant element in the conjugacy class. This is clearly independent of all choices. 

Recall that we have a partial ordering on $(X_*(T)\otimes \mathbb{Q})^+$ given by $\lambda_1\leq\lambda_2$ if $\lambda_2-\lambda_1$ is a non-negative linear combination of simple coroots of $G$.

For any $\mu\in X_*(T)$, let $F\subset L$ be a finite unramified Galois extension of $\mathbb{Q}_p$ over which $\mu$ is defined, and let
\begin{equation*}
  \bar{\mu}=\frac{1}{[F:\mathbb{Q}_p]}\sum_{\tau\in \operatorname{Gal}(F/\mathbb{Q}_p)}\tau(\mu).
\end{equation*}
\subsubsection{}
For any $b\in G(L)$, and any minuscule cocharacter $\mu\in X_*(T)$, we define the affine Deligne-Lusztig variety
\begin{equation*}
    X_\mu(b):=\{g\in G(L)/G(O_L):g^{-1}b\sigma(g)\in G(O_L)p^\mu G(O_L)\}.
\end{equation*}
$X_\mu(b)$ depends only upon the $\sigma$-conjugacy class $[b]$, and we can replace $\mu$ by any Weyl conjugate. Thus, we may assume that $\mu$ is dominant with respect to some choice of borel $B$ and maximal torus $T$. By the work of \cite{Zhu2014}, $X_\mu(b)$ is the set of $\overline{\bbF}_p$ points of a perfect scheme over $\bar{\bbF}_p$, and a subscheme of the (Witt vector) affine Grassmannian.

We define
\begin{equation*}
    B(G,\mu):=\{[b]:\kappa_G([b])=\mu^\sharp\in\pi_1(G)_\Gamma\text{ and }\nu([b])\leq\bar{\mu}\}.
\end{equation*}

We have a partial order on $B(G,\mu)$ such that $[b_1]\preceq [b_2]$ if and only if $\nu([b_1])\leq \nu([b_2])$. By \cite{Win05}, we know that $X_\mu(b)$ is non-empty if and only if $b\in B(G,\mu)$.
\subsubsection{}
We define the algebraic group $J_b$ over $\bbQ_p$ by
\begin{equation*}
    J_b(R) :=\{g\in G(R\otimes_{\bbQ_p} L): g^{-1}b\sigma(g)=b\}.
\end{equation*}
There is an inclusion $J_b\subset G$, defined over $L$, which is given on $R$-points ($R$ an $L$-algebra) by the natural map $G(R\otimes_{{\bbQ}_p}L)\rightarrow G(R)$.Then the inclusion $J_b\subset G$ identifies $J_b$ with $M$ over $L$, and moreover $J_b$ is an inner form of $M$ \cite[3.3]{Ko97}.

For any $b\in G(L)$, we define
\begin{equation*}
    \defect_G(b)=\operatorname{rk}_{\bbQ_p}(G)-\operatorname{rk}_{\bbQ_p}(J_b),
\end{equation*}
where $\operatorname{rk}_{\mathbb{Q}_p}(G)$ is the rank of the maximal split $\mathbb{Q}_p$ torus of the algebraic group $G$.

Let us now recall a key result about the dimensions of affine Deligne-Lusztig varieties (c.f. \cite[Thm 3.1]{Zhu2014})
\begin{theorem}
\label{thm:dimADLV}
    Assume the affine Deligne-Lusztig variety is non-empty, i.e. $[b]\in B(G,\mu)$. The dimension of the affine Deligne-Lusztig variety  $X_{\mu}(b)$ is
    \begin{equation*}
    \langle\rho,\mu-\nu([b])\rangle-\frac{1}{2}\operatorname{def}_{G}(b).
    \end{equation*}
\end{theorem}
\subsection{Connected components}
\label{section:connectedcomponentADLV}
We now assume that $\mu$ is a minuscule cocharacter of $G$. Let $[b]$ be a $\sigma$-conjugacy class of $G$, lying in $B(G,\mu)$. By \cite[2.5.2]{CKV} we can choose a representative $b$ of $[b]$ such that $\nu_b\in (X_*(T)\otimes\bbQ)^{+}$, is defined over $\bbQ_p$, and is $G$-dominant. Moreover, $b\in M_b$, the centralizer of $\nu_b$. A standard Levi subgroup $M\subset G$ is one defined by a dominant cocharacter in $X_*(T)^\Gamma$. For the choice of $b$ above, $M_b$ is a standard Levi subgroup. 
\subsubsection{}
By \cite[2.5.4]{CKV}, if $M\supset M_b$ is a standard Levi subgroup and $\kappa_M(b)=\mu^\natural$, then the natural inclusion
\begin{equation*}
    X^M_\mu(b)\rightarrow X^G_\mu(b)
\end{equation*}
is an isomorphism. The pair $(\mu,b)$ is called HN-indecomposable if $\kappa_M(b)\neq\mu^\natural$ for any standard Levi $M\supset M_b$. The following is \cite[Thm 1.1]{CKV}.
\begin{theorem}
\label{thm:CKV}
 Assume that $G^{\mathrm{ad}}$ is simple and that $\mu$ is minuscule, and suppose that $(\mu,b)$ is Hodge-Newton indecomposable in $G$. Then $\kappa_G$ induces a bijection
 \begin{equation*}
     \pi_0(X_\mu(b))\simeq c_{b,\mu} \pi_1(G)^\Gamma
 \end{equation*}
 for some $c_{b,\mu}\in\pi_1(G)$, unless $[b] = [p^\mu]$ with $\mu$ central in $G$, in which case we have an isomorphism
 \begin{equation*}
     X_\mu(b)\simeq G(\bbQ_p)/G(\bbZ_p)
 \end{equation*}
 so $X_\mu(b)$ is discrete.
\end{theorem}
Moreover, by \cite[Cor 2.4.2]{CKV}, for $\omega\in c_{b,\mu}\pi_1(G)^\Gamma$, with image $\omega_{\mathrm{ad}}\in c_{b_{\mathrm{ad}},\mu_{\mathrm{ad}}}\pi_1(G^\mathrm{ad})^\Gamma$ we have an isomorphism of connected components
\begin{equation*}
    X_{\mu_{\mathrm{ad}}}(b_{\mathrm{ad}})^{\omega_{\mathrm{ad}}}\xlongrightarrow{\sim}X_\mu(b)^\omega.
\end{equation*}
\subsection{Reduction to Levi subgroup}
We continue to assume that $\mu$ is a minuscule cocharacter of $G$. Let $[b]$ be a $\sigma$-conjugacy class of $G$, lying in $B(G,\mu)$. Consider $M_{[b]}$, the standard Levi subgroup given as the centralizer of the Newton cocharacter $\nu([b])$. For the rest of this subsection, for notational simplicity, we let $M:=M_{[b]}$. We continue to assume that the representative $b$ of this $\sigma$-conjugacy class satisfies $b\in M(L)$, and moreover $\nu_b$ lies in $(X_*(T)\otimes\bbQ)^{+}$, is defined over $\bbQ_p$, and is $G$-dominant.

\subsubsection{}
\label{section:reductionLevi}
We now want to reduce to understanding the irreducible components in the case where $b$ is basic. To do this, we will recall the results of \cite{HV2018}, applying the well-known reduction method of \cite{GHKR}.

Let $P=MN$ be the parabolic subgroup associated to $\nu_b$, where $N$ is the unipotent radical of $P$. Since $b\in M(L)$, we have an induced decomposition
\begin{equation*}
  J_b(\mathbb{Q}_p)\cap P(L)=(J_b(\mathbb{Q}_p)\cap M(L))(J_b(\mathbb{Q}_p)\cap N(L)).
\end{equation*}
As discussed above, we have $J_b(\mathbb{Q}_p)\subset M(L)$. In particular, we see that $J_b(\mathbb{Q}_p)\cap N(L)=\{1\}$.

We now consider the following subvariety of the affine Grassmanian $\Gr_M$ of $M$:
    \begin{equation*}
        X^{M\subset G}_\mu(b)=\{gM_{c}\in \Gr_{M}|g^{-1}b\sigma(g)\in K\mu(p) K\}.
    \end{equation*}
Note that $X^{M\subset G}_\mu(b) =\coprod_{\mu'\in I_{\mu,b}} X^{M}_{\mu'}(b)$, where $I_{\mu,b}$ is the set of $M$-conjugacy classes of cocharacters $\mu'$ in the $G$-conjugacy class of $\mu$ with $[b]_{M}\in B(M, \mu')$. Such a set is non-empty, see Remark \ref{rk:nonemptyI} below.

We also consider
    \begin{equation*}
        X^{P\subset G}_\mu(b):=\{gK_P\in \Gr_P|g^{-1}b\sigma(g)\in K\mu(p)K\}.
    \end{equation*}

The Iwasawa decomposition shows that $X^{P\subset G}_\mu(b)$ is a decomposition of $X^G_\mu(b)$ into locally closed subsets, hence the set of irreducible components of $X^{P\subset G}_\mu(b)$ is equal to the set of irreducible components of $X^G_\mu(b)$. Moreover, the surjection $\Gr_P\twoheadrightarrow \Gr_{M}$ induces a surjection 
\begin{equation*}
  \beta: X^{P\subset G}_\mu(b)\rightarrow X^{M\subset G}_\mu(b).
\end{equation*}
We want to investigate the image of irreducible components under $\beta$. Define the set
\begin{equation*}
  \Sigma'(X_\mu^{M\subset G}(b)):=\bigcup_{\mu'\in I_{\mu,b}} \Sigma^{\text{top}}(X_{\mu'}^M(b))
\end{equation*}
where $\Sigma^{\text{top}}$ denotes the set of top-dimensional irreducible components. Note that $X^{M\subset G}_\mu(b)$ may not be equidimensional, as the example in \cite[\S8]{RV2014} shows.

We recall the following two results \cite[5.5,5.6]{HV2018}.

\begin{lemma}
  $\beta$ induces a well-defined surjective map 
  \begin{equation*}
    \beta_{\Sigma}:\Sigma^{\text{top}}(X_\mu^{P\subset G}(b))\rightarrow \Sigma'(X_\mu^{M\subset G}(b)),
  \end{equation*}
  which is $J_b(\mathbb{Q}_p)\cap P(L)$-equivariant for the natural action of the left hand-side, and the action through the natural projection $J_b(\mathbb{Q}_p)\cap P(L)\twoheadrightarrow J_b(\mathbb{Q}_p)\cap M(L)$ on the right hand side.
\end{lemma}

\begin{proposition}
  Let $Z\subset  X_\mu^{M\subset G}(b)$ be an irreducible subscheme. Then $J_b(\mathbb{Q}_p)\cap N(L)$ acts transitively on $\Sigma(\beta^{-1}(Z))$.
\end{proposition}

Since $J_b(\mathbb{Q}_p)\cap N(L)=\{1\}$ we immediately get the following corollary.

\begin{corollary}
\label{cor:bijirreducible}
  The map $\beta_{\Sigma}$ induces a bijection between irreducible components of $X^G_\mu(b)$ and irreducible components of $\coprod_{\mu'\in I_{\mu,b}} X^{M}_{\mu'}(b)$.
\end{corollary}

\subsection{Unramified $\sigma$-conjugacy classes}
\label{section:unramifiedclasses}
Recall that we have chosen a Borel subgroup $B$ and maximal torus $T$ such that $T\subset B\subset G$. Observe that we have a canonical map $B(T)\rightarrow B(G)$. 
\begin{lemma}
\label{lemma:unramified}
    The following are equivalent:
    \begin{enumerate}
        \item $[b]$ lies in the image of $B(T)$
        \item $\defect_G(b)=0$
        \item There is a representative of $[b]$ of the form $p^\tau$, for some $\tau\in X_*(T)$
    \end{enumerate}
\end{lemma}
\begin{proof}
We denote the image of $B(T)$ by $B(G)_{\mathrm{unr}}$. This result is contained in the results of \cite[\S4.2]{XZ17}, and we sketch the proof here. Firstly, the canonical map $X_*(T)\rightarrow B(G)$, $\tau\mapsto [p^\tau]$ induces a bijection $X_*(T)_\sigma/\Omega(\mathbb{Q}_p)\simeq B(G)_{\mathrm{unr}}$, where $X_*(T)_\sigma$ are the $\sigma$-coinvariants of $X_*(T)$. This shows $(1)\Leftrightarrow (3)$.

For any $b\in G(L)$, we can choose some standard Levi subgroup $M$ and cocharacter $\mu$ such that $b$ is $\sigma$-conjugate to a basic element in $B(M,\mu)$. If $[p^\tau]$ is basic in $M$, then $J_{p^\tau}$ contains the torus $T$ and therefore, $\defect_G(b) = 0$.  Conversely, $\defect_G(b) = 0$ implies that $b_{\mathrm{ad}}$ is $\sigma$-conjugate to $1$ in $M_{ad}(L)$. Since $M_{ad}(L)$ is generated by the image of $M(L)\rightarrow G_{ad}(L)$ and $T_{ad}(L)$, $b$ is $\sigma$-conjugate to $p^\tau$ for some $\tau\in X_*(T)$. This shows $(2)\Leftrightarrow(3)$ 
\end{proof}
\begin{remark}
\label{rmk:splitnounramified}
  We repeat here the remark \cite[4.2.11]{XZ17}. The identity element represents the basic element in $B(M_{ad},\mu_{ad})$ if and only if $[1]\in B(M_{ad},\mu_{ad})$, i.e. $\mu_{ad} = 0$ in $\pi_1(M_{ad})_{\Gamma}$. If $M_{ad}$ is split, then $\mu_{ad} = 0$ in $\pi_1(M_{ad})_{\Gamma}$ implies that $\mu_{ad}$ is the sum of coroots of $M_{ad}$, and hence $\mu_{ad}$ is either the identity or it cannot be minuscule. 
\end{remark}
If $[b]\in B(G)$ satisfies any of the equivalent conditions in Lemma \ref{lemma:unramified}, we say that $[b]$ is \emph{unramified}.

We can define the set of dominant elements in $X_*(T)_\sigma$ as
\begin{equation*}
    X_*(T)^+_{\sigma}=\left\{\lambda_\sigma\in X_*(T)_\sigma\vert \left\langle\lambda,\sum_{i=0}^{k-1}\sigma^i(\alpha)\right\rangle\geq 0\text{ for every }\alpha\in\Delta\right\}
\end{equation*}
where $k$ is degree of the splitting field of $G$ over $\bbQ_p$. The set $X_*(T)^+_{\sigma}$ is in bijection with $X_*(T)_\sigma/\Omega(\mathbb{Q}_p)$ under the natural map.
\subsubsection{} We now recall the results of \cite[\S4.4]{XZ17} which give us a more explicit description of $X^G_\mu(b)$ in the case where $G$ has a connected center.

Let us first recall results (c.f. \cite[\S3.2]{XZ17}) about the geometry of semi-infinite orbits in the affine Grassmannian and Mirkovi\'{c}-Vilonen cycles.

Let $H$ be an affine group scheme of finite type defined over $\mathscr{O}$, the ring of integers of a non-archimedian local field. We denote by $L^+H$ (resp. $LH$) the jet group (resp. loop group) of $H$. As a presheaf, we have $L^+H(R)= H(W_\mathscr{O}(R))$ and $LH(R) =H(W_\mathscr{O}(R)[1/p])$, where 
\begin{equation*}
    W_{\sO}(R):=W(R)\hat{\otimes}\sO
\end{equation*}
is the ring of Witt vectors with coefficients in $\sO$. $L^+H$ (resp. $LH$) is represented by an affine group scheme (resp. ind-scheme). 

We define the (spherical) Schubert variety $Gr_\mu$ as the closed subset 
\begin{equation*}
  \Gr_\mu=\{(E,\beta)\in \Gr|\mathrm{Inv}(\beta)\prec\mu\}
\end{equation*} of the affine Grassmannian $\Gr$ of $G$ over $\bar{\bbF}_p$. It contains the Schubert cell 
\begin{equation*}
  \mathring{\Gr}_\mu:=\{(E,\beta)\in \Gr|\mathrm{Inv}(\beta)=\mu\}= \Gr_\mu\backslash \cup_{\lambda\prec\mu}\Gr_\lambda
\end{equation*}

Let $U$ be the unipotent radical of the Borel subgroup $B$.
\begin{definition}
  For $\lambda$ a coweight of $G$, the semi-infinite orbit associated to $\lambda$ is
  \begin{equation*}
    S_\lambda:=LUp^\lambda L^+G/L^+G.
  \end{equation*}
\end{definition}

The following result is the key result about the geometry of semi-infinite orbits.
\begin{theorem}[Mirkovi\'{c}-Vilonen]
  For $\lambda$ and $\mu$ two coweights of G with $\mu$ dominant, every irreducible component of  the intersection $S_\lambda\cap \Gr_\mu$ is of dimension $\langle \rho,\lambda+\mu\rangle$.  In  addition, the number of its irreducible components equals to the dimension of the $\lambda$-weight space $V_\mu(\lambda)$ of the irreducible representation $V_\mu$ of $\hat{G}$ of highest weight $\mu$.
\end{theorem}

Irreducible components of $S_\lambda\cap \Gr_\mu$ are referred to as Mirkovi\'{c}-Vilonen cycles. We will denote by $\mathbb{MV}_\mu(\lambda)$ the set of irreducible components  of $S_\lambda\cap \Gr_\mu$, and for $\mathbf{b}\in \mathbb{MV}_\mu(\lambda)$, write $(S_\lambda\cap \Gr_\mu)^{\mathbf{b}}$ the irreducible component (a.k.a. MV cycle) labeled by $\mathbf{b}$.

While the results above hold for any general dominant coweight $\mu$, we are primarily interested in the case when $\mu$ is minuscule. If $\mu$ is a minuscule coweight of $G$, then $\Gr_\mu= \mathring{\Gr}_\mu$ and $S_\lambda \cap \Gr_\mu$ is non-empty if and only if $\lambda=w\mu$ for some $w\in\Omega(\overline{\mathbb{Q}}_p)$. Thus, 
\begin{equation*}
  S_\lambda \cap \Gr_\mu=L^+U p^\lambda L^+G/L^+G\simeq L^+U p^\lambda L^+U/L^+U
\end{equation*}
is irreducible.
\subsubsection{}
Now, suppose that $\tau_\sigma\in X_*(T)^+_\sigma$. Consider the intersection
\begin{equation*}
  Y_\nu:=S_\nu\cap X_\mu(p^\tau),
\end{equation*}
which is a locally closed sub ind-scheme of $X_\mu(\tau)$. It fits into the following Cartesian diagram
\begin{equation*}
  \begin{tikzcd}
    Y_\nu\arrow{rr}\arrow{dd} && (S_\nu\tilde{\times}S_\lambda)\cap (\Gr\tilde{\times}\Gr_\mu)\arrow{dd}{pr_1\times m}\\
    &&\\
    S_\nu\arrow{rr}{1\times p^\tau\sigma} && S_\nu\times S_{\tau+\sigma(\nu)}
  \end{tikzcd}
\end{equation*}
where $\lambda=\tau+\sigma(\nu)-\nu$, and $\tilde{\times}$ denotes the twisted product, and $m$ is the product of the convolution map. Since $(S_\nu\tilde{\times}S_\lambda)\cap(Gr\tilde{\times}Gr_\mu) =S_\nu\tilde{\times}(S_\lambda \cap Gr_\mu)$ by \cite[3.2.8]{XZ17}, every $\mathbf{b}\in \mathbb{MV}_\mu(\lambda)$ gives a closed subset $Y^{\mathbf{b}}_\nu=Y^{\mathbf{b}}_\nu(\tau)$ of $Y_\nu$ that fits into the Cartesian diagram
\begin{equation*}
  \begin{tikzcd}
    Y_\nu^\mathbf{b}\arrow{rr}\arrow{dd} && S_\nu\tilde{\times}(S_\lambda \cap Gr_\mu)^\mathbf{b}\arrow{dd}{pr_1\times m}\\
    &&\\
    S_\nu\arrow{rr}{1\times p^\tau \sigma} && S_\nu\times S_{\tau+\sigma(\nu)}
  \end{tikzcd}
\end{equation*}

We need the following lemma. Let $\Delta$ denote the set of simple roots of $G$. In \cite[\S3.3]{XZ17} the set $\mathbb{MV}_\mu$ is endowed with a $\hat{G}$-crystal structure and therefore to every $\mathbf{b}\in \mathbb{MV}_\mu(\lambda)$ can be attached a collection of non-negative integers $\{\varepsilon_\alpha(\mathbf{b}),\alpha\in \Delta\}$. The following lemma is \cite[4.4.3]{XZ17}.

\begin{lemma}
\label{lem:taub}
  Assume that $Z_G$ is connected.
  \begin{enumerate}
    \item Let $\lambda\in X_*(T)$ and assume that $\lambda_\sigma=\tau_\sigma$ in $X_*(T)_\sigma$. Then there exists a dominant coweight $\nu$ such that $\lambda+\nu-\sigma(\nu)$ is dominant, and $\langle \nu,\alpha\rangle\geq \varepsilon_\alpha(\mathbf{b})$ for all $\alpha\in\Delta$.
    \item Among all $\nu$'s satisfying the above property, there is a  ``minimal" $\nu_\mathbf{b}$, unique up to addition by an element in $X_*(Z_G)$.  Here ``minimality" means that for any other $\nu$ satisfying the above property, $\nu-\nu_\mathbf{b}$ is dominant. In addition,  for every $\alpha\in\Delta$, at least one of the following inequalities is an equality
    \begin{equation*}
      \langle\alpha,\nu_\mathbf{b}\rangle\geq \varepsilon_\alpha(\mathbf{b}), \langle\alpha_{\sigma(\alpha)},\nu_\mathbf{b}\rangle\geq \varepsilon_{\sigma(\alpha)}(\mathbf{b}), \dots, \langle\alpha_{\sigma^{d-1}(\alpha)},\nu_\mathbf{b}\rangle\geq \varepsilon_{\sigma^{d-1}(\alpha)}(\mathbf{b})
    \end{equation*}
    where $d$ is the cardinality of the $\sigma$-orbit of $\alpha$
  \end{enumerate}
\end{lemma}

For a (choice of) minimal weight $\nu=\nu_\mathbf{b}$ as above, we let $\tau_\mathbf{b}=\lambda+\nu_{\mathbf{b}}-\sigma(\nu_\mathbf{b})$

\begin{lemma}
\label{lem:nub}
  Suppose $\mu$ is minuscule. Let $\mathbf{b}\in\mathbb{MV}_\mu(\lambda)$, and $n$ be the minimal positive integer such that $\sigma^n(\lambda)=\lambda$. Then $\sigma^n(\nu_\mathbf{b})-\nu_{\mathbf{b}}\in X_*(Z_G)$.
\end{lemma}
\begin{proof}
  Firstly, since $\mu$ is minuscule, we must have $\lambda=\omega(\mu)$ for some $\omega\in\Omega(\overline{\bbQ}_p)$. Moreover, since $\tau_{\mathbf{b}},\nu_{\mathbf{b}}$ are dominant, so too are $\sigma^n(\tau_{\mathbf{b}}),\sigma^n(\nu_{\mathbf{b}})$. In this case, we have
  \begin{equation}
      \varepsilon_\alpha(\mathbf{b})=\max\{0,-\langle\lambda,\alpha\rangle\}.
  \end{equation}
  Since $\sigma^n(\lambda)=\lambda$, we see that for any $\alpha\in\Delta$, we have $\varepsilon_\alpha(\mathbf{b})=\varepsilon_{\sigma^n(\alpha)}(\mathbf{b})$. Thus, $\sigma^n(\nu_\mathbf{b})$ also satisfies the condition in Lemma \ref{lem:taub}(1), and thus $\sigma^n(\nu_\mathbf{b})-\nu_\mathbf{b}$ must be dominant. However, $\sigma^n(\nu_\mathbf{b})-\nu_\mathbf{b}$ dominant implies
  \begin{equation*}
      \langle \sigma^n(\nu_\mathbf{b}),\alpha \rangle\geq \langle \nu_\mathbf{b},\alpha \rangle
  \end{equation*}
  for all $\alpha\in\Delta$. Since $\langle \sigma^n(\nu_\mathbf{b}),\alpha \rangle= \langle \nu_\mathbf{b},\sigma^{-n}(\alpha) \rangle$, we must have
  $\langle \sigma^n(\nu_\mathbf{b})-\nu_\mathbf{b},\alpha\rangle=0$, for all $\alpha$, hence $\sigma^n(\nu_\mathbf{b})-\nu_\mathbf{b}\in X_*(Z_G)$.
\end{proof}

For $\mathbf{b}\in \cup_{\lambda\in \tau +(\sigma-1)X_*(T)}\mathbb{MV}_\mu(\lambda)$, let $X^{\mathbf{b}}_\mu(p^\tau)$ denote the closure of $\cup_{\nu\in X_*(T)}Y^{\mathbf{b}}_\nu$ in $X_\mu(p^\tau)$. Note that by \cite[Lem. 4.3.6 (2)]{XZ17}, the isomorphism $X_\mu(p^\tau)\xrightarrow{\sim}X_\mu(p^{\tau_\mathbf{b}})$ induces an isomorphism $X^{\mathbf{b}}_\mu(p^{\tau})\xrightarrow{\sim}X^{\mathbf{b}}_\mu(p^{\tau_{\mathbf{b}}})$. Observe that since semi-infinite  orbits  form  a  partition  of  the  affine  Grassmannian  into  locally  closed subsets, ${Y^\mathbf{b}_\nu}$ form  a  partition  of $X_\mu(b)$ into locally closed subsets. Thus, $X^\mathbf{b}_\mu(p^\tau)$ is a union of some irreducible components of $X_\mu(p^\tau)$, and $X_\mu(p^\tau) =\cup_\mathbf{b}X^\mathbf{b}_\mu(p^\tau)$. The following result, contained in \cite[4.4.5, 4.4.7]{XZ17}, describes an irreducible component of $X^\mathbf{b}_\mu(p^\tau)$.
\begin{theorem}
  Assume that $Z_G$ is  connected.  Let $\mathbf{b}\in \mathbb{MV}_\mu(\lambda)$.  Let $\nu=\nu_\mathbf{b}$ and $\tau=\tau_\mathbf{b}$ as  in Lemma \ref{lem:taub}. Let $X^{\mathbf{b},x_0}_\mu(\tau)$ be the closure of $Y^\mathbf{b}_\nu \cap \Gr_\nu$. Then $X^{\mathbf{b},x_0}_\mu(\tau)$ is geometrically irreducible of dimension $\langle\rho,\mu-\tau\rangle$.
\end{theorem}
The following result is \cite[4.4.14]{XZ17}.
\begin{theorem}
  Let $b=\omega^\tau$ be an unramified element, and assume that $\tau_\sigma$ is dominant.  Consider the action of $J_{p^\tau}(F)$ on the set of irreducible components of $X_\mu(p^\tau)$.
  \begin{enumerate}
    \item The stabilizer of each irreducible component is a hyperspecial subgroup.
    \item The subset $X_\mu^\mathbf{b}(p^\tau)$ is invariant under the action of $J_{p^\tau}$, and the group $J_{p^\tau}$ acts transitively on the set of irreducible components of $X_\mu^\mathbf{b}(p^\tau)$.
    \item Assume that $Z_G$ is connected. There is a canonical $J_{p^\tau}(F)$-equivariant bijection between the set of irreducible components of $X_\mu(p^\tau)$ and
    \begin{equation*}
      \bigsqcup_{\lambda\in \tau+(1-\sigma)X_*(T)}\mathbb{MV}_\mu(\lambda)\times \mathbb{HS}_b,\qquad(b,x)\mapsto X^{\mathbf{b},x}_\mu(\tau)
    \end{equation*}
    where $\mathbb{HS}_b$ denote the set of hyperspecial subgroups of $J_\tau(F)$, and $X^{\mathbf{b},x}_\mu(\tau)$ is the irreducible component in $X^\mathbf{b}_\mu(\tau)$ whose stabilizer is the hyperspecial subgroup  of $J_\tau(F)$ corresponding to $x$.
  \end{enumerate}
\end{theorem}
\begin{proposition}
\label{prop:sigmafix}
  Assume that $Z_G$ is trivial. Let $\mu$ be minuscule, and $p^\tau$ an unramified element such that $\tau_\sigma$ is dominant. Fix $\mathbf{b}\in\mathbb{MV}_\mu(\lambda)$ such that $\tau_\sigma=\lambda_\sigma$. We further assume $\tau=\tau_\mathbf{b}$. Let $n$ be the minimal positive integer such that $\sigma^n(\lambda)=\lambda$. Then $\sigma^n$ fixes the irreducible components of $X_\mu^{\mathbf{b}}(p^\tau)$, i.e. $g,\sigma^n(g)$ lie in the same irreducible component of $X_\mu^{\mathbf{b}}(p^\tau)$.
\end{proposition}
We first need a lemma about the group $J_{p^\tau}$.
\begin{lemma}
  With the assumptions above, the set $J_{p^{\tau_{\mathbf{b}}}}(\mathbb{Q}_p)\subset G(L)$ is fixed by $\sigma^n$.
\end{lemma}
\begin{proof}
We observe that since $\tau_\mathbf{b}=\lambda+\nu_b-\sigma(\nu_b)$, $J_{p^{\tau_{\mathbf{b}}}}$ is the conjugate of $J_{p^\lambda}$ by $p^{\nu_\mathbf{b}}$. Since $Z_G$ is trivial, by Lemma \ref{lem:nub}, we see that $\sigma^n(\nu_\mathbf{b})=\nu_\mathbf{b}$, so it remains to check that the points $J_{p^\lambda}(\bbQ_p)\subset G(L)$ are fixed by $\sigma^n$. Then $J_{p^\lambda}$ is an inner form of $M_{\bar{\lambda}}$, the Levi subgroup which centralizes $\bar{\lambda}$, and since $p^\lambda$ is unramified, we in fact have $J_{p^\lambda}=M_{\bar{\lambda}}$, so the points $J_{p^\lambda}(\bbQ_p)$ are in fact defined over $\bbQ_p$, and hence fixed by $\sigma^n$.
\end{proof}
We now prove Proposition \ref{prop:sigmafix}. 
\begin{proof}
Since $J_{p^\tau}(\bbQ_p)$ acts transitively on the set of irreducible components of $X_\mu^{\mathbf{b}}(p^\tau)$, choose some $h\in J_{p^\tau}(\bbQ_p)$ such that $hg\in X^{\mathbf{b},x_0}_\mu(p^\tau)$. Since $X^{\mathbf{b},x_0}_\mu(b)$ is the closure of $Y^\mathbf{b}_{\nu_\mathbf{b}} \cap \Gr_{\nu_\mathbf{b}}$, we see that $\sigma^n(hg)$ lies in the closure of $Y^\mathbf{b}_{\sigma^n(\nu_\mathbf{b})} \cap \Gr_{\sigma^n(\nu_\mathbf{b})}$. Since $\nu_\mathbf{b}$ is fixed by $\sigma^n$, then $\sigma^n(hg)=h\sigma^n(g)$ lies in $X^{\mathbf{b},x_0}_\mu(b)$ as well. Hence $g$, $\sigma^m(g)$ lie in the same irreducible component of $X^{\mathbf{b}}_\mu(p^\tau)$.
\end{proof}
\subsection{Reduction of isogenies}
We now assume that we have a $p$-divisible group $\sG/k$ such that the Hodge filtration is given by $\mu_0^{-1}$. Let $K/L$ be a finite extension and fix a Galois closure $\bar{K}$ of $K$, with residue field $\bar{\bbF}_p$. Let $\tilde{\sG}$ be a $G_{O_L}$-adapted lifting of $\sG$ to a $p$-divisible group over $O_K$, with a lift of the slope filtration. By Proposition \ref{prop:ptorsor}, taking $u=0$ there exists an isomorphism
\begin{equation}
\label{eqn:tpD1}
  T_p\tilde{\mathscr{G}}^*(-1)\otimes_{\mathbb{Z}_p} O_L\xlongrightarrow{\sim}\mathbb{D}(\tilde{\mathscr{G}})
\end{equation}
which takes the tensors $s_{\alpha,\acute{e}t}$ to $s_{\alpha,0}$ and the filtration on $T_p\tilde{\mathscr{G}}^*(-1)$ to the graded pieces of the filtration on $\bbD(\widetilde{\sG})$ induced by the slope filtration. We may take $U$ to be $T_p\tilde{\mathscr{G}}^*(-1)$ equipped with the tensors $s_{\alpha,\acute{e}t}$, and the filtration $\calF_F$ on $T_p\tilde{\mathscr{G}}^*(-1)$; up to modifying the isomorphism \eqref{eqn:tpD1}, we can assume that the Frobenius is given by $b\sigma$.

Let $g\in G(\mathbb{Q}_p)$. There is a finite extension $K'/K$ in $\bar{K}$ such that $g^{-1}T_p\tilde{\mathscr{G}}$ is $G_{K'}$-stable, and hence corresponds to a $p$-divisible group $\tilde{\mathscr{G}}'$ over $K'$. Let $\mathscr{G}'= \tilde{\mathscr{G}}'\otimes \overline{\mathbb{F}}_p$. The quasi-isogeny $\theta: \tilde{\mathscr{G}}\rightarrow \tilde{\mathscr{G}}'$ identifies $\mathbb{D}(\tilde{\mathscr{G}}')$ with $g_0\mathbb{D}(\mathscr{G})$ for some $g_0\in GL(\mathbb{D}(\mathscr{G})\otimes_{\mathbb{Z}_p}\mathbb{Q}_p)$.

Moreover, by construction of the isomorphism \eqref{eqn:tpD1}, there is an isomorphism
\begin{equation*}
  T_p\mathscr{G}^*(-1)\otimes_{\mathbb{Z}_p}\mathfrak{S}\xlongrightarrow{\sim} \mathfrak{M}(T_p\mathscr{G}^*),
\end{equation*}
that takes $s_{\alpha,\acute{e}t}$ to $\tilde{s}_\alpha$, induces an isomorphism on the associated graded of the filtrations on both sides, and which, by setting $u=0$, recovers the isomorphism in \eqref{eqn:tpD1}.

Let $\mathfrak{M}':=\mathfrak{M}(T_p\tilde{\mathscr{G}}'^*)$ and $\mathfrak{M}:=\mathfrak{M}(T_p\tilde{\mathscr{G}}^*)$, then the quasi-isogeny $\theta: \tilde{\mathscr{G}}\rightarrow \tilde{\mathscr{G}}'$ induces an identification $\mathfrak{M}(\theta):\mathfrak{M}(T_p\tilde{\mathscr{G}}'^*)[1/p]\xrightarrow{\sim}\mathfrak{M}(T_p\tilde{\mathscr{G}}^*)[1/p]$ so that $\mathfrak{M}'= \tilde{g}\mathfrak{M}$ for some $\tilde{g}\in GL(\mathfrak{M}[1/p])$.

We have the following lemma. 
\begin{lemma}
\label{lemma:g_0onFrakM}
We have $\tilde{g}\in P(\calO_{\widehat{\calE_{ur}}})gG(\calO_{\widehat{\calE_{ur}}})$.
\end{lemma}
\begin{proof}
  Over $\calO_{\widehat{\calE_{ur}}}$ there are canonical identifications:
  \begin{align*}
     T_p\widetilde{\sG}^*(-1)\otimes_{\bbZ_p} \calO_{\widehat{\calE_{ur}}}&\simeq \frakM\otimes_{\frakS} \calO_{\widehat{\calE_{ur}}}\\
     T_p\widetilde{\sG}'^*(-1)\otimes_{\bbZ_p} \calO_{\widehat{\calE_{ur}}}&\simeq \frakM'\otimes_{\frakS} \calO_{\widehat{\calE_{ur}}}
  \end{align*}
The first (respectively the second) taking $s_{\alpha,\acute{e}t}$ to $\tilde{s}_\alpha$ (respectively $s'_{\alpha,\acute{e}t}$ to $\tilde{s}'_\alpha$). Moreover, the first isomorphism also induces an isomorphism of the graded pieces of the filtrations on both sides. Thus, if we identify $T_p \widetilde{\sG}^*(-1)$ with $T_p\widetilde{\sG}'^*(-1)$ via $g$, these isomorphisms differ from the ones above by elements of $P(\calO_{\widehat{\calE_{ur}}})$ and $G(\calO_{\widehat{\calE_{ur}}})$ respectively. Since the map 
\begin{equation*}
    T_p\widetilde{\sG}^*(-1)\otimes\calO_{\widehat{\calE_{ur}}}\xlongrightarrow{\sim}T_p\widetilde{\sG}'^*(-1)\otimes \calO_{\widehat{\calE_{ur}}}\xlongrightarrow{can} \frakM'\otimes\calO_{\widehat{\calE_{ur}}} \xlongrightarrow{\theta}\frakM\otimes \calE_{ur}\xlongrightarrow{can} T_p\widetilde{\sG}^*(-1)\otimes \calE_{ur}
\end{equation*}
is given by $g$, we have $\tilde{g}\in P(\calO_{\widehat{\calE_{ur}}})gG(\calO_{\widehat{\calE_{ur}}})$.
\end{proof}

Combined with \cite[1.2.18]{K2017}, we have the following theorem
\begin{theorem}
\label{thm:Mred}
 We retain the assumptions of $(\ast)$. The association $g\mapsto g_0$ induces a well defined map
 \begin{equation*}
     G(\bbQ_p)/G(\bbZ_p)\rightarrow G(L)/G(O_L).
 \end{equation*}
 For two choices of \ref{eqn:tpD1} the corresponding maps $g\mapsto g_0$ differ by an automorphism of $G(L)/G(O_L)$ given by left multiplication by an element of $M(O_L)$. For any $g$ we have $g_0\in X_{\upsilon_0}(b)$ and $\tilde{\kappa}_G(g_0) =  \tilde{\kappa}_G(g)\in \pi_1(G)^\Gamma$. If moreover $m\in M(\bbQ_p)$, we have $\tilde{\kappa}_M(m_0) =  \tilde{\kappa}_M(m)\in \pi_1(M)^\Gamma$.
\end{theorem}
Let $N$ be the unipotent radical of the parabolic subgroup associated to $\nu([b])^{-1}$. We also observe the following corollary:
\begin{corollary}
\label{cor:um}
  With the same assumptions as in Theorem \ref{thm:Mred}, consider $um$, where $u\in N(\mathbb{Q}_p)$, and $m\in Z_M(\mathbb{Q}_p)$. Then $(um)_0$ is of the form $u_0m$. 
\end{corollary}

\section{Shimura varieties $\Sh_K(G,X)$}
\subsection{Integral Models} We recall the results of Kisin \cite{K2010} about the construction of integral models of Hodge type Shimura varieties with hyperspecial level structure at $p>2$. Let $G$ be a connected algebraic group such that $(G,X)$ is a Shimura datum, with reflex field $E$, and $G$ is unramified over $\mathbb{Q}_p$. For the rest of this section we assume that $(G,X)$ is of Hodge type.

Let $\mathbb{A}^p_f$ be the ring of finite adeles with trivial component at $p$. Let $K^p\supset G(\mathbb{A}^p_f)$ be an open subgroup. Since $(G,X)$ is a Shimura datum of Hodge type we have an embedding of Shimura datum $(G, X)\hookrightarrow (\mathrm{GSp}(V),\mathcal{H}^{\pm}$) that induces an embedding over $E$ of Shimura varieties
\begin{equation*}
\Sh_K(G,X)\hookrightarrow \Sh_{K'}(\mathrm{GSp}(V),\mathcal{H}^{\pm})
\end{equation*}
for $K=K_pK^p$, and $K'=K'_pK'^{p}$, where $K'_p\subset K_p$ is a small enough compact subgroup of $G(\mathbb{Q}_p)$, and $K'_p=GSp(\Lambda)(\mathbb{Z}_p)$, where $\Lambda$ is a $\mathbb{Z}_p$ lattice of $V$. $K'_p$ is a hyperspecial subgroup of $GSp(V)(\mathbb{Q}_p)$.

$\Sh_{K'}(GSp,\mathcal{H}^{\pm})$ is the moduli space of triples $(\mathcal{A},\lambda,\varepsilon_{K'}^p)$, where $\mathcal{A}$ is a  polarized abelian scheme, $\lambda$ is a polarization of $\mathcal{A}$, and $\varepsilon_{K'}^p$ is a section of the \'{e}tale sheaf
\begin{equation*}
\varepsilon_{K'}^p\in \operatorname{Isom}(V_{\mathbb{A}_f^p},\hat{V}^p(\mathcal{A})_{\mathbb{Q}})/K'^{p},
\end{equation*}
where $\hat{V}^p(\mathcal{A}) = \lim_{p\nmid n}\mathcal{A}[n]$ is viewed as an \'{e}tale local system and $\hat{V}^p(\mathcal{A})_\mathbb{Q}=\hat{V}^p(\mathcal{A})\otimes_\mathbb{Z}\mathbb{Q}$.

We have a canonical integral model $\mathscr{S}_{K'}(GSp,\mathcal{H}^{\pm})$ over $\mathbb{Z}_{(p)}$ for $\Sh_{K'}(GSp,\mathcal{H}^{\pm})$ given by extending the moduli interpretation to schemes over $\mathbb{Z}_{(p)}$. By taking the closure of $\Sh_{K'}(G,X)$ in $\mathscr{S}_{K'}(GSp,\mathcal{H}^{\pm})$, we have a model $\widetilde{\mathscr{S}}_{K'}(G,X)$ of $\Sh_K(G,X)$ over $O_{E,(v)}$, equipped with an embedding 
\begin{equation*}
\widetilde{\mathscr{S}}_{K'}(G,X)\hookrightarrow\mathscr{S}_{K'}(GSp,H^{\pm}).
\end{equation*}
Taking the normalization of $\widetilde{\mathscr{S}}_{K'}(G,X)$, this gives us a smooth scheme $\mathscr{S}_{K'}(G,X)$ over $O_{E,(v)}$, which is the canonical model of $\Sh_K(G,X)$. We also define the projective limit
\begin{equation*}
  \mathscr{S}_{K_p}(G,X):=\varprojlim_{K^p}\mathscr{S}_{K^pK_p}(G,X).
\end{equation*}  

Observe that there is a universal abelian scheme $h:\mathcal{A}\rightarrow\mathscr{S}_{K'}(GSp,H^{\pm})$, and by pulling back to $\mathscr{S}_{K}(G,X)$, we have an abelian scheme over $\mathscr{S}_{K}(G,X)$. By abuse of notation, we also let 
\begin{equation*}
  h:\mathcal{A}\rightarrow \mathscr{S}_{K}(G,X)
\end{equation*}
be the universal abelian scheme over $\mathscr{S}_{K}(G,X)$.

Moreover, the constructions in \cite{K2010} also give us \'{e}tale, deRham, and crystalline tensors for points in $\mathscr{S}_{K}(G,X)$. Let us briefly review this construction. Recall that we have an embedding $G_{\mathbb{Z}_p}\hookrightarrow GSp(\Lambda)\subseteq GL(\Lambda)$ of group schemes over $\mathbb{Z}_p$. Let $V_{\mathbb{Z}}$ be the $\mathbb{Z}$-lattice of $V$ such that $V_{\mathbb{Z}}\otimes_{\mathbb{Z}}\mathbb{Z}_p=\Lambda$. Write $V_{\mathbb{Z}_{(p)}}=V_{\mathbb{Z}}\otimes\mathbb{Z}_{(p)}$. Then we have a closed embedding of group schemes over $\mathbb{Z}_{(p)}$
\begin{equation*}
  G_{\mathbb{Z}_{(p)}}\hookrightarrow GSp(V_{\mathbb{Z}_{(p)}}).
\end{equation*} 
Let $(s_\alpha)\subseteq V_{\mathbb{Z}_{(p)}}^\otimes$ be tensors such that $G_{\mathbb{Z}_{(p)}}$ is the stabilizer of the tensors $(s_\alpha)$, where for an $R$-module $M$, we write $M^\otimes$ for the direct sum of all tensor products, duals, symmetric and exterior powers of $M$.

Let $\mathcal{V}_{dR}=R^1h_{*}\Omega^\bullet$ where $\Omega^\bullet$ is the sheaf of differentials on $\mathscr{A}$. By the de Rham isomorphism, we can view $(s_\alpha)$ as sections $(s_{\alpha,dR})$ of $\mathcal{V}_{dR}$ defined over $\mathbb{C}$. In fact, $(s_{\alpha,dR})$ are horizontal sections of $\mathcal{V}_{dR}^\otimes$ defined over $O_{E,(v)}$, which lie in the $\operatorname{Fil}^0$ part of the Hodge filtration.

Similarly, for $l\neq p$, let $\mathcal{V}_{l}=R^1h_{\acute{e}t,*}\mathbb{Q}_l$, and let $\mathcal{V}_{p}=R^1h_{\eta,\acute{e}t,*}\mathbb{Q}_p$, where $h_\eta$ denotes the generic fibre of $h$. These are \'{e}tale local systems over $\mathscr{S}_{K}(G,X)$ and $\Sh_K(G,X)$ respectively. By the \'{e}tale-Betti comparison isomorphism, we have sections $(s_{\alpha,l})$ of $\mathcal{V}_{l}$ defined over $\mathbb{C}$, which are in fact defined over $E$ for $l=p$, and $O_{E,(v)}$ for $l\neq p$. 

If $T$ is an $O_{E,(v)}$-scheme (resp. an $E$-scheme), $x\in \mathscr{S}_K(G,X)(T)$, and $\ast=l$ or $dR$ (resp. $\ast=p$), we denote by $(s_{\alpha,\ast,x})$ the pullback of $(s_{\alpha,\ast})$ to $T$. Similarly, we denote by $\mathcal{A}_x$ the pullback of $\mathcal{A}$ to $x$. Furthermore, note that for $l\neq p$, since the $l$-adic rational Tate module $V_l(\mathcal{A}_x)$ is dual (as a $\mathbb{Q}_p$ vector space) to $H^1_{\acute{e}t}(\mathcal{A}_x,\mathbb{Q}_l)$, we can view the sections $(s_{\alpha,l,x})$ as sections of $V_l(\mathcal{A}_x)$, which we also denote by $(s_{\alpha,l,x})$. In this way, we have a section $(s_{\alpha,l,x})_{l\neq p}$ of the \'{e}tale sheaf $\hat{V}^p(\mathcal{A}_x)_{\mathbb{Q}}$, and if $x$ corresponds to the triple $(\mathcal{A},\lambda,\varepsilon^p_{K'})$, $\varepsilon^p_{K'}$ is in fact a section
\begin{equation*}
  \varepsilon^p_K\in \underline{\operatorname{Isom}}((V_{\mathbb{A}^p_f},\hat{V}^p(\mathcal{A})_{\mathbb{Q}})/K^{p}),
\end{equation*}
which sends $(s_\alpha)$ to $(s_{\alpha,l,x})_{l\neq p}$.

Let $\mathscr{S}_\kappa$ denote the special fiber of $\mathscr{S}_{K}(G,X)$. Furthermore, suppose that $\tilde{x}$ is a geometric point of the special fiber of $\mathscr{S}_\kappa$ with residue field $k$. Consider the crystalline cohomology of the abelian variety $\mathcal{A}_{\tilde{x}}$, $H^1_{cris}(\mathcal{A}_{\tilde{x}}/W(k))$. Note that this is canonically isomorphic to the contravariant Dieudonn\'{e} module $\mathbb{D}(\mathcal{A}_{\tilde{x}}[p^\infty])$ of the $p$-divisible group $\mathcal{A}_{\tilde{x}}[p^\infty]$. Then via the $p$-adic comparison isomorphism, we can construct crystalline tensors $(s_{\alpha,0,\tilde{x}})\subset H^1_{\mathrm{cris}}(\mathcal{A}_{\tilde{x}}/W(k))^\otimes$.  

$H^1_{\mathrm{cris}}(\mathcal{A}_{\tilde{x}}/W(k))$ is an example of an $F$-crystal with $G$-structure over $\operatorname{Spec}\kappa$. 

For any geometric point $x\in \mathscr{S}_\kappa(\bar{\bbF}_p)$, while the $p$-divisible group $\mathcal{A}_x[p^\infty]$ and the crystalline tensors $(s_{\alpha,0,x})$ depends on some choices made during the construction of $\mathscr{S}_K(G,X)$, it induces an $F$-crystal with $G$-structure $\mathbb{D}(\mathcal{G}_x)$ over $\bar{\mathbb{F}}_p$ which is independent of them. Moreover, \cite{lov2017} (c.f. \cite[Appendix A]{KMPS}) constructed a universal $F$-isocrystal with $G$-structure over $\mathscr{S}_K(G,X)$, which we denote by $\mathbb{D}_0$. $\mathbb{D}_0$ specialises to $\mathbb{D}(\sG_x)$ for every $x\in\sS_K(G,X)(\overline{\mathbb{F}}_p)$. Moreover, we also have a section $(s_{\alpha,0})\subset\mathbb{D}_0^\otimes$ defined over $\kappa$, such that the pull-back to any geometric point $x$ is the crystalline tensor $(s_{\alpha,0,x})$ constructed above. If $T$ is a scheme over $\operatorname{Spec}\kappa$, for any $x\in\sS_\kappa$, we denote by $(s_{\alpha,0,x})$ the pullback of $(s_{\alpha,0})$ to $T$.

From the arguments in \cite{K2010}, we have the following proposition
\begin{proposition}
  Let $x\in\sS_\kappa(k)$ be a geometric point of $\sS_\kappa$. By construction, we have tensors $(s_{x,\alpha,0})\subset\bbD(\sG_x)(W(k))^\otimes$. Then the completion $\hat{U}_x$ of $\sS_\kappa$ at $x$ is isomorphic to the deformation space $\operatorname{Spf}R_G$ of $G_W$-adapted lifts of $\sG_x$. 
\end{proposition}
\subsubsection{}
For any point $x\in \sS_{\kappa}(\overline{\mathbb{F}}_p)$, let $b_x\in G(L)$ be such that the Frobenius on $\bbD(\sG_x)$ is given by $\varphi=b_x\sigma$. $b_x$ is determined up to $\sigma$-conjugacy by $G(O_L)$, we simply choose a representative here. 

We can in fact describe the elements $[b]$ of $B(G)$ which arise from points on the Shimura variety. Recall from Section \ref{subsection:hecke-poly-def} that attached to the Shimura datum $(G,X)$ we have a dominant minuscule cocharacter $\mu$. If we choose $T\subset G_{\bbZ_p}$, then $\mu$ is also $G_W$-valued. We set $\upsilon=\sigma(\mu^{-1})$. Note that, for any $G_W$-valued cocharacter $\mu_0$ such that $\mu_0^{-1}$ defines the Hodge filtration on $\bbD(\sG_x)$, $\mu,\mu_0$ must be conjugate. Since $\mu,\mu_0$ are $G_W$-valued they are conjugate by an element of $G(W)$. Hence, if we set $\upsilon_0=\sigma(\mu_0^{-1})$, we have
\begin{equation*}
    b\in G(W)p^{\upsilon_0} G(W)=G(W)p^\upsilon G(W),
\end{equation*}
hence $[b]\in B(G,\upsilon)$.

\subsubsection{}
  For any $[b]\in B(G,\upsilon)$, we define the Newton Stratum $\mathscr{S}_\kappa^{[b]}$ as
  \begin{equation*}
    \mathscr{S}_\kappa^{[b]}=\{x\in \mathscr{S}_{K}(\overline{\mathbb{F}}_p):[b_x]=[b]\}.
  \end{equation*}
This is a locally closed subscheme of $\mathscr{S}_{\kappa}$ with reduced subscheme structure.

We also have the following result of Lee \cite{lee2018} about the non-emptiness of each Newton strata, which was generalized in \cite{KMPS} without the condition that $G$ is unramified.
\begin{theorem}
  Suppose that $b\in B(G,\upsilon)$. Then $\mathscr{S}_\kappa^{[b]}$is non-empty.
\end{theorem}

Note that there is a unique $[b]\in B(G,\upsilon)$ such that $\nu([b])=\bar{\upsilon}$, which we refer to as the $\mu$-ordinary element $[b^\ord]$. This corresponds to a maximal Newton stratum $\mathscr{S}_\kappa^{[b^\ord]}$. We denote this by $\mathscr{S}^{\ord}_\kappa$, and refer to this as the $\mu$-ordinary locus of $\mathscr{S}_\kappa$, and points as $\mu$-ordinary points. Let us state a few properties of the Newton stratification.

\begin{theorem}[\cite{Ham2017},\cite{Zha2015}]
  \label{thm:dimensionnewtonstrata}
  Let $[b]\in B(G,\upsilon)$. Then
  \begin{enumerate}
    \item $\mathscr{S}^{[b]}_\kappa$ is non-empty and of pure dimension
    \begin{equation*}
      \langle \rho,\mu+\nu([b])\rangle-\frac{1}{2}\operatorname{def}(b),
    \end{equation*} where $\rho$ denotes the halfsum of positive roots of $G$, $\nu(b)$ denotes the Newton cocharacter of $[b]$ and $\operatorname{def}(b)$ denotes the defect of $b$.
    \item $\overline{\mathscr{S}^{[b]}_\kappa}=\bigcup_{[b']\preceq [b]}\mathscr{S}^{[b']}_\kappa$
  \end{enumerate}
\end{theorem}

\subsection{Isogeny classes mod $p$}
For any $g\in X_\mu(b)$, we see that $g\cdot \bbD(\sG_x)$ is a Dieudonn\'{e} module. Hence $g\cdot \bbD(\sG_x)$ corresponds to a $p$-divisible group $\sG_{gx}$ which is naturally equipped with a quasi-isogeny $\sG_x\rightarrow \sG_{gx}$, corresponding to the isomorphism $g\cdot \bbD(\sG_x)\otimes\bbQ_p\xrightarrow{\sim}\bbD(\sG_x)\otimes \bbQ_p$. Note that since $g$ preserves the tensors $(s_{\alpha,0,x})$, we have tensors $(s_{\alpha,0,gx}):=(s_{\alpha,0,x})\subset \bbD(\sG_{gx})^\otimes$.

We denote by $\calA_{gx}$ the corresponding abelian variety, which is isogenous to $\calA_x$, and canonically equipped with a $K'$-level structure, induced by that on $\calA_x$. Moreover, the weak polarization on $\calA_x$ induces a weak polarization $\lambda_{gx}$ on $\calA_{gx}$. The following result  \cite[1.4.4]{K2017} shows that $(\calA_{gx},\lambda_{gx})$ induces a point on $\sS_K(G,X)$.

\begin{proposition}
  There is a unique map
  \begin{equation*}
  \label{eq:iotadef}
      \iota_x:X_\upsilon(b)\rightarrow \sS_K(G,X)(\bar{\bbF}_p)
  \end{equation*}
  equivariant for the action of the $r$-th power of the Frobenius, such that $(s_{\alpha,0,x})=(s_{\alpha,0,\iota_x(g)})\subset \bbD(\sG_{gx})^\otimes$.
\end{proposition}
\subsubsection{}
The map $\iota_x$ in \eqref{eq:iotadef} extends to a $G(\mathbb{A}_p^f)$-equivariant map
\begin{equation*}
    \iota_x:X_\upsilon(b)\times G(\mathbb{A}_p^f)\rightarrow \sS_{K}(G,X)(\bar{\bbF}_p).
\end{equation*}
We call the image of $\iota_x$ in \eqref{eq:iotadef}, the isogeny class of $x$.

Fix an embedding of $\bar{\bbQ}$-algebras $\bar{L}\hookrightarrow\mathbb{C}$. Let $x\in \sS_{K}(G,X)(\overline{\bbF}_p)$, and $\tilde{x}\in \sS_{K}(G,X)(O_{\bar{L}})$ a point lifting $x$. By \cite[2.2.6]{K2010} there is an element $(h,1,g^p_{\tilde{x}})\in X\times G(\bbQ_p)\times G(\mathbb{A}_p^f)$ which maps to $\tilde{x}\in \Sh_K(G,X)(\mathbb{C})$. Attached to $(h,1,g^p_{\tilde{x}})$, we have an isomorphism $H_1(\calA_{\tilde{x}}(\mathbb{C}), \bbZ_{(p)})\xrightarrow{\sim}V_{\bbZ_{(p)}}$, and hence an isomorphism $T_p\sG_{\tilde{x}}\xrightarrow{\sim} V_{\bbZ_p}$ which takes $s_\alpha$ to $s_{\alpha,p,\tilde{x}}$. Then we can apply the construction of \ref{thm:Mred}, and obtain a commutative diagram
\begin{equation*}
\begin{tikzcd}
{h\times G({\bbQ}_p)\times G(\mathbb{A}_p^f)} \arrow[d, "g_p\mapsto g_{p,0}"] \arrow[r] & {\sS_{K}(G,X)(O_{\bar{L}})} \arrow[d] \\
{X_\upsilon(b)\times G(\mathbb{A}_p^f)} \arrow[r, "\iota_x"]       & {\sS_{K}(G,X)(\bar{\bbF}_p)}          
\end{tikzcd}
\end{equation*}
\subsection{Special point liftings}
In this subsection, we will recall results from \cite[\S2]{K2017} about liftings to special points up to isogeny. Choose $x\in \sS_k(G,X)(k)$ for some finite field $k\supset\kappa$. Then $x$ defines an element $b:=b_x\in B(G,\upsilon)$. Recall that $\nu_b$ is central in $J_b$.

\begin{theorem}
    Let $T\subset J_b$ be a maximal torus (defined over $\mathbb{Q}_p$). Then there exists a cocharacter $\mu_T\in X_*(T)$ defined over $\bar{\bbQ}_p$ such that
    \begin{enumerate}
        \item As a $G_{K_0}$-valued cocharacter, $\mu_T$ is conjugate to $\mu$
        \item $\bar{\mu}_T^T=\nu^{-1}_b$.
    \end{enumerate}
\end{theorem}
\begin{proof}
  This result is contained in \cite[\S5]{LR} (c.f. \cite[2.2.2]{K2017}). We sketch the proof here. Let $K_0$ be the field of definition of $b$, and let $T'\subset T$ be the maximal $\bbQ_p$-split sub-torus, so that the centralizer of $T'$ is  $M_{K_0}\subset G_{K_0}$. Up to $\sigma$-conjugation by an element of $G(K_0)$, we may assume that $b\in M_{K_0}(K_0)$. Then $T'$ may also be viewed as a subgroup of $G_{\bbQ_p}$, so $M_{K_0}$ is defined over $\bbQ_p$. Moreover, $M$ contains the centralizer of a maximal split torus $T_2$ of $G_{\bbQ_p}$. 
  
  Let $P$ be a parabolic subgroup of $G$ with unipotent radical $N$, such that $P=MN$. Let $g\in X_\upsilon(b)$. By the Iwasawa decomposition, we may assume that $g=nm$ with $n\in N(L), m\in M(L)$. Then $g^{-1}b\sigma(g)=m^{-1}b\sigma(m)n'$ for some $n'\in N(L)$. Let $\upsilon'\in X_*(T_2)$ with $m^{-1}b\sigma(m)\in (M(L)\cap G(O_L))p^{\upsilon'}(M(L)\cap G(O_L))$. Then $\upsilon'$ is conjugate to $\upsilon=\sigma(\mu^{-1})$ in $G$. We let $\mu_T\in X_*(T)$ be a cocharacter which is conjugate to $\sigma^{-1}(\upsilon'^{-1})$ in $M$. One checks that $\mu_T$ satisfies (1) and (2).
\end{proof}
\begin{remark}
\label{rk:nonemptyI}
  The latter half of this argument can be used to show that the set $I_{b,\mu}$ (defined in \ref{section:reductionLevi}) is non-empty, by taking $M$ to be the standard Levi, and $T_2$ to be the torus $T$ defined there.  
\end{remark}
\begin{remark}
  If $b$ is unramified, we can take $J_b$ to be $M_{[b]}$, the torus $T$ to be the fixed one we referred to in \ref{section:unramifiedclasses}, and $\mu_T$ is $\sigma^{-1}(\tau^{-1})$ for the representative $p^\tau$ such that $\tau_\sigma\in X_*(T)^+_\sigma$. 
\end{remark}
\subsubsection{}
$\mu^{-1}_T$ induces an admissible filtration on $\bbD(\sG_x)_K$, where $K$ is the field of definition of $\mu_T$, thought of as a $G_{K_0}$ character. From \cite[1.1.19]{K2017} there is some $p$-divisible group $\sG'/k$ such that $\bbD(\sG')=g\bbD(\sG)$ for some $g\in X_{\upsilon}(b)$ and there is an $O_K$ lift of $\sG'$ with filtration given by $\mu^{-1}_T$. We then replace $x$ with $\iota_x(g)$, and take $\tilde{x}\in \Sh_K(G,X)(K)$ lifting $x$ such that the filtration on $\bbD(\sG_{\tilde{x}})_K$ is given by $\mu^{-1}_T$. Up to modifying the choice of torus $T$, it turns out that $\tilde{x}$ is a special point. We thus have the following (c.f. \cite[2.2.3]{K2017} for more details).
\begin{theorem}
  The isogeny class $\iota_x(X_\upsilon(b)\times G(\mathbb{A}^p_f))$ contains a point which is the reduction of a special point on $\Sh_K(G,X)$.
\end{theorem}
We now consider the case there the point $x$ is $\mu$-ordinary. Let us first recall the results of \cite{Wor2013} and \cite{SZ2016} about canonical lifts of $\mu$-ordinary points.

The following result is a reformulation of \cite[Prop 7.2]{Wor2013}.
\begin{lemma}
  Let $[b_x]$ be $\mu$-ordinary. Then up to $\sigma$-conjugacy by elements in $G(O_L)$, we have $b_x=p^\upsilon$.
\end{lemma}
Using this, we can consider the special point lifting $\widetilde{\sG}_x$ of $\sG_x$ obtained from the above procedure such that $\mu^{-1}_T$ is $\sigma^{-1}(\upsilon)$, and we have the following theorem \cite[Thm 3.5]{SZ2016}.
  
\begin{theorem}
\label{thm:muordlifting}
  Let $\mathscr{G}_x$ be the $p$-divisible group associated to some $x\in \mathscr{S}_\kappa(\bar{\mathbb{F}}_p)$ which is $\mu$-ordinary. Then $\widetilde{\mathscr{G}}_x$ is the unique $G_{O_L}$-adapted lift of $\sG_x$ to $O_L$, such that the action of $J_b(\mathbb{Q}_p)$ on $\mathscr{G}_x$ lifts (in the isogeny category) to $\widetilde{\mathscr{G}}$. Moreover $\tilde{x}$ is a special point.
\end{theorem}
The lift $\tilde{x}$ is known as the canonical lift of $x$.

Let $M$ be the standard Levi subgroup given as the centralizer of $\tilde{\upsilon}$. Note that in the above theorem $\widetilde{\mathscr{G}}_x$ is also, by construction, the unique $M_{O_L}$-adapted lift of $\mathscr{G}$. Moreover, we have $J_b=M$, and an isomorphism
\begin{equation*}
    X^G_\upsilon(b)\simeq X^M_\upsilon(b)=M(\bbQ_p)/M(\bbZ_p)
\end{equation*}
of perfect schemes of dimension 0. In particular, we see that all isogenies of $\mathscr{G}_x$ lifts to isogenies of $\widetilde{\mathscr{G}}_x$.

\section{The moduli space $p-\operatorname{Isog}$}
\subsection{Preliminaries}
Let $T$ be a scheme over $O_{E,(v)}$, and consider any two points $x,y$ lying in $\mathscr{S}_K(G,S)(T)$. For any geometric point $t$ of $T$, let $x_t,y_t$ be the pullback of $x,y$ to $t$. As described above, we have $l$-adic \'{e}tale and de Rham tensors $(s_{x_{t},\alpha,l})$ for $l\neq p$, and $(s_{x_{t},\alpha,dR})$ (respectively $(s_{y_{t},\alpha,l}),(s_{y_{t},\alpha,dR})$ for $y_{t}$). Observe that $k(t)$, the residue field at $t$, could be of either characteristic $0$ or characteristic $p$. Suppose $k(t)$ is a field of characteristic $0$, i.e. it is an extension of $E$. Then, we also have $p$-adic \'{e}tale tensors $(s_{x_{t},\alpha,p}),(s_{y_{t},\alpha,p})$. Otherwise, if $k(t)$ is of characteristic $p$, it is an extension of $\kappa$. Similarly, we have crystalline tensors $(s_{x_{t},\alpha,0}),(s_{y_{t},\alpha,0})$.

We define a quasi-isogeny between $x,y$ to be a quasi-isogeny $f:\mathcal{A}_{x}\rightarrow \mathcal{A}_{y}$ of abelian schemes over $T$, such that for any geometric point $t$, the induced quasi-isogeny $f_t:\mathcal{A}_{x_t}\rightarrow {A}_{y_t}$ of abelian varieties over $k(t)$ preserves all the tensors described above. 

We define a $p$-quasi-isogeny between $x,y$ to be a quasi-isogeny as defined above, such that the isomorphism on the rational prime-to-$p$ Tate modules $f:\hat{V}^p(\mathcal{A}_x)_\mathbb{Q}\xrightarrow{\sim}\hat{V}^p(\mathcal{A}_y)_\mathbb{Q}$, induced by the quasi-isogeny $\mathcal{A}_x\rightarrow\mathcal{A}_y$, respects the prime to $p$-level structures $\varepsilon^p_x,\varepsilon^p_y$, i.e. $\varepsilon^p_y$ is given by the composition
\begin{equation*}
    V_{\mathbb{A}_p^f}\xrightarrow[\varepsilon^p_x]{\sim}\hat{V}^p(\mathcal{A}_x)_\mathbb{Q}\xrightarrow[f]{\sim}\hat{V}^p(\mathcal{A}_y)_\mathbb{Q}.
\end{equation*}
In particular, we see that the weak polarizations on $\calA_x,\calA_y$ differ by some power of $p$.

\subsubsection{} Consider the \emph{fppf}-sheaf of groupoids $p-\operatorname{Isog}$ of $p$-quasi-isogenies between points on $\mathscr{S}_{K_p}(G,X)$. Concretely, for any $O_{E,(v)}$-scheme $T$, points of $p-\operatorname{Isog}$ are pairs $(x,f)$, where $x\in\mathscr{S}_{K_p}(G,X)(T)$, and $f$ is a $p$-quasi-isogeny $f:\calA_x\rightarrow \calA_y$, where $y\in \mathscr{S}_{K_p}(G,X)(T)$. We have canonical 1-morphisms sending a $p$-quasi isogeny $(x,f)$ to $x$ (respectively $y$)
\begin{equation*}
	s:p-\operatorname{Isog}\rightarrow \mathscr{S}_{K_p}(G,X)\qquad t:p-\operatorname{Isog}\rightarrow\mathscr{S}_{K_p}(G,X).
\end{equation*}
Denote by $\underline{\mathbb{Z}}$ the constant sheaf associated to $\mathbb{Z}$ on $Sch/O_{E,v}$, which we can consider as a rigid stack. To any pair $(x,f)$ we can associate a multiplicator $d\in\mathbb{Z}$, defined as follows. Let $m$ be a positive integer such that $p^mf$ is an actual isogeny from $\calA_x$ to $\calA_y$, then the degree of $f$ is $d-m$, where $d$ is the degree of $p^mf$. The map which takes a $p$-quasi-isogeny to its multiplicator defines a 1-morphism $p-\Isog \rightarrow \underline{\mathbb{Z}}$. For every global section $c$ of $\underline{\mathbb{Z}}$ we define $p-\Isog^c$ to be its inverse image, i.e. $p$-quasi-isogenies with multiplicator $c$. Thus, $p-\Isog$ is the disjoint union of $p-\Isog^c$. We have the following well-known proposition:
\begin{proposition}
    \label{prop:st}
  The 1-morphisms $s,t: p-\Isog^c\rightarrow \mathscr{S}_{K_p}(G,X)$ are representable by proper surjective morphisms. Moreover, $p-\Isog^c$ is a relative scheme over $\mathscr{S}$, and hence $p-\Isog$ is an algebraic stack. Furthermore, the morphisms are finite \'{e}tale in characteristic zero.
\end{proposition}
Note that we can also impose level structures on $p-\operatorname{Isog}$. For $K^p\subseteq G(\mathbb{A}_p^f)$, we can define $p-\operatorname{Isog}_{K^p}$ in a similar way, by setting $K=K^pK_p$ and considering $p$-quasi-isogenies between points on $\mathscr{S}_{K}(G,X)$ instead. For small enough $K^p$ such that $\mathscr{S}_{K}(G,X)$ is a scheme, the arguments above show that $p-\operatorname{Isog}_{K^p}$ is in fact also a scheme over $O_{E,(v)}$. In the following, we always assume sufficient level structure $K^p$ such that $p-\operatorname{Isog}_{K^p}$ is a scheme, and for notational simplicity we will simply denote this by $p-\operatorname{Isog}$. 

\begin{lemma}
    $p-\Isog\otimes E$ is equidimensional of dimension $2\langle\rho,\mu\rangle$.
\end{lemma}
\begin{proof}
    This follows from Proposition \ref{prop:st}, and the fact that dimension $\Sh_K(G,X)=2\langle\rho,\mu\rangle$.
\end{proof}

Note that composition of isogenies induces a morphism
\begin{equation}
  \label{eqn:compositionisogeny}
  p-\operatorname{Isog}\times_{\mathscr{S}_{K}(G,X)}p-\operatorname{Isog}\xlongrightarrow{c} p-\operatorname{Isog}
\end{equation}
that maps any pair of $T$-valued isogenies $f_1:\mathcal{A}_1\rightarrow\mathcal{A}_2$ and $f_2:\mathcal{A}_2\rightarrow\mathcal{A}_3$ to $f_2\circ f_1$.

For any $O_{E,(v)}$ scheme $S$, let $p-\Isog\otimes S$ denote the base change $p-\Isog\times_{O_{E,(v)}}S$. We denote $\mathbb{Q}[p-\Isog\otimes S]$ by the $\mathbb{Q}$-vector space of irreducible components of $p-\Isog\otimes S$. If $S$ is a scheme over $E$, \eqref{eqn:compositionisogeny} gives us a multiplicative structure on $\mathbb{Q}[p-\Isog\otimes S]$. Concretely, let $C,D$ be irreducible components $p-\Isog\otimes S$, denote by $C\cdot D$ the image of $C\times_{\mathscr{S}_{K}(G,X)} D$ under the map $c$, multiplied by the degree of $c$. Then since $s$ is finite \'{e}tale in characteristic zero, $\dim C\cdot D=\dim C$, and hence $C\cdot D$ defines an element of $\mathbb{Q}[p-\Isog\otimes S]$. 

\subsubsection{}
Consider the closure $\mathscr{J}$ of the generic fiber $p-\Isog\otimes E$ in $p-\Isog$. We abuse notation and still denote the special fiber of $\mathscr{J}$ by $p-\Isog\otimes\kappa$, and the $\bbQ$-vector space of irreducible components by $\bbQ[p-\Isog\otimes\kappa]$. Since $\mathscr{J}$ is flat over $O_{E,(v)}$, Irreducible components of $p-\Isog\otimes\kappa$ are hence of dimension $2\langle\rho,\mu\rangle$. We still have proper maps $s,t:p-\Isog\otimes\kappa\rightarrow \sS_\kappa$.

As described above, we have a stratification of $\mathscr{S}_\kappa$ by elements $[b]\in B(G,\upsilon)$. Since $[b]$ is invariant under quasi-isogeny, we can similarly define a stratification of $p-\Isog\otimes\kappa$ as follows.
\begin{equation*}
    p-\Isog^{[b]}\otimes\kappa:=\{x\in p-\Isog\otimes\kappa(\overline{\mathbb{F}}_p):[b_{s(x)}]=[b]\}
\end{equation*}
with reduced subscheme structure. In particular, $p-\Isog^{[b]}\otimes\kappa=s^{-1}(\mathscr{S}^{[b]}_\kappa)$. When $[b]$ is $\mu$-ordinary, we also denote the associated subscheme by $p-\Isog^{\ord}\otimes \kappa$. Observe that the closure relations on $\sS_\kappa$ induce closure relations on $p-\Isog\otimes\kappa$. In particular, we see that
\begin{equation*}
    \overline{p-\Isog^{[b]}\otimes\kappa}\subset \bigcup_{[b']\preceq [b]} p-\Isog^{[b']}\otimes\kappa.
\end{equation*}For any $C\in \bbQ[p-\Isog\otimes\kappa]$, we define  
\begin{equation*}
    C^{[b]} = p-\Isog^{[b]}\otimes\kappa \cap C.
\end{equation*}
\subsubsection{}
Let 
  \begin{equation*}
  	f:\mathcal{A}_1\rightarrow \mathcal{A}_2
  \end{equation*} 
be the $p$-quasi-isogeny corresponding to an $\overline{\mathbb{Q}}$-valued point in $p-\Isog\otimes E$. Choose isomorphisms $\alpha_i:\Lambda\simeq T_p(\mathcal{A}_i)$, $i=0,1$. Then $\alpha_2^{-1}\circ V_p(f)\circ \alpha_1: V_{\mathbb{Q}_p}\rightarrow V_{\mathbb{Q}_p}$ is an element of $G(\mathbb{Q}_p)$. Its class $\tau(f)$ in $K_p\backslash G(\mathbb{Q}_p)/K_p$ is independent of the choices involved. Moreover, because we can locally trivialize the \'{e}tale local systems $T_p(\calA_i)$, the function $\tau$ is locally constant, and hence gives a well-defined map on irreducible components of $p-\operatorname{Isog}\otimes E$.

Thus, we can define a $\mathbb{Q}$-algebra homomorphism
  \begin{equation*}
 h_G : \calH(G(\mathbb{Q}_p)//K_p) \rightarrow \mathbb{Q}[p-\Isog \otimes E]
  \end{equation*}
as the map which takes $1_{K_pgK_p}$ to the formal sum of
irreducible components $C\subset p-\operatorname{Isog} \otimes E$ such that $\tau(C) = K_pg^{-1}K_p$.
\begin{remark}
  Note that we map $K_pg^{-1}K_p$ rather than $K_pgK_p$ to this is a correction of \cite{W2000}, following the remarks of \cite[A4]{Nek2018}. 
\end{remark}
\subsubsection{}Given an irreducible component $C\subset p-\Isog\otimes E$, we can consider the closure $\mathscr{C}\subset p-\Isog$, and let $\cup_i C_{p,i}$ be the special fiber of $\mathscr{C}$, where each $C_{p,i}$ is an irreducible component. We thus define a specialization of cycles map $S:\bbQ[p-\Isog\otimes E]\rightarrow\bbQ[p-\Isog\otimes\kappa]$, such that $S(C)=\sum_{i}C_{p,i}$. We let $h$ denote the composition
\begin{equation*}
    h:\calH(G(\mathbb{Q}_p)//K_p) \rightarrow \mathbb{Q}[p-\Isog \otimes E]\xrightarrow{S}\bbQ[p-\Isog\otimes\kappa].
\end{equation*}
This is an algebra homomorphism, with the multiplicative structure on $\bbQ[p-\Isog\otimes\kappa]$ defined as in Appendix \ref{app:b}.
\subsubsection{}
\label{subsub:commutative}
Since the Hecke algebra $\calH(G(\mathbb{Q}_p)//K_p)$ is commutative, it follows that its image under $h$ is a commutative subalgebra of $\mathbb{Q}[p-\Isog\otimes\kappa]$. There is another element which lies in the center, namely the Frobenius section $\Frob$. This was shown by Koskivirta (c.f.\cite[Prop. 25]{Kos2014}). In particular, this means that when multiplying the terms in $H_{G,X}(\Frob)$, we may freely multiply terms without worrying about the order of multiplication.

\subsubsection{}
\label{section:constructionbarh}
We now consider the case where $[b]$ is the $\mu$-ordinary $\sigma$-conjugacy class. In this situation, we will show that $p-\Isog^{M,[b]}\otimes\kappa=p-\Isog^{[b]}\otimes\kappa$, and moreover we can construct a map $\bar{h}$ from the Hecke algebra for $M_{[b]}$ to $\bbQ[p-\Isog^{\ord}\otimes\kappa]$. 

Consider any geometric point of $p-\Isog^{\ord}\otimes\kappa$, which gives us a $p$-quasi-isogeny $f:\calA_x\rightarrow \calA_y$ of abelian varieties over $\bar{\mathbb{F}}_p$, where $x,y\in \mathscr{S}_k(\bar{\mathbb{F}}_p)$. This induces a quasi-isogeny of $p$-divisible groups over $\bar{\mathbb{F}}_p$. To lift $f$ to a $p$-quasi-isogeny over $O_{\bar{L}}$, it suffices by Serre-Tate theory to lift the associated quasi-isogeny of $p$-divisible groups, which we also denote by $f$. 

From the results in \ref{thm:muordlifting}, we see that given a quasi-isogeny $f:\mathcal{A}_x\rightarrow \mathcal{A}_y$ there is a unique lift to an isogeny $f^{can}:\mathcal{A}_x^{can}\rightarrow\mathcal{A}_y^{can}$ over $O_L$, such that $\mathcal{A}_x^{can},\mathcal{A}_y^{can}$ are the canonical lifts. Let $M:=M_{[b^{\ord}]}$. Consider the $p$-adic Tate modules $T_p(\mathcal{A}_x^{can})$ and $T_p(\mathcal{A}_y^{can})$. By construction, $M$ is the centralizer of $\bar{\upsilon}$. We can choose isomorphisms $\alpha_1:\Lambda\simeq T_p(\mathcal{A}_x^{can})$, $\alpha_2:\Lambda\simeq T_p(\mathcal{A}_y^{can})$ such that the map $\alpha_2^{-1}\circ f^{can}\circ \alpha_1$ gives a well-defined coset of in $M_c\backslash M(\mathbb{Q}_p)/M_c$
	
Similar to the case over the generic fiber, to each geometric point in $p-\Isog^{\ord}\otimes\kappa$ we can associate an element $\tau'$ which maps to each point the coset representative in $M_c\backslash M(\mathbb{Q}_p)/M_c$. Such a map is locally constant, by a proof that is almost identical to the proof of \cite[Lem 4.2.11]{M2004}.
\begin{proposition}
		The function $\tau'$ is constant on irreducible components of $p-\Isog^{\ord}\otimes \kappa$.
\end{proposition}
  
Thus, we can define a $\mathbb{Q}$-module homomorphism
  \begin{equation}
  \label{eqn:barh}
    \bar{h} : H_0(M(\mathbb{Q}_p)//M_c) \rightarrow \mathbb{Q}[p-\Isog^{\ord}\otimes\kappa]
  \end{equation}
as follows. $\bar{h}$ maps $1_{M_cmM_c}$ to the formal sum $\sum\frac{1}{\mathrm{deg}(C)}C$ over all $C$ such that $\tau'(C) = M_cm^{-1}M_c$, and $\mathrm{deg}(C)$ is the degree of the finite flat map $s:C^{ord}\rightarrow \sS^{ord}$.

$\bar{h}$ is clearly a map of $\mathbb{Q}$-modules. Note that $\mathbb{Q}[p-\Isog^{\ord}\otimes\kappa]$ also has the structure of a $\mathbb{Q}$-algebra given by the map \eqref{eqn:compositionisogeny}, since $s,t$ are finite flat morphisms when restricted to $p-\Isog^{\ord}\otimes\kappa$ (proof is identical to that in \cite[4.2.2]{M2004}). With this structure, $\bar{h}$ is clearly also a $\mathbb{Q}$-algebra homomorphism. $\bar{h}$ is related to $h$ in the following way:
\begin{proposition}
   We have the following commutative diagram
    \begin{equation}
    \label{eq:comdiagrammuord}
        \begin{tikzcd}
          \mathcal{H}_0(G(\mathbb{Q}_p)//K_p)\arrow{rr}{h}\arrow{dd}{\dot{\mathcal{S}}^G_M} &&\mathbb{Q}[p-\Isog\otimes E]\arrow{dd}{\ord\circ S}\\
          &&\\
          \mathcal{H}_0(M(\mathbb{Q}_p)//M_c)\arrow{rr}{\bar{h}}&&\mathbb{Q}[p-\Isog^{\ord}\otimes \kappa].
        \end{tikzcd}
        \end{equation}
\end{proposition}
\begin{proof}
  Consider any irreducible component $C\subset p-\Isog^{\ord}\otimes\kappa$, and any geometric point on $C$. This is a quasi-isogeny between $\mu$-ordinary $p$-divisible groups $\sG_1,\sG_2$ over $\bar{\bbF}_p$. Let $\tilde{\sG}_1$ be the canonical lift of $\sG_1$, and suppose that the canonical lift of the quasi-isogeny is given by $mK_p$ for some $m\in M(L)$. Under the twisted Satake isomorphism $\dot{S}^G_M$, the function $1_{mK_p}$ gets mapped to $1_{mM_c}$. Thus, we see that the image of the type $\tau$ as a $K_p$-double-coset gets mapped to the image of the type $\tau'$, so the diagram commutes.
\end{proof}
\subsubsection{}
For any abelian scheme $\calA_x$, consider the relative Frobenius isogeny. In particular, since the residue field $\kappa$ is of order $p^n$, there is a Frobenius section of the source morphism $s: p-\operatorname{Isog}\otimes\kappa\rightarrow\mathscr{S}_\kappa$, sending the abelian variety $\mathcal{A}_x$ to the $n$-th power Frobenius map on $\mathcal{A}_x$. Let $\Frob$ denote its image, which is a closed reduced subscheme of $p-\Isog\otimes\kappa$. In fact, restricting to $p-\Isog^{\ord}\otimes\kappa$, it is a union of irreducible components of $p-\operatorname{Isog}^{\ord}\otimes\kappa$. This allows us to consider $\Frob$ as an element of $\mathbb{Q}[p-\operatorname{Isog}^{ord}\otimes\kappa]$. 
 	
We now would like to determine the double coset corresponding to $\Frob$. Let $x\in\mathscr{S}^{\ord}_\kappa(\overline{\mathbb{F}}_p)$. The lift of the relative Frobenius isogeny is hence given on $T_p^*\sG (-1)$ by
\begin{equation*}
\mathbb{D}(\mathscr{G}^{(p)}_x)=g\mathbb{D}(\mathscr{G}_x)\xrightarrow{\sim} \mathbb{D}(\mathscr{G}_x),
\end{equation*}
where $g=(b\sigma)^n(1)=\upsilon\sigma(\upsilon)\dots\sigma^{n-1}(\upsilon)(p)=\lambda\sigma(\lambda)\dots\sigma(\lambda)(p)$, where we recall that $\lambda$ was the unique dominant Weyl conjugate of $\mu^{-1}$.  

Hence, we see that a $p$-power isogeny $f:\mathcal{A}_1\rightarrow \mathcal{A}_2$ of $\mu$-ordinary abelian varieties is isomorphic to the Frobenius if and only if its type is $M_c\tilde{\lambda}^{-1}(p)M_c$. Hence, we see that under $\bar{h}$, $\Frob$ is the image of $1_{M_c\tilde{\lambda}(p)M_c}$. This, combined with the commutativity of the map in \eqref{eq:comdiagrammuord} and the result of B\"{u}ltel \eqref{eqn:bultelresult}, gives us the following proposition.
\begin{proposition}
\label{prop:muordinary}
  Viewing $H_{G,X}(x)$ as a polynomial in $\bbQ[p-\Isog^\ord\otimes\kappa]$ via the morphism $\mathrm{ord}\circ h$, the following relation holds:
    \begin{equation*}
      H_{G,X}(\Frob)^\ord=0.
    \end{equation*}
\end{proposition}

\subsubsection{}
Now, observe that since the $\mu$-ordinary locus is open and dense by the main result of \cite{Wor2013}, and $s$ is finite and flat over the $\mu$-ordinary locus, we see that $\dim p-\Isog^{\ord}\otimes\kappa=\dim\mathscr{S}_\kappa^{\ord}=2\langle\rho,\mu\rangle$. Hence, we see that the closure of $p-\Isog^{\ord}\otimes\kappa$ is the union of some irreducible components of $p-\Isog\otimes\kappa$. However, it is \emph{not} true that $p-\Isog^{\ord}\otimes\kappa$ is dense in $p-\Isog\otimes\kappa$. For instance, this property fails for the Hilbert-Blumenthal modular variety attached to a totally real field of degree 2 where $p$ is inert (c.f. main result of \cite{St1997}). 

\subsubsection{}
We would like to determine what possibilities we have for the irreducible components of $p-\Isog\otimes\kappa$ which are not contained in the $\mu$-ordinary locus.

Firstly, we want to obtain an upper bound for the dimenision of $p-\Isog^{[b]}\otimes\kappa$. Thus, fix a point $x\in\mathscr{S}_\kappa(\overline{\mathbb{F}}_p)$, and let $y\in\mathscr{S}_\kappa(\overline{\mathbb{F}}_p)$ be such that there is $p$-quasi-isogeny $f:\mathcal{A}_x\rightarrow \mathcal{A}_y$. Note that given $\mathcal{A}_x$, $f$ is entirely determined by the induced quasi-isogeny of $p$-divisible groups $\mathscr{G}_x\rightarrow\mathscr{G}_y$, which we know by (contravariant) Dieudonne theory is determined by the induced map of Dieudonne modules
\begin{equation*}
    f:\mathbb{D}(\mathscr{G}_y)\rightarrow\mathbb{D}(\mathscr{G}_x).
\end{equation*}

From the previous section, we see that the set of possible $p$-quasi-isogenies of $x$ is given by the image of $\iota_x(X_\upsilon(b))$. We now apply Theorem \ref{thm:dimADLV}. Then, we see that the dimension of $X_\upsilon(b)$, viewed as a scheme over $\bar{\bbF}_p$, is $\langle\rho,\mu-\nu({[b]})\rangle-\frac{1}{2}\defect(b)$, hence the dimension of $s^{-1}(x)$ is less than or equal to $\langle\rho,\mu-\nu({[b]})\rangle-\frac{1}{2}\defect(b)$.

\begin{remark}
    In fact, following the work of Kim (c.f. \cite{kim_2018}) and Zhu (c.f. \cite{Zhu2014}), $X_\upsilon(b)$ is isomorphic to the perfection of the reduced special fiber of the Rapoport-Zink space of Hodge type $\mathcal{M}(G,\mu,x)$, and we know that by definition $s^{-1}(x)$ is a subscheme of the special fiber of $\mathcal{M}(G,\mu,x)$.
\end{remark}

Combining this with Theorem \ref{thm:dimensionnewtonstrata}, we have the following lemma.
\begin{lemma}
    \label{lem:dimensionY}
    The dimension of $p-\Isog^{[b]}\otimes\kappa$ is at most
    \begin{equation*}
    2\langle\rho,\mu\rangle-\operatorname{def}_{G}(b).
    \end{equation*}
\end{lemma}  
We thus have the following key proposition. For any geometrically irreducible subset $C\subset p-\Isog\otimes\kappa$, we let $C^{[b]}$ denote the intersection of $C$ and $p-\Isog^{[b]}\otimes\kappa$. Since $B(G,\upsilon)$ is a finite set, there must exist some $[b]$ such that $C^{[b]}$ is dense in $C$.
\begin{proposition}
  Let $C$ be any geometrically irreducible component of $p-\Isog\otimes\kappa$. For any $[b]\in B(G,\upsilon)$, if $C^{[b]}$ is dense in $C$, then $b$ must be unramified.
\end{proposition}
\begin{proof}
  It is clear that the dimension of any geometrically irreducible component of $p-\Isog\otimes\kappa$ is $2\langle\rho,\mu\rangle$, since $p-\Isog\otimes E$ is of dimension $2\langle\rho,\mu\rangle$. If $\defect_G(b)>0$, then the dimension of $C^{[b]}$ is strictly less than $2\langle\rho,\mu\rangle$, so it cannot be dense in $C$. 
\end{proof}
\begin{remark}
\label{rmk:split}
  Remark \ref{rmk:splitnounramified} implies that if $G$ is split over $\mathbb{Q}_p$, then the only unramified element in $B(G,\upsilon)$ is the $\mu$-ordinary $\sigma$-conjugacy class, and hence the $\mu$-ordinary locus is dense in $p-\Isog\otimes\kappa$. This, combined with the $\mu$-ordinary congruence relation, gives a simple proof of the congruence relation in the case where $G$ is split over $\mathbb{Q}_p$. This extends the main result of \cite{W2000} to the Hodge type case. We note that there is a slight difference in the result obtained here because Wedhorn shows the density result of the special fiber of the fbeginull moduli space $p-\Isog$, whereas here we show the density result of the special fiber of the flat closure of the generic fiber $p-\Isog\otimes E$, which may be smaller. 
\end{remark}
We will refer to geometrically irreducible components $C$ of $p-\Isog\otimes\kappa$ such that $C^{[b]}$ is dense in $C$ as being $[b]$-dense.

Given any $D\in \bbQ[p-\Isog\otimes\kappa]$ and any $[b]\in A$, we define
$\mathrm{res}^{[b]}(D)$ to be the sum of $[b]$-dense irreducible components in $D$.

\subsection{Structure of $[b]$-dense irreducible components}
In this subsection, we describe the $[b]$-dense irreducible components of $p-\Isog\otimes\kappa$ in terms of the irreducible components of the affine Deligne-Lusztig variety $X_\mu(b)$. 

\subsubsection{} We recall the construction in \cite{mantovan2005cohomology,HK2019adic} of the finite and infinite-level Igusa varieties for Siegel modular varieties and general Hodge-type Shimura varieties.

Let us fix an embedding of $G\hookrightarrow \mathrm{GSp}_{2n}$ for some $n$. Let $b'$ denote the image of $b$ in $B(\mathrm{GSp}_{2n})$, $\mu'$ denote the image of $\mu$ in $X_*(\mathrm{GSp}_{2g})$. Let $\sG_0$ be the $p$-divisible group over $\bar{\bbF}_p$ which corresponds to $b'$. Since $\sG_0$ is isogenous to a completely slope divisible $p$-divisible group (see \cite[\S 3.2]{Ham2017}), we may assume that $\sG_0$ is completely slope divisible. Attached to $\sG_0$ is a polarization which we denote by $\lambda_0$. We define the central leaf $C'\subseteq \mathbf{A}^{b'}_g$ associated to $(\sG_0,\lambda_0)$ as follows.
\begin{equation*}
    C':=\{x\in \mathbf{A}^{b'}_g: (\calA_{x}[p^\infty],\lambda_x)\simeq (\sG_0,\lambda_0)\otimes k(x)\}.
\end{equation*}

Since $C'$ is smooth by \cite[Prop 3.8]{Ham2017}, the universal $p$-divisible group over the central leaf $\calX_0:=\calA_{univ}[p^\infty]|_{C'}$ is completely slope divisible, as shown in \cite[\S3]{mantovan2005cohomology} Let $\sG^i_0$ (resp. $\calX^i_0$) be the $i$-th piece of slope decomposition of $\sG_0$ (resp. $\calX_0$).

For any positive integer $m$, we let $J'_m\rightarrow C'$ be the Igusa variety of level $m$, defined as the moduli space of isomorphisms $\{j^i_m\}$ such that $j^i_m:\sG^i_0[p^m]\xrightarrow{\sim}\calX^i_0[p^m]$ which commute with the polarisations and for any $m'>m$ can be lifted \'{e}tale locally to an isomorphism of $p^{m'}$-torsion points. Moreover, we have natural inclusion maps $J'_{m'}\rightarrow J'_{m}$ for any $m'>m$, so we define the inverse limit
\begin{equation*}
    J'_\infty :=\varprojlim J'_m.
\end{equation*}
Over $J'_\infty$, we have universal isomorphisms, which we denote by $j^{i,univ}$. Let $J'^{(p^{-\infty})}_\infty$ denote the perfection of $J'_\infty$. As the slope filtration splits canonically over perfect schemes, we in fact have an isomorphism 
\begin{equation*}
    \calX_{0,J'^{(p^{-\infty})}_\infty}=\oplus \calX^i_{0,J'^{(p^{-\infty})}_\infty}\simeq \sG_0\times J'^{(p^{-\infty})}_\infty
\end{equation*}
induced the perfection of $\oplus j^{i,univ}$, which we denote by $j$.

We now define the central leaf and Igusa varieties with $G$-structure. Recall that to every point $s\in\sS_\kappa(\bar{\bbF}_p)$ we have defined crystalline Tate tensors $(s_{G,\alpha,x})$ on the Dieudonn\'{e} module $\bbD(\calA_{G,x}[p^\infty])$. Moreover, $\sG_0$ is also equipped with tensors $(s_{0,\alpha})$. Let

\begin{equation*}
    C:=\{x\in \sS_\kappa: (\calA_{x}[p^\infty],\lambda_x,s_{G,\alpha,x})\simeq (\sG_0,\lambda_0,s_{0,\alpha})\otimes k(x)\}.
\end{equation*}

We define the perfection of the infinite level Igusa variety over $C$ in the following way. Let $J_\infty^{(p^{-\infty})}\subset (J'_\infty\times_{C'} C)^{(p^{-\infty})}$ be the locus where $s_{0,\alpha}=j_*s_{G,\alpha}$ for every $\alpha$, i.e. we want the locus where the isomorphisms preserve the tensors over the $p$-divisible group. This gives a well-defined scheme, and we define the infinte level Igusa variety as
\begin{equation*}
    J_\infty:=im(J_\infty^{(p^{-\infty})}\rightarrow J'_\infty\times_{C'}C),
\end{equation*}
and the level $m$ Igusa variety as
\begin{equation*}
    J_m:=im(J_m^{(p^{-\infty})}\rightarrow J'_m\times_{C'}C).
\end{equation*}

\subsubsection{}
Following \cite{HK2019adic}, we can define a surjective map
\begin{equation*}
    \pi_\infty:J_\infty^{(p^{-\infty})}\times \RZ(G,b,\mu)^{red}\rightarrow \sS^b_0
\end{equation*}
which is a $J_b$-torsor for the pro-\'{e}tale topology, defined as follows. Let $(\calA_{univ},\lambda_{univ},\eta_{univ})$ be the universal abelian variety over $\mathbf{A}^{b'}_g$, and let $(\calA,\lambda,\eta)$ be the pullback of  $(\calA_{univ},\lambda_{univ},\eta_{univ})$ to $J'^{(p^{-\infty})}_\infty\times \RZ(\GSp_{2n},b',\mu')^{red}$ via the map
\begin{equation*}
    J'^{(p^{-\infty})}_\infty\times \RZ(\GSp_{2n},b',\mu')^{red}\rightarrow J'^{(p^{-\infty})}_\infty\rightarrow \mathbf{A}^{b'}_g,
\end{equation*}
where the first map is the projection onto the first factor, and the second is the composition $J'^{(p^{-\infty})}_\infty\rightarrow J'_\infty\rightarrow C'\rightarrow \mathbf{A}^{b'}_g$.

Denote by $\rho_{univ}$ be the universal quasi-isogeny over $\RZ(\GSp_{2n},b',\mu')$. Zariski-locally there exists an integer $m_1$ such that $p^{m_1}\rho_{univ}$ is an isogeny. By glueing $\calA/j(\ker p^{m_1}\rho)$ over a suitable Zariski covering, we obtain a polarised abelian variety over $J'^{(p^{-\infty})}_\infty\times \RZ(\GSp_{2n},b',\mu')^{red}$, with polarization and level structure induced by $\lambda,\rho$ respectively. Hence, we get a map
\begin{equation*}
    \pi'_\infty: J'^{(p^{-\infty})}_\infty\times \RZ(\GSp_{2n},b',\mu')^{red}\rightarrow \mathbf{A}^{b'}_g.
\end{equation*}
If we restrict $\pi'_\infty$ to $J^{(p^{-\infty})}_\infty\times \RZ(G,b,\mu)$, this map factors through $\sS^{-}_\kappa$, and in fact there is a unique lift of $\pi'_\infty\vert_{J^{(p^{-\infty})}_\infty\times \RZ(G,b,\mu)^{red}}$ to a map
\begin{equation*}
    \pi_\infty:J_\infty^{(p^{-\infty})}\times \RZ(G,b,\mu)^{red}\rightarrow \sS^b_0.
\end{equation*}

Now, we can define a map 
\begin{equation*}
    \tilde{\pi}_\infty:\RZ(G,b,\mu)^{red}\times J_\infty^{(p^{-\infty})}\times \RZ(G,b,\mu)^{red}\rightarrow p-\Isog\otimes\kappa^{[b]}
\end{equation*}
as follows. Given $(x,y,z)\in \RZ(G,b,\mu)^{red}\times J_\infty^{(p^{-\infty})}\times \RZ(G,b,\mu)^{red}$, the image of $(x,y,z)$ is the pair $(x',f)$ where $x'=\pi_\infty(y,x)\in \sS^b_0$, and $f$ is the quasi isogeny between $x'$ and $x'':=\pi_\infty(y,z)$ given by $\rho_{univ,z}\circ \rho^{-1}_{univ,x}$. 

Observe that $\tilde{\pi}_\infty$ is equivariant for the action of $J_b\times J_b$, and also a $J_b$-torsor for the pro-\'{e}tale topology. Moreover, we observe that since $\pi_\infty$ is surjective, $\tilde{\pi}_\infty$ is also surjective. Since $\pi_\infty$ is universally closed when restricted to the product of $J_{\infty}^{p^{-\infty}}$ and a quasi-compact closed subset of $\RZ(G,b,\mu)^{red}$, so too is $\tilde{\pi}_\infty$. Thus, observe the following:

\begin{proposition}
\label{prop:tripleproduct}
Every irreducible $[b]$-dense component of $p-\Isog\otimes\kappa^{[b]}$ is the image of some triple $(X,Y,Z)$, where $X,Z\subset \RZ(G,b,\mu)^{red}$ are irreducible components, and $Y\subset J_\infty^{(p^{-\infty})}$ is an irreducible component. Moreover, the pair $(X,Z)$ is determined up to the action of $J_b$.
\end{proposition}

\section{Hecke Correspondences}
The goal of this section is to understand the action of the Hecke algebra on the cohomology of the Rapoport-Zink space, in particular the top-dimensional compactly supported cohomology, and show that it determines irreducible components of $p-\Isog\otimes\kappa$.

\subsection{Preliminaries}

For the entirety of this section, we will fix an unramified $\sigma$-conjugacy class $[b]\in B(G,\upsilon)$, and a representative $b\in [b]$ which we assume is of the form $b=\beta(p)$, for some $\beta$ is the Weyl-orbit of $\upsilon$. For notational simplicity, in this entire section, we will denote the formal schemes $\RZ(G,b,\mu)$ and $\RZ(P,b,\mu)$ by $\RZ_G$, $\RZ_P$, respectively. Let $\calH(G):=\calH(G(\bbQ_p)//G(\bbZ_p))$, and $\calH(M):=\calH(M(\bbQ_p)//M(\bbZ_p))$.

\subsubsection{}Let $\widehat{S}$ be the completion of the Shimura variety $\sS_K(G,X)$ along the special fiber $\sS_{\kappa}$, and $\calS$ be the rigid-analytic generic fiber. Similarly, we let $\widehat{S}_{[b]}$ be the completion of the Shimura variety $\sS$ along the the Newton strata $\sS_\kappa^{[b]}$, and $\calS_{[b]}$ be the generic fiber of $\widehat{S}_{[b]}$, viewed as a rigid-analytic space. Similarly, given any closed subscheme $C$ of $p-\Isog\otimes\kappa$, we define $\calC$ to be the generic fiber of the completion $\widehat{p-\Isog}_{C}$, viewed as a rigid-analytic space. 

We can define a cohomological correspondence on $\calS$ supported on $\calC$, as follows. (See Appendix \ref{appendix1} for the definition of cohomological correspondence) Since $s,t$ induce finite \'{e}tale maps, 
\begin{equation*}
    \calS \xleftarrow{s} \calC \xrightarrow{t} \calS,
\end{equation*}
we define the map from $u_\calC:s^*\bbQ_l\rightarrow t^!\bbQ_l=t^*\bbQ_l$ to be the identity. 

If $C$ is the finite union of some irreducible components $A_1,\dots,A_n$ of $p-\Isog\otimes\kappa$, then observe that as cohomological correspondences on $(\calS,\bbQ_l)$, we have $u_\calC=u_{\calA_1}+\dots+u_{\calA_n}$, where on the right-hand side we $!$-pushforward the correspondence supported on $\calA_i$ to $\calC$ via the inclusion map.

\subsubsection{}We can form a commutative diagram of formal schemes
\begin{equation*}
    \begin{tikzcd}
\RZ_G \arrow[d, "\pi_\infty"] & \frakD \arrow[l, "s'"'] \arrow[r, "t'"] \arrow[d, "f'"] & \RZ_G \arrow[d, "\pi_\infty"] \\
\widehat{S}                  &  \widehat{p-\Isog}_{C}\arrow[l, "s"'] \arrow[r, "t"]                  & \widehat{S} \end{tikzcd}
\end{equation*}
 where we define $\frakD=\widehat{p-\Isog}_{C}\times_{s,\pi_\infty}\RZ_G$, so the left square is Cartesian. Here, $t'$ is the map defined as follows. Given an element $(x,(H,\beta))$ of $\frakD$, such that $x\in \widehat{p-\Isog}_{C}, (H,\beta)\in\RZ_G$, $x$ defines a quasi-isogeny of $p$-divisible groups $x:\sG\rightarrow H$, for some $p$-divisible group $\sG$ with $G$-structure. We then define $s'(x,(H,\beta))=(\sG,x^{-1}\beta)$. Observe that both squares are Cartesian in the above diagram. In particular, we can consider the $!$-pullback of the cohomological correspondence from $\widehat{S}$ to $\RZ_G$, on both the rigid analytic generic fiber and the special fiber. 
 
For $g\in G(\bbQ_p)$, other than the correspondence associated with the closed subscheme $C=h(1_{K_pgK_p})\subset p-\Isog\otimes\kappa$, we also have another natural model for the Hecke correspondence given by $K_pgK_p$ on $(\calS,\bbQ_l)$, as follows. For $m\geq 0$, we define $K_p(m)$ to be the kernel of the canonical projection $G(\bbZ_p)\rightarrow G(\bbZ/p^m)$. Then we can define the cover $\calS^m$ as the cover classifying $K_p(m)$-level structures over $\calS$. More precisely, we have
 \begin{equation*}
     \calS^m(R) =\{(P,\eta_m)|P\in \calS(R),\eta_m: \Lambda/p^m\xrightarrow{\sim} \sA_{G,P}[p^m]=T_p(\sA_{G,P})/p^m \text{ sending } (s_\alpha)\text{ to }(t_{\alpha,\acute{e}t})\}.
 \end{equation*}
 Similarly, we also have an \'{e}tale covering $\RZ_G^{rig,m}$ as the cover classifying $K_p(m)$ level structures over $\RZ^{rig}_G$.
 
For $g\in G(\bbQ_p)$, we define $e(g)$ to be the minimal non-negative integer such that $g\Lambda\subset p^{-e(g)}\Lambda$. For $m\geq e(g)$, we can consider the correspondence given by
\begin{equation}
\label{eqn:corr1}
\begin{tikzcd}
  & \RZ^{rig,m}_G \arrow[ld,"c_1"] \arrow[r, "{[g]}"] & \RZ_G^{rig,g^{-1}K_p(m)g} \arrow[rd,"c_2"] &   \\
\RZ^{rig}_G &                             &              & \RZ^{rig}_G,
\end{tikzcd}
\end{equation}
where the map $[g]:\RZ_G\rightarrow \RZ_G^{g^{-1}K_p(m)g}$ is defined as in \cite[\S7.4]{kim_2018}, and $c_1,c_2$ are the natural projection maps induced by change of level. Similarly, we also have a correspondence over $\calS$ supported on $\calS^m$. It is clear from the construction that, if we pushforward the correspondences onto their image in $\calS\times\calS$ (similarly $\RZ^{rig}_G\times\RZ_G^{rig}$), then the induced cohomological correspondence on the rigid-analytic generic fiber (where we again let the map $u:c_1^*\bbQ_l\rightarrow c_2^!\bbQ_l=c_2^*\bbQ_l$ be the identity) is the same as that defined by $u_\calC$ above for $C=h(1_{K_pgK_p})$.  
\subsubsection{}
\label{section:modelTp}
Moreover, we have a model for this cohomolgical correspondence as a formal scheme over $O_E$. Indeed, following \cite[6.1.2]{HK2019adic} and \cite[\S6]{mantovan2005cohomology}, we have an integral
model $\sS_{m,g}$ of $\calS^m$, and a formal model $\RZ_{G,m,g}$
of $\RZ^{rig,m}_G$ constructed via Drinfeld level structures, equipped with the following proper morphisms
\begin{equation*}
    [g] : \sS_{m,g} \rightarrow \sS_{m-e(g)}
\end{equation*}
extending the map on the generic fiber given by
\begin{equation*}
    \calS^m \xrightarrow{[g]}\calS^{g^{-1}K_p(m)g}\twoheadrightarrow \calS^{m-e(g)}.
\end{equation*}
Similarly, we have a proper morphism
\begin{equation*}
    [g] : \RZ_{G,m,g} \rightarrow \RZ_{m-e(g)}
\end{equation*}
extending the map on the generic fiber given by
\begin{equation*}
    \RZ_G^m \xrightarrow{[g]}\RZ_G^{g^{-1}K_p(m)g}\twoheadrightarrow \RZ_G^{m-e(g)}.
\end{equation*}
Thus, we have an integral model for the correspondence \eqref{eqn:corr1} given by 
\begin{equation*}
\begin{tikzcd}
  & \RZ_{G,m,g} \arrow[ld,"c_1"] \arrow[r, "{[g]}"] & \RZ_{G,m-e(g)} \arrow[rd,"c_2"] &   \\
\RZ &                             &              & \RZ,
\end{tikzcd}
\end{equation*}
and similarly we also have a formal model of the correspondence on $\calS$.
\subsection{Parabolic reduction}
Now, we define the formal alternating sum of cohomology groups 
\begin{equation*}
    H^{\bullet}_c(\RZ_G^{rig}):=\sum_{i}(-1)^i H^i_c(\RZ_G^{rig})
\end{equation*}
in the Grothedieck group of representations of $\calH(G)\times J_b(\bbQ_p)\times W_E$. Recall here that the compactly supported cohomology $H_c^i(\RZ_G^{rig})$ is defined as
\begin{equation*}
    \lim_{\substack{\longrightarrow\\U}}H_c^i(U,\bbQ_l)
\end{equation*}
where $U\subset \RZ_G^{rig}$ are quasi-compact open subsets.

If $Z\subset \RZ^{red}_G$ is a closed subscheme, we let $\RZ_G(Z)$ denote the formal scheme given as the completion of $\RZ_G$ along $Z$.

We now want to show the following:
\begin{theorem}
\label{thm:pr1}
In the Grothedieck group of representations of $\mathcal{H}(G)\times J_b(\bbQ_p)\times W_E$, we have the equality
\begin{equation*}
    H^{\bullet}_c(\RZ(G,b,\mu)^{rig})=\sum_{\mu'\in I}H^{\bullet}_c(\RZ(M,b,\mu')^{rig}),
\end{equation*}
where 
\begin{equation*}
    I=\{\mu': \mu' \text{ is a dominant cocharacter of $M$ conjugate to }\mu \text{ in }G, \text{with }[b]\in B(M,\mu')\},
\end{equation*}
and on the right-hand side $\calH(G)$ acts via twisted Satake $S^G_M$.
\end{theorem}
\begin{proof}
This result, without the action of the Hecke algebra, follows from \cite[Thm 9.3]{Man08}. We describe how to adapt the proof of that result to get the action of cohomological correspondences.

For notational simplicity, we define
\begin{equation*}
    \RZ_M:=\coprod_{\mu'\in I}\RZ(M,b,\mu').
\end{equation*}
We can form the following commutative diagram associated to the Hecke correspondence for $f=1_{K_pgK_p}$
\begin{equation}
\label{eqn:GPM}
    \begin{tikzcd}
\RZ^{rig}_M  & \coprod_{P(\bbQ_p)\setminus G(\bbQ_p)/K_p(m)}\RZ_M^{rig,m} \arrow[r] \arrow[l] & \RZ_M^{rig} \\
\RZ_P^{rig} \arrow[d,"j"] \arrow[u,"\Xi"] & \coprod_{P(\bbQ_p)\setminus G(\bbQ_p)/K_p(m)}\RZ_P^{m,rig} \arrow[d] \arrow[r] \arrow[l] \arrow[u] & \RZ_P^{rig} \arrow[d] \arrow[u] \\
\RZ_G^{rig}           & \RZ_G^{m,rig} \arrow[r] \arrow[l]           & \RZ_G^{rig}.          
\end{tikzcd}
\end{equation}
Observe that the bottom two squares in the diagram are Cartesian.

Observe that we can have the inclusion map $\Theta:\RZ_M\rightarrow \RZ_P$, and the projection map $\Xi:\RZ_P\rightarrow \RZ_M$ given by taking the direct sum of the filtrands, which satisfy $\Xi\circ\Theta=\mathrm{Id}_{\RZ_M}$. Moreover, the arguments in \cite[\S7.3,7.4]{Man08} show that $\Xi$ induces an isomorphism
\begin{equation*}
     R\Gamma_c(\RZ_P^{rig},\bbQ_l)\simeq R\Gamma_c(\RZ_M^{rig},\Xi!(\bbQ_l))=R\Gamma_c(\RZ_M^{rig},\bbQ_l(d'))[2d'].
\end{equation*}
where $d'=\dim\RZ_P^{rig}-\RZ_M^{rig}$, and moreover from \cite[Prop 7.8]{Man08} we also have an isomorphism
\begin{equation*}
     R\Gamma_c(\RZ_P^{rig,m},\bbQ_l)\simeq R\Gamma_c(\RZ_M^{rig,m},\bbQ_l(d'))[2d'].
\end{equation*}
The isomorphisms are compatible with the change in level, and the action of $p\in P(\bbQ_p)$, hence we have a commutative diagram
\begin{equation*}
\begin{tikzcd}
R\Gamma_c(\RZ_P^{rig},\Xi^!(\bbQ_l)) \arrow[r,"\Xi^!(u_f)"] \arrow[d] & R\Gamma_c(\RZ_P^{rig},\Xi^!(\bbQ_l)) \arrow[d] \\
R\Gamma_c(\RZ_M^{rig},\bbQ_l) \arrow[r,"u_f"]           & R\Gamma_c(\RZ_M^{rig},\bbQ_l)          
\end{tikzcd}
\end{equation*}
Observe that the induced Hecke action on $\RZ_M$ of $\calH(M)$ given by the twisted Satake map $\calS^G_M(1_{K_pgK_p})$.

We now want to compare $R\Gamma_c(\RZ_P^{rig},\bbQ_l)$ and $R\Gamma_c(\RZ_G^{rig},\bbQ_l)$. Firstly, let $\mathfrak{X}$ denote the formal scheme given as the image of $\RZ_P$ in $\RZ_G$, as constructed in \cite[Thm 9.2]{Man08}. We may consider the cohomological correspondence on the rigid-analytic generic fiber $\calX$ of $\mathfrak{X}$, give by $!$-pullback via the following commutative diagram
\begin{equation*}
    \begin{tikzcd}
\calX \arrow[d,"j"] & \calE \arrow[l, "s'"'] \arrow[r, "t'"] \arrow[d,"\tilde{j}"] & \calX \arrow[d,"j"] \\
\RZ_G^{rig}                  &  \RZ_G^{m,rig}\arrow[l, "s'"'] \arrow[r, "t'"]                  & \RZ_G^{rig}. \end{tikzcd}
\end{equation*}
Here, we define $\calE=\calX\times_{j,s'}\RZ_G^{m,rig}$, so the left square is Cartesian. We have a map $t'$ on the top row because by construction the map $\calD\rightarrow T(f)\xrightarrow{t'} \RZ_G^{rig}$ factors through $\calX$. 

Compatibility with $!$-pullback induces the following commutative diagram
\begin{equation*}
\begin{tikzcd}
R\Gamma_c(\calX,j^!(\bbQ_l)) \arrow[r,"j^!(u_f)"] \arrow[d] & R\Gamma_c(\calX^{rig},j^!(\bbQ_l)) \arrow[d] \\
R\Gamma_c(\RZ_G^{rig},\bbQ_l) \arrow[r,"u_f"]           & R\Gamma_c(\RZ_G^{rig},\bbQ_l),          
\end{tikzcd}
\end{equation*}
and we know that the vertical maps are isomorphisms, since $\RZ_G^{red}=\mathfrak{X}^{red}$.

Now, recall that we may compute the compactly supported cohomology of $\RZ_P^{rig},\calX$ by using Cech covers, as follows. Let $Z\subset \RZ_G^{red}=\mathfrak{X}^{red}$ be a union of irreducible components such that $\RZ_G^{red}=\cup_{t\in J_b(\bbQ_p)} tZ$. Let $\mathfrak{X}(Z)$ be the completion of $\mathfrak{X}$ along $Z$, and $\calZ \subset \calX$ be the generic fiber of $\mathfrak{X}(Z)$, which is an open subspace satisfying the condition $\calX^{rig} = \cup_{t\in J_b(\bbQ_p)} t\calZ$. Associated to this Cech covering we have a spectral sequence of $W_E$-representations
\begin{equation*}
    E_1^{p,q} = \bigoplus_{t_1,\dots,t_p\in J_b(\bbQ_p)}
H^q_c(\calZ_{t_1,\dots,t_p} ,\bbQ_l) \Rightarrow H^{p+q}_c (\calX,\bbQ_l)
\end{equation*}
where $\calZ_{t_1,\dots,t_p}=\cap t_i\calZ$, and $t_i\calZ\neq t_j\calZ$ for $i\neq j$. 

Now, we describe the Hecke action. Let $f=1_{K_p gK_p}\in\calH(G)$. As discussed above, we have an associated cohomological correspondence, on $\RZ_G^{red}$ which we denote by $u_f$, and which we $!$-pullback to a cohomological correspondence on $\calX$, given by $\tilde{j}^!(u_f)$. Let $Z':= t'(s'^{-1}(Z))$, which is also a union of irreducible components in $\RZ_G^{red}$, and its $J_b(\bbQ_p)$-orbit covers $\RZ_G^{red}$. We can repeat the above constructions with $Z'$ in place of $Z$, and denote the associated objects by $\calZ'_\bullet$. 

Consider the open inclusion $\iota:\calZ_{t_1,\dots,t_p}\rightarrow \calX$. We can define a commutative diagram
\begin{equation*}
    \begin{tikzcd}
\calZ_{t_1,\dots,t_p} \arrow[d, "\iota"] & \calE_{t_1,\dots,t_p} \arrow[l, "s'"'] \arrow[r, "t'"] \arrow[d, "\tilde{\iota}"] & \calZ'_{t_1,\dots,t_p} \arrow[d, "\iota'"] \\
\calX                  &  \calE\arrow[l, "s'"'] \arrow[r, "t'"]                  & \calX, \end{tikzcd}
\end{equation*}
where $\calE_{t_1,\dots,t_p}=\calZ_{t_1,\dots,t_p}\times_{s',\tilde{\iota}}\calE$, so the left square is Cartesian. Then, $!$-pullback of the correspondence $\tilde{\iota}^!\tilde{j}^!(u_f)$ gives a cohomological correspondence from $R\Gamma_c(\calZ_{t_1,\dots,t_p},\bbQ_l)$ to $R\Gamma_c(\calZ'_{t_1,\dots,t_p},\bbQ_l)$. Since the Hecke correspondences are $J_b(\bbQ_p)$-equivariant, and compatibility with $!$-pullback of cohomological correspondences tells us that we have a commutative diagram
\begin{equation*}
\begin{tikzcd}
E_1^{p,q} = \bigoplus_{t_1,\dots,t_p\in J_b(\bbQ_p)}
H^q_c(\calZ_{t_1,\dots,t_p} ,\bbQ_l) \arrow[d,"\tilde{\iota}^!(u_f)"] \arrow[r, Rightarrow] & H^{p+q}_c (\RZ^{rig},\bbQ_l) \arrow[d,"u_f"] \\
E_1^{p,q} = \bigoplus_{t_1,\dots,t_p\in J_b(\bbQ_p)}
H^q_c(\calZ'_{t_1,\dots,t_p} ,\bbQ_l) \arrow[r, Rightarrow]           & H^{p+q}_c (\RZ^{rig},\bbQ_l).         
\end{tikzcd}
\end{equation*}

We may consider a similar construction for $\RZ_P$. In particular, consider the map $h:\RZ_P\rightarrow\calX$. Then we have a covering of $\RZ_P^{red}$ given by $\cup_{t\in J_b(\bbQ_p)} t\cdot h^{-1}(Z)$. If we let $\calY$ be $h^{-1}(\calZ)$, and define $\calY_\bullet,\calY'_\bullet$ similarly to above, we also have a cohomological correspondence from $\calY_{t_1,\dots,t_p}$ to $\calY'_{t_1,\dots,t_p}$, and a spectral sequence relating this to the cohomological correspondence $u_{f,P}$.

We will now compare the cohomology groups $H^q_c(\calZ_{t_1,\dots,t_p} ,\bbQ_l)$ and $H^q_c(\calY_{t_1,\dots,t_p} ,\bbQ_l)$. Firstly, we may assume, by passing to a refinement of the open covering, that over $\calZ$ the Tate-module of the universal $p$-divisible group $T_p(\sG^{univ})$ can be trivialized. Then, by construction $\calE$ is the disjoint union of open and closed subschemes $\calE_i$, each of which is isomorphic to $\calZ$. Observe then that for each $i$ we may consider the cohomological correspondence $u_{f,i}$ given by
\begin{equation}
\label{eqn:1}
    \calZ\xleftarrow{s'} \calE_i\xrightarrow{t'} \calZ'
\end{equation}
where $t'$ is the map given by $\calE_i\xrightarrow{[g]} \calE_i'\rightarrow \calZ'$. $u_{f,i}:s'^*\bbQ_l\rightarrow t'^!\bbQ_l$ is the identity. Observe that $\tilde{i}^!j^!(u_{f})=\sum_i u_{f,i}$, where on the right hand side we pushforward the correspondence from $\calE_i$ to $\calE$ via the inclusion map. In particular, since the cohomological correspondence decomposes, it suffices to consider each cohomological correspondence \eqref{eqn:1} separately.

Recall from Proposition \ref{prop:strata} that we have locally closed stratifications $\{Z_\alpha\}$ on $Z$, such that restricted to each strata the universal $p$-divisible group over $Z_\alpha$ admits a filtration $\{\sG_{\bullet}\}$ and the universal quasi-isogeny $\beta:\Sigma\rightarrow \sG$ preserves the filtration. This induces a locally closed stratification $\{\calZ_\alpha\}$ on $\calZ$. Similarly we also have a locally closed filtration $\{\calZ'_\beta\}$ on $\calZ'$. Consider the fiber product $\calE_i\times_{\calZ} \calZ_\alpha$. By construction we have a filtration on the $p^m$ torsion, so we may view $\calE_i\times_{\calZ} \calZ_\alpha\subset \RZ_P^{m,rig}$. Moreover, by the Iwasawa decomposition we may assume that the element $g\in P(\bbQ_p)$. Hence, $[g]$ preserves the filtration, and thus the image $t'(\calE_i\times_{\calZ} \calZ_\alpha)$ also carries a filtration on the universal $p$-divisible group. If we look at the specialization $sp(t'(\calE_i\times_{\calZ} \calZ_\alpha))$, this is a connected subscheme of $Z'$ and we also have a filtration on the universal $p$-divisible group, hence by the construction of the stratification it is contained in some strata $Z'_\beta$. Thus $t'(\calE_i\times_{\calZ} \calZ_\alpha)\subset \calZ'_\beta$.

Moreover, we see that by the same argument that $s(t^{-1}(\calZ'_{\beta}))$ intersects $\calZ_\alpha$, and hence must lie entirely in $Z_\alpha$ by the construction of the locally closed filtration. Thus we have commutative diagram with both squares Cartesian
\begin{equation*}
\begin{tikzcd}
\calZ_\alpha \arrow[d] & \calE_i\times_{\calZ} \calZ_\alpha \arrow[l] \arrow[d] \arrow[r] & \calZ'_\beta \arrow[d] \\
\calZ           & \calE_i \arrow[l] \arrow[r]           & \calZ'        
\end{tikzcd}
\end{equation*}

Thus, we can apply \ref{prop:excision} to see that the cohomological correspondence $u_f$ is compatible with excision, which implies that the maps between virtual representations
\begin{equation*}
    H^\bullet_c(\calZ_{t_1,\dots,t_p},\bbQ_l) \xrightarrow{u_f} H^\bullet_c(\calZ'_{t_1,\dots,t_p},\bbQ_l),
\end{equation*}
and
\begin{equation*}
    \sum_{\alpha}H^\bullet_c(\calZ_{t_1,\dots,t_p,\alpha},\bbQ_l) \xrightarrow{\sum_\alpha u_{f,\alpha}}                      \sum_{\alpha'}H^\bullet_c(\calZ'_{t_1,\dots,t_p,\alpha'},\bbQ_l)
\end{equation*}
are equal.

We have a similar locally closed stratification on $\calY$, given by $\calY_\alpha=h^{-1}(\calZ_\alpha)$, and since $h$ restricted to $\calY_\alpha\rightarrow \calZ_\alpha$ is an isomorphism, we have the equality
\begin{equation*}
    \sum_{\alpha}H^\bullet_c(\calZ_{t_1,\dots,t_p,\alpha},\bbQ_l)=\sum_{\alpha}H^\bullet_c(\calY_{t_1,\dots,t_p,\alpha},\bbQ_l)
\end{equation*}
the second map is equal to that for $H^\bullet_c(\calY_{t_1,\dots,t_p})$.
\end{proof}
\begin{remark}
This result would follow by taking $K_p$-invariants from the Harris-Viehmann Conjecture.
\end{remark}
\subsubsection{} Note that we can recover the representation $H^{2(d-r)}_c(\RZ^{rig},\bbQ_l)$ from the alternating sum $H^{\bullet}_c(\RZ^{rig})$ by the action of $W_E$, since we have
\begin{equation*}
    H_c^i(\RZ_G^{red})=\lim_{\substack{\longrightarrow\\U}}H_c^i(U,\bbQ_l)
\end{equation*}
where $U\subset \RZ_G^{red}$ are quasi-compact open subsets. In particular, we know that $U$ are open subsets of dimension $r$, defined over a finite field $k$, and thus the top-dimensional compactly supported cohomology $H^{2r}_c(U,\bbQ_l)$ can be recovered from the alternating sum $\sum_i(-1)^iH^{i}_c(U,\bbQ_l)$ using Frobenius weights, since the degree $2r$-th compactly supported cohomology of $U$ is the only part with Frobenius weight $r$. Thus, we see that by passing to the limit, we can also recover the degree $2r$-th cohomology of $\RZ_G^{red}$, which is the only part with Frobenius weight $r$.

The above theorem shows us that as elements in the Grothendieck group, we have an equality
\begin{equation*}
    H^{2r}_c(\RZ(G,b,\mu)^{red})=\bigoplus_{\mu'\in I}H^{2r'}_c(\RZ(M,b,\mu')^{red}).
\end{equation*}
It turns out that this result is true even without needing to pass to the image in the Grothedieck group, as the following theorem shows.
\begin{theorem}
\label{thm:pr2}
We have an equality of $\calH(G)$-representations
\begin{equation*}
    H^{2r}_c(\RZ(G,b,\mu)^{red})=\bigoplus_{\mu'\in I}H^{2r'}_c(\RZ(M,b,\mu')^{red}),
\end{equation*}
and on the right-hand side $\calH(G)$ acts via twisted Satake $\dot{\calS}^G_M$.
\end{theorem}
\begin{proof}
First observe that we can take the special fiber of the commutative diagram \eqref{eqn:GPM} (here we use the formal model of the correspondences constructed in \ref{section:modelTp})
\begin{equation*}
    \begin{tikzcd}
\RZ_M^{red}  & \coprod_{P(\bbQ_p)\setminus G(\bbQ_p)/K_p(m)}\RZ_M^{m,red} \arrow[r] \arrow[l] & \RZ_M^{red} \\
\RZ_P^{red} \arrow[d,"j"] \arrow[u,"\Xi"] & \coprod_{P(\bbQ_p)\setminus G(\bbQ_p)/K_p(m)}\RZ_P^{m,red} \arrow[d] \arrow[r] \arrow[l] \arrow[u,"\tilde{\Xi}"] & \RZ_P^{red} \arrow[d] \arrow[u] \\
\RZ_G^{red}           & \RZ_G^{m,red} \arrow[r] \arrow[l]           & \RZ_G^{red}.          
\end{tikzcd}
\end{equation*}
Note that the map $\Xi$ is smooth because if we let $Y'$ be an open subspace in $\RZ_M^{red}$, then let
\begin{equation*}
    \Phi:=\Xi^{-1}(Y')= \{(m, n) \in \iota(Y') \times N(L) :mnK_P \in X^{P\subset G}_\mu(b)\},
\end{equation*}
and we have a surjective map $\gamma : \Phi \rightarrow \calE$, where
\begin{equation*}
    \calE:= \{(m, c) |m \in Y', c\in N(L)\cap b^{-1}K\mu(p)K\}=Y' \times (N(L)\cap \mu'(p)^{-1}K\mu(p)K).
\end{equation*}
$\gamma$ is a pro-etale map, while $\mu'(p)N(L)\cap K\mu(p)K$ is an affine space.

Thus, since the top left square is Cartesian, we can consider $!$-pullback of cohomological correspondences, which gives us a commutative diagram
\begin{equation*}
\begin{tikzcd}
H^{2r}_c(\RZ_P^{red},\bbQ_l) \arrow[r,"\Xi^!(u_f)"] \arrow[d] & H^{2r}_c(\RZ_P^{red},\bbQ_l)) \arrow[d] \\
H^{2r'}_c(\RZ_M^{red},\bbQ_l) \arrow[r,"u_f"]           & H^{2r'}_c(\RZ_M^{red},\bbQ_l).          
\end{tikzcd}
\end{equation*}
From Corollary \ref{cor:bijirreducible}, the vertical map $\Xi^!$ induces an isomorphism, because we have a bijection on irreducible components. 

Let $\calX^m$ be the image of $\coprod_{P(\bbQ_p)\backslash G(\bbQ_p)/K_m}\RZ_P^m$ in $\RZ_G^m$. We have the following commutative diagram:
\begin{equation*}
\begin{tikzcd}
\RZ_P^{red}\arrow[d] & \coprod_{P(\bbQ_p)\backslash G(\bbQ_p)/K_m}\RZ_P^{m,red} \arrow[l] \arrow[d] \arrow[r] & \RZ_P^{red} \arrow[d] \\
\calX^{red}          & \calX^{m,red} \arrow[l] \arrow[r]           & \calX^{red}  
\end{tikzcd}
\end{equation*}
The left square is Cartesian, hence pullback of cohomological correspondences gives us a commutative diagram
\begin{equation*}
\begin{tikzcd}
H^{2r}_c(\RZ_P^{red},\bbQ_l) \arrow[r,"j^!(u_f)"] \arrow[d] & H^{2r}_c(\RZ_P^{red},\bbQ_l)) \arrow[d] \\
H^{2r}_c(\calX^{red},\bbQ_l) \arrow[r,"u_f"]           & H^{2r}_c(\calX^{red},\bbQ_l)          
\end{tikzcd}
\end{equation*}
where the vertical maps are isomorphisms, since from \ref{cor:bijirreducible} we have a bijection of top-dimensional irreducible components of dimension $r$. 

By definition, $j^!(u_g)$ and $\Xi^!(u_g)$ are equal, since on the generic fiber they are equal, hence we have an isomorphism of $\calH(G)$-modules.
\end{proof}
\subsection{Irreducible components of $p-\Isog\otimes\kappa$}
\subsubsection{}For any subscheme $Z\subset \RZ^{red}_G$, we denote by $\RZ_G(Z)$ the completion along $Z$. Observe that for any finite union of irreducible components $Y\subset \RZ_G^{red}$, since $\RZ_G(Y)$ is smooth, by Poincar\'{e} duality, we have 
\begin{equation*}
    R\Gamma_c(\RZ_G^{rig}(Y),\bbQ_l)=R\Gamma(\RZ_G^{rig}(Y),\bbQ_l(-d))[-2d])^\vee,
\end{equation*}
and 
\begin{equation*}
    R\Gamma(\RZ_G^{rig}(Y),\bbQ_l)=R\Gamma(Y,R\Psi\bbQ_l), 
\end{equation*}
where $R\Psi\bbQ_l$ denotes the vanishing cycles of $\RZ_G^{rig}(Y)$. Recall that $d$ is the dimension of the Shimura variety, and thus the dimension of the formal scheme $\RZ_G$. Here we used that since $\RZ_G(Y)$ is a smooth formal scheme, we have $R\Psi\bbQ_l=\bbQ_l$. Since $Y$ is proper, we see that the degree $2r$-\'{e}tale cohomology of $\RZ_G^{rig}(Y)$, where $r$ is the dimension of $\RZ_G^{red}$, or by duality the $2(d-r)$-th compactly supported cohomology  $H^{2d-2r}_c(\RZ_G^{rig}(Y),\bbQ_l)^\vee$, then this corresponds to the irreducible components of $Y$. This also holds for an infinite union, since we have that 
\begin{equation*}
    R\Gamma_c(\RZ_G^{rig},\bbQ_l(-d))[-2d])^\vee \simeq R\Gamma_{\calC_{\RZ_G}}(\RZ^{rig},\bbQ_l)\simeq R\Gamma_c(\RZ_G^{red},\bbQ_l)\simeq \varinjlim_{Z} R\Gamma_c(Z,\bbQ_l),
\end{equation*}
where we take the limit over $Z$ which closed quasi-compacts subsets of  $\RZ^{red}$. The first and second equality follow from \cite[1.4.7,1.2.13]{HK2019adic}. Thus we also see that the cohomology classes in degree $2(d-r)$ correspond to irreducible components of $\RZ_G^{red}$. For any irreducible component $X$ of $\RZ_G^{red}$, we let $c(X)$ be the cohomology class it corresponds to in $H^{2d-2r}_c(\RZ_G^{rig})$.
\subsubsection{}Suppose that $C$ is a $[b]$-dense irreducible component, i.e. it is of the form $\tilde{\pi}_\infty(X_1,Z,X_2)$, for some $X_1,X_2$ irreducible components of $\RZ_G^{red}$. Then we observe that for any irreducible component $X_1'$ such that $X_1\cap X_1'\neq \emptyset$ we can further consider the following commutative diagram
\begin{equation*}
    \begin{tikzcd}
\RZ_G(X_1') \arrow[d, "\pi_\infty"] & \frakE \arrow[l, "s'"'] \arrow[r, "t'"] \arrow[d, "f'"] & \RZ_G(\tilde{X}_2) \arrow[d, "\pi_\infty"] \\
\RZ_G                  &  \frakD\arrow[l, "s'"'] \arrow[r, "t'"]                  & \RZ_G, \end{tikzcd}
\end{equation*}
Recall here that we defined $\mathfrak{D}$ in \eqref{section:modelTp}. Here, $\tilde{X}_2$ is the $H$-orbit of the irreducible component $X_2$, where $H$ is the hyperspecial subgroup of $J_b(\bbQ_p)$-stabilizing $X_1'$. Note that this is a finite set. Moreover, we define $\frakE:=\frakD\times_{s',\pi_\infty} \RZ_G(X_1')$. Observe that the map $t'\circ f'$ factors through $\RZ_G(\tilde{X}_2)$, and by definition the left square is Cartesian. We can consider the pullback on the generic fiber $f'^!(u_\calC)$ to a cohomological correspondence from $R\Gamma_c(\RZ(X_1'))$ to $R\Gamma_c(\RZ(X_2'))$. By \ref{prop:shriekpullback}, we have a commutative diagram
\begin{equation}
\label{eqn:comm1}
    \begin{tikzcd}
R\Gamma_c(\RZ_G(X_1')) \arrow[r, "f'^!(u_\calC)"] \arrow[d] & R\Gamma_c(\RZ_G(\tilde{X}_2)) \arrow[d] \\
R\Gamma_c(\RZ_G) \arrow[r, "u_\calC"]            & R\Gamma_c(\RZ_G),   
\end{tikzcd}
\end{equation}
and we consider the induced map on $2(d-r)$-degree compactly supported cohomology, which is just given by the inclusion map on irreducible components.
\subsubsection{}
\label{section:nonzerob}
In particular, we see that we can determine the irreducible components of $p-\Isog\otimes\kappa$ from the action of the cohomological correspondence on the cohomology $R\Gamma_c(\RZ_G)$, in the following way. Suppose we know that a $[b]$-dense irreducible component $C$ is of the form $\tilde{\pi}_\infty(X_1\times Y\times X_2)$, where we know the irreducible component $X_1$, but not $X_2$. Applying the correspondence $f^!(u_\calC))$ to the cohomological class $c(X_1')$ for some $X_1'\cap X_1\neq\emptyset$, commutativity of the diagram \eqref{eqn:comm1} tells us we will get $\sum_i b_ic(h_iX_2)$, where $h_i\in H$, for some $b_i\in\bbZ$. Note that $J_b(\bbQ_p)$-equivariance implies that all the $b_i$ are equal. As long as not all $b_i$ are zero, we can recover $C$ as the image of $X_1\times Y\times h_iX_2$, since $J_b(\bbQ_p)$-equivariance of $\tilde{\pi}_\infty$ implies that
\begin{equation*}
    \tilde{\pi}_\infty(X_1\times Y\times h_iX_2)=\tilde{\pi}_\infty(h_i^{-1}X_1\times h_i^{-1}Y\times X_2)=\tilde{\pi}_\infty(X_1\times Y\times X_2).
\end{equation*}
\begin{definition}
\label{defnHN}
Let $([b],\upsilon)$ be such that $[b]\in B(G,\upsilon)$, and $M$ is a Levi subgroup containing $M_{[b]}$. $([b],\upsilon)$ is said to be Hodge-Newton decomposable for $M$ if $\kappa_M([b])=\upsilon^{\sharp}\in \pi_1(M)_\Gamma$.
\end{definition}
\begin{proposition}
\label{prop:nonzerob}
If either
\begin{itemize}
    \item The pair $([b],\upsilon)$ is Hodge Newton decomposable for $M_{[b]}$, or
    \item The Shimura variety is a Hilbert-Blumenthal moduli scheme and $p$ is inert
\end{itemize}
then the integers $b_i$ cannot all be non-zero.
\end{proposition}
\begin{proof}
By the definition of $f'^!(u_\calD)$, we have a commutative diagram
\begin{equation*}
    \begin{tikzcd}
R\Gamma_c(\RZ_G(X_1')^{rig},\bbQ_l) \arrow[r, "f'^!(u_\calD)"] \arrow[d] & R\Gamma_c(\RZ_G(\tilde{X}_2)^{rig},\bbQ_l) \arrow[d] \\
R\Gamma_c(\calS,\bbQ_l) \arrow[r, "u_\calD"]            & R\Gamma_c(\calS,\bbQ_l),    
\end{tikzcd}
\end{equation*}
thus, to show that the map on the top is non-zero, it suffices to show that $u_\calD$ restricted to the image of $R\Gamma_c(\RZ(X_1')^{rig},\bbQ_l)$ is non-zero. We can specialize the correspondence in the bottom row to the special fiber, observing that \begin{equation*}
    R\Gamma_c(\sS_\kappa,\bbQ_l)\simeq R\Gamma_c(\calS,\bbQ_l),
\end{equation*}
since $\hat{S}$ is a smooth formal scheme, and the image of $c(X_1')$ in $R\Gamma_c(\sS_\kappa,\bbQ_l)$ is simply the image $[Y'_1]\in H^{2d-2r}_c(\sS_\kappa)$ under the cycle class map on $\sS_\kappa$, where $Y'_1=\pi_\infty(X_1'\times\{id\}))$. Moreover, by definition, the cohomological correspondence $u_D$ takes $[Y_1']$ to (the image under the cycle class map) of $p_{2*}(p_1^*([Y_1'])\cdot[D])$, where $p_1,p_2$ are the projection maps from $\sS_\kappa\times\sS_\kappa$ to $\sS_\kappa$. To show that this is non-zero, it suffices to show that $p_1^*([Y_1'])\cdot[D]$ is non-zero.

In the situation where $[b]$ is Hodge-Newton decomposable for $M_{[b]}$, recall that we have an isomorphism of special fibers 
\begin{equation}
\label{eqn:HNisom1}
    \RZ(G,b,\mu)^{red}\simeq \RZ(M_{[b]},b,\mu)^{red}.
\end{equation}
We will use this result to do the above calculation of $p_1^*([Y_1'])\cdot[D]$ on a different Shimura variety, one where the Rapoport-Zink space is the one associated to the basic locus.

Observe that the correspondence given by $u_\calD$ on $\RZ(M_{[b]},b,\mu)^{red}$ via the isomorphism  \eqref{eqn:HNisom1} may also be defined as follows. We have a correspondence
\begin{equation*}
    X_1\leftarrow X_1\times \tilde{X_2}\rightarrow \tilde{X_2},
\end{equation*}
where $\tilde{X_2}$ is the $H$-orbit of $X_2$. Here $H$ is, as before, the hyperspecial subgroup of $J_b$ which stabilizes $X_1$. There is some large positive integer $N$ such that for every $\alpha\in X_1,\beta\in \tilde{X_2}$, $p^N\beta^{-1}\alpha$ is an isogeny; this defines a universal isogeny over $X_1\times \tilde{X_2}$. Consider the closed formal subscheme $\mathfrak{A}$ of $\RZ_M(X_1)\times \RZ_M(\tilde{X_2})$ over which the universal isogeny lifts. Then $\mathcal{A}=\mathfrak{A}^{rig}$ is finite \'{e}tale over $\RZ(X_1)^{rig}$, $\RZ(\tilde{X_2})^{rig}$, and we have a natural cohomological correspondence from $\RZ(X_1)^{rig}$ to $\RZ(\tilde{X_2})^{rig}$ supported on $\calA$ over the rigid analytic generic fibers. We can take the reduction of this correspondence; it gives a correspondence on the special fiber supported on $X_1\times \tilde{X_2}$.   

Furthermore, observe that we can also pass to the adjoint group $M_{[b]}^{ad}$, because from Section \ref{section:connectedcomponentADLV} there is an isomorphism of connected components, and thus we can consider the correspondence on $u_{\calD^{ad}}$ on $\RZ(M_{[b]}^{ad},b^{ad},\mu^{ad})^{red}$ supported on
\begin{equation*}
    X_1^{ad}\leftarrow X_1^{ad}\times \tilde{X_2}^{ad}\rightarrow \tilde{X_2}^{ad}.
\end{equation*}
Note that $\tilde{X_2}^{ad}$ is the $H^{ad}$-orbit of $X_2^{ad}$. 

Thus, we may assume that $M$ is adjoint. We can construct an abelian type Shimura variety such that the Rapoport-Zink space associated to the basic locus is $\RZ(M_{[b]},b,\mu)$, as follows. Since $M_{[b]}$ is quasi-split, and $\mu$ is minuscule, using the root datumn we may construct a connected reductive group $M'$ over $\bbQ$ with a $M'(\bbR)$-conjugacy class $X'$ determined by $\mu$ such that $(M',X')$ is a Shimura datumn. Since $G$ is a subgroup of $\GSp_{2n}$, $M'$ is a classical group, and hence $(M',X')$ is also classical, hence of type $A,B,C,D$. Thus, $(M',X')$ is an abelian type Shimura datumn. Let $(M'_1,X'_1)$ be some associated Hodge-type Shimura datumn. If we consider the basic locus $[b'_1]$ of the Shimura variety, this is non-empty by the main result of \cite{lee2018}, hence there is some Rapoport-Zink space $\RZ(M'_1,b'_1,\mu_1')$ whose adjoint data is the same as $(M_{[b]}^{ad},b^{ad},\mu^{ad})$. Let $X'_1,X'_2$ be irreducible components of $\RZ(M'_1,b'_1,\mu_1')$ which map to $X_1^{ad},X_2^{ad}$ respectively under the adjoint quotient; and consider some irreducible component $D_1'$ of the moduli space of $p$-power quasi isogenies $p-\Isog_{M'_1}$ for $\Sh(M'_1,X'_1)$ given as the image under the almost product structure map of $X'_1\times X'_2$. This is an irreducible component entirely supported in the basic locus, and it suffices to check that the induced cohomological correspondence on the basic locus of the Shimura variety $\Sh(M'_1,X'_1)$ is also non-trivial. 

To see that the correspondence acts non-trivially, it suffices to observe that the self-intersection of irreducible components of the basic locus is always non-zero. We see this from the calculation of the intersection matrix in \cite[\S7.4.3]{XZ17}, and the fact that the determinant is always non-zero, hence there is always some other irreducible component $X_1'$ whose intersection product $X_1'\cdot X_1=\sum_ia_i[Z_i]$ is non-trivial, and $p_1^*(X_1')\cdot D_1'=\sum_i a_i[p_1^{-1}(Z_i)\cap C]$ which is non-zero and supported on cycles of dimension $\dim(X_1)$ and hence the projection $p_{2,*}(p_1^*(X_1')\cdot D_1')$ is also non-zero, so $u_\calD([X_1'])$ is also non-zero.

In the case of Hilbert modular varieties, we will also show that not all $b_i$ are non-zero. We first observe that for the Hilbert modular variety attached to a totally real field of degree $g$ over $\bbQ$, we have $\lfloor \frac{g}{2}\rfloor$ unramified $\sigma$-conjugacy classes, as follows. For $r=1,\dots,\lfloor \frac{g}{2}\rfloor$, the slope of the Newton polygon is given by $((\frac{r}{g})^{(g)},(\frac{g-r}{g})^{(g)})$. The geometry of these Newton strata is described in \cite{TX19}. We briefly recall this construction here. For a fixed $r$, we let the $\sigma$-conjugacy class be $[b]$. There are $\binom{g}{r}$ $\bbMV$ cycles, and each of them corresponds to a periodic semi-meander $\mathfrak{a}$ with $g$ nodes and $r$ arcs. Each $\mathfrak{a}$ gives a generalized Goren-Oort strata $X_\mathfrak{a}$, which is the closure of some irreducible component of the Newton strata $\sS^{[b]}$. Each $X_\mathfrak{a}$ is an iterated $\mathbb{P}^1$-fibration over another quaternionic Shimura variety $S_\mathfrak{a}$. In particular, we observe that all $X_\mathfrak{a}$ are smooth, and the irreducible component of the affine Deligne-Lusztig variety $X^{\mathfrak{a}}_\mu(b)$ is isomorphic to an iterated $\mathbb{P}^1$-bundle.

Moreover, we note that $C^{[b]}$ is smooth. This is because the complete local ring is given as a product $\mathfrak{N}(X)\times\mathfrak{C}(X)$, hence the image of the map embedding $C\rightarrow \sS\times \sS$ is a regular embedding, since the complete local ring at any point in $C^{[b]}$ is given by the image of
\begin{equation*}
    \mathfrak{N}(X)\times\mathfrak{C}(X)\times \mathfrak{N}(X)\xrightarrow{id\times\Delta\times id} \mathfrak{N}(X)\times\mathfrak{C}(X)\times\mathfrak{C}(X)\times \mathfrak{N}(X)\hookrightarrow \mathrm{Def}(X)\times\mathrm{Def}(X),
\end{equation*}
which is clearly smooth.

Thus, to show that the terms $b_i$ are not all zero, we want to determine the intersection product $p_1^*(X_\mathfrak{a})\cdot C$. Firstly, observe that the self-intersection product $X_\mathfrak{a}\cdot X_\mathfrak{a}$ is non-zero; in particular, it is given by $(-2)^a p^b[Z]$, where $Z$ is the zero-section of iterated $\mathbb{P}^1$-bundle, from \cite[Theorem 4.3(2)]{TX19} and its proof. Thus, by applying the excess intersection formula, we see that the intersection $p_1^*(X_\mathfrak{a})\cdot C$ is exactly $(-2)^a p^b[p_1^{-1}(Z)\cap C]$. Since $p_1^{-1}(Z)\cap C$ is non-empty, $p_2(p_1^{-1}(Z)\cap C)$ is of dimension $r$, hence the projection $p_{2,*}([p_1^{-1}(Z)\cap C]$ is clearly non-zero. 
\end{proof}
\begin{remark}
\label{conj:nonzero}
We conjecture that the integers $b_i$ are non-zero for general Shimura varieties of Hodge type.
\end{remark}
The main theorem of this section is the following.
\begin{theorem}
\label{thm:parabolicred}
Let $C$ be a $[b]$-dense irreducible component of $p-\Isog\otimes\kappa$. There exists a map $\bar{h}$ (depending on $C$)
\begin{equation*}
    \bar{h}:\calH(M_{[b]}(\bbQ_p)//M_{[b]}(\bbZ_p),\bbQ)\rightarrow \bbQ[p-\Isog\otimes\kappa]
\end{equation*}
such that for any element $f\in\mathcal{H}(G(\mathbb{Q}_p)//K_p,\mathbb{Q})$, as cohomological correspondences acting on $H^{2r}_c(\RZ(G,b,\upsilon)^{red})$, we have
\begin{equation}
\label{eqn:71equal}
    u_{C\cdot h(f)}=u_{\bar{h}(\dot{\calS}^G_M(f))},
\end{equation}
where $\dot{\calS}^G_M$ is the twisted Satake homomorphism. Moreover, for any $f_1,f_2\in\mathcal{H}(G(\mathbb{Q}_p)//K_p,\mathbb{Q})$, we have 
\begin{equation*}
    u_{C\cdot h(f_1\cdot f_2)}=u_{\bar{h}(\dot{\calS}^G_M(f_1\cdot f_2))}.
\end{equation*}

If in addition we also assume that the Shimura datumn satisfies the same assumptions in Proposition \ref{prop:nonzerob}, then we have an equality in $\bbQ[p-\Isog\otimes\kappa]$
\begin{equation}
\label{eqn:72equal}
    C\cdot h(f)=\bar{h}(\dot{\calS}^G_M(f)).
\end{equation}
and moreover if we let $z_m=b\sigma(b)\cdots \sigma^{nm-1}(b)$, then
    \begin{equation*}
        p^{nmr}\bar{h}(1_{z_mM(\bbZ_p)})=(C\cdot \mathrm{Frob}^m)
    \end{equation*}
where $r=\dim \RZ^{red}_{G,b}$, and $m$ be the smallest positive integer such that $\sigma^{nm}(b)=b$

\end{theorem}

\begin{remark}
When $b$ is $\mu$-ordinary, and $C$ is the formal sum of the irreducible components of the identity section, this construction recovers the map $\bar{h}$ constructed in \eqref{eqn:barh}.
\end{remark}
\begin{proof}
We first show how to construct the map $\bar{h}$. Let $C$ be any $[b]$-dense irreducible component of $p-\Isog\otimes\kappa$, given by $\tilde{\pi}_\infty(X_1\times Z\times X_2)$. Consider any $f\in\calH(G)$, and consider the product $C\times h(f)\in \bbQ[p-\Isog\otimes\kappa]$. We can write
\begin{equation*}
    C\times h(f)=\sum_{i}a_iC_i
\end{equation*}
where $C_i$ will be a $[b]$-dense irreducible component of $p-\Isog\otimes\kappa$ of the form $\tilde{\pi}_\infty(X_1\times Z\times Y_i)$, for some $Y_i$ irreducible component of $\RZ_G^{red}$. We may consider the associated cohomological correspondences, and observe that 
\begin{equation*}
    u_{C\cdot h(f)}=u_{h(f)}\circ u_{C}, 
\end{equation*}
since the $!$-pullback of composition of cohomological correspondences is the composition of the $!$-pullbacks.           
We now construct $\bar{h}$. Consider the class $c(X_2)\in H^{2n-2r}_c(\RZ_G^{rig},\bbQ_l)$. Observe that its image lies in $H^{2n'-2r'}_c(\RZ(M,b,\mu')^{rig},\bbQ_l)$ for some $\mu'\in I_{\mu}$. For any $f'\in\calH(M)$, let 
\begin{equation*}
    \sum_j a_j c(W_j)
\end{equation*}
be the image of $c(X_2)$ under the action of $f'$ on $H^{2n'-2r'}_c(\RZ(M,b,\mu')^{rig},\bbQ_l)$. Then we let 
\begin{equation*}
    \bar{h}(f')=\sum_j b_j\tilde{\pi}_\infty(X_1\times Z\times W_j),
\end{equation*}
where the factor $b_j$ is such that if $D:=\tilde{\pi}_\infty(X_1\times Z\times W_j)$, then $u_\calD(c(X_1))=\frac{a_j}{b_j}\sum_{h\in H}c(hW_j)$. It is clear from the construction that 
\begin{equation*}
    u_{C\cdot h(f)}=u_{\bar{h}(\calS^G_M(f))},
\end{equation*}and 
\begin{equation*}
    u_{C\cdot h(f_1\cdot f_2)}=u_{\bar{h}(\calS^G_M(f_1\cdot f_2))}.
\end{equation*}
If in addition we also assume that the Shimura datumn satisfies the same assumptions in Proposition \ref{prop:nonzerob}, then the equality of cohomological correspondences implies the equality in $\bbQ[p-\Isog\otimes\kappa]$, since the difference $C\cdot h(f)-\bar{h}(\calS^G_M(f))$ will consist of some irreducible components, which must all act non-trivially as a cohomological correspondence.

We now compare the action with Frobenius. We first show the following lemma. Let $C'=\bar{h}(1_{z_mM_c})$, where we note that $z_m\in J_b(\bbQ_p)$ is central.

\begin{lemma}
\label{lemma:frobaction}
We have the following equality of underlying closed subschemes:
\begin{equation*}
    \widetilde{C\cdot \Frob^m}=C'
\end{equation*}
\end{lemma}
\begin{proof}
We now consider any pair $(x,f)$ which lies in $C^{[b]}$, such that $f$ is a $p$-quasi-isogeny between $x$ and $y$ for some $y\in\sS_\kappa$. Then the image of $(x,f)$ under multiplication by $\Frob^m$ is the pair $(x,f_1)$, where $f_1$ is the composition
\begin{equation*}
    \sG_x\rightarrow\sG_y\rightarrow\sG_{hy},
\end{equation*}
where $h=b_y\sigma(b_y)\dots\sigma^{nm-1}(b_y)$.

By definition of $C'$, observe that $1_{z_mM_c}$ takes the irreducible component $X_2$ to the irreducible component $Z':=z_mX_2$, since $z_m$ is central. In particular, $h':=z_mb_y$ lies in $Z'$. 

Thus, we need to show that $h$ lies in $Z'$. We first show that $h,h'$ lie in the same connected component of $X_\upsilon(b)$. Let $b_y=g^{-1}b\sigma(g)$, for some $g\in X_\upsilon(b)$. Thus, we see that we need to show that $g,\sigma^{nm}(g)$ lie in the same connected component. By Theorem \ref{thm:CKV}, we are reduced to showing that $\sigma^{nm}(c_{b,\upsilon})=c_{b,\upsilon}$, since  $\tilde{\kappa}_M(\sigma^{nm}(m))=\sigma^{nm}(\tilde{\kappa}_M(m))$. By the construction of $z_m$, we see that $c_{b,\upsilon}=\tilde{\kappa}_M(z_m)$, and since we could have chosen $b=\omega(\upsilon)(p)$, $z_m$ is clearly invariant under $\sigma^{nm}$.

By \eqref{section:connectedcomponentADLV}, we see that to show $h,h'$ are in the same irreducible component, we are reduced to showing that $h_{ad},h'_{ad}$ lie in the same irreducible component of $X^{G^{\mathrm{ad}}}_\upsilon(b_{y,\mathrm{ad}})$. Hence, we may assume $G$ is adjoint, so that $Z_G$ is trivial.

Let $\mathbf{b}\in \cup_{\lambda\in\tau+(\sigma-1)X_*(T)}\mathbb{MV}_\upsilon(\lambda)$ be such that $g\in X^\mathbf{b}_\upsilon(b)$. We replace $b$ with $p^{\tau_\mathbf{b}}$, where $\tau_{\mathbf{b}}$ is as defined in \eqref{lem:taub}. Let $g'$ be such that $g'^{-1}p^{\tau_\mathbf{b}}\sigma(g')=b_y$. Then we also have $g'\in X^\mathbf{b}_\upsilon(p^{\tau_\mathbf{b}})$. Observe that since $\tau_\mathbf{b}$ is dominant, we must have
\begin{equation*}
    \tau_\mathbf{b}+\dots+\sigma^{nm-1}(\tau_\mathbf{b})=\beta+\dots+\sigma^{nm-1}(\beta),
\end{equation*}
since $\overline{\tau}_\mathbf{b}=\bar{\beta}$, and both $\tau_\mathbf{b}$ and $\beta$ are fixed by $\sigma^{nm}$. Thus, we see that to show $h,h'$ lie in the same irreducible component of $X_\upsilon(b_y)$, we are reduced to showing that $g',\sigma^{nm}(g')$ lie in the same irreducible component of $X^{\mathbf{b}}_\upsilon(p^{\tau_\mathbf{b}})$. This follows from Proposition \ref{prop:sigmafix}.
\end{proof}
The last part of Property (3) follows by simply observing that since $r=\dim \RZ_G^{red}$, the Frobenius is a map of degree $p^{nr}$, which exactly accounts for the factors.
\end{proof}
\section{Proof of Main Theorem}
In this section, we give the proof of Theorem \ref{thm:algcycles}. The key idea is to give, for each unramified $[b]\in B(G,\upsilon)$, a factor $H_{[b]}(x)$ of the Hecke polynomial which kills the irreducible components of $p-\Isog\otimes\kappa$ which are $[b]$-dense. 
\subsection{Constructing $H_{[b]}$}
\label{section:constructionHb}
Our first observation is that the factors of $H_{G,X}$ are closely related to unramified elements in $B(G,\upsilon)$. In particular, we observe that the unramified elements $p^\tau$ are exactly the elements which are $\sigma$-conjugate to $p^\lambda$ for some $\lambda\in \Omega(\overline{\bbQ}_p)\upsilon$.

Hence, for an unramified $\sigma$-conjugacy class $[b]$=$[p^\tau]$ which we fix for the rest of the section, consider the set
\begin{equation*}
    \mathbb{MV}(\tau)=\{\lambda\in \Omega(\overline{\bbQ}_p)\upsilon:\lambda\in\tau+(\sigma-1)X_*(T)\}.
\end{equation*}
We let $m$ be the smallest positive integer such that $\sigma^{nm}$ fixes every element in this set. Choose $\lambda$ in $\mathbb{MV}(\tau)$ such that $\sigma^{ni}(\lambda)\neq \lambda$ for any $i<m$. 

Moreover, since $\tau$ is $\sigma$-conjugate to $\lambda$, it is also $\sigma$-conjugate to $\beta:=\sigma^{-1}(\lambda)$ which is a Weyl conjugate of $\mu^{-1}$. We also observe that $\sigma^{ni}(\beta)\neq \beta$ for any $i<m$. 

Let $Z$ be the galois orbit of $\Omega(\overline{\bbQ}_p)\mu^{-1}$ containing $\beta$. We assume that $\bar{\beta}$ is dominant, otherwise we replace $\beta$ with conjugate in $\Omega(\bbQ_p)$, which has the same properties. Then applying Proposition \ref{prop:factorHecke}, we see that $H_{G,X}(x)$ has a factor of the form
\begin{equation}
\label{eqn:heckepoly1}
    H'_{[b]}(x):=x^m-p^{nm\langle\rho,\mu-\beta\rangle }\tilde{\beta},
\end{equation}
where $\tilde{\beta}=\sum_{i=1}^{nm}\sigma(\beta)$. However, we only know that this factorization holds over the larger Hecke algebra $\calH(M_{\tilde{\beta}}(\bbQ_p)//M_{\tilde{\beta},c})$, and so \emph{a priori}, we do not have a factorization in $\bbQ[p-\Isog\otimes\kappa]$. However, we can consider the product of all the Weyl conjugates of $H'_{[b]}(x)$, which we will denote by $H_{[b]}(x)$. $H_{[b]}(x)$ is defined in $\calH(G)(x)$. For notational simplicity, we let $M:=M_{\tilde{\beta}}$ be the standard Levi subgroup centralizing the cocharacter $\tilde{\beta}$. $M$ also centralizes the Newton cocharacter $\nu([b])$.

\begin{proposition}
\label{prop:Hbkillsfactors}
  Let $C$ be a geometric irreducible component of $p-\Isog\otimes\kappa$ which is $[b]$-dense. Then we have
  \begin{equation*}
    H_{[b]}(\Frob)\cdot C=0.
  \end{equation*}
\end{proposition}
\begin{proof}
Since $\dot{\calS}^G_M(H_{[b]})$ factorizes in $\calH(M)(x)$, and has a factor $H'_{[b]}$, from the compatibility with $\dot{\calS}^G_M$ in Theorem \ref{thm:parabolicred}, we see that $H_{[b]}(\Frob)\cdot C$ has a factor $H'_{[b]}(\Frob)\cdot C$, and we know that $H_{[b]}'(\Frob)\cdot C=0$, also from theorem \ref{thm:parabolicred}.
\end{proof}
\subsection{Key vanishing result} 
Using the previous results, we can now show the following theorem.
\begin{theorem}
\label{thm:keyvanishing}
 For any unramified $[b]\in B(G,\upsilon)$, let $H_{G,X}^{\succeq [b]}(x)$ be the polynomial with coefficients in $\mathbb{Q}[p-\Isog\otimes\kappa]$ given as the image of $H_{G,X}(x)$ under $\mathrm{res}^{\succeq [b]}\circ h$. Then 
 \begin{equation*}
     H_{G,X}^{\succeq [b]}(\Frob)=0,
 \end{equation*}
 where $\Frob$ is the Frobenius section of $p-\Isog\otimes\kappa$.
\end{theorem}
\begin{proof}
  We show this by induction on $[b]$, under the reverse of the partial order in $B(G,\upsilon)$. We already know this holds on the $\mu$-ordinary locus, so suppose that for all $[b']\succ[b]$, the analogous result is true. Note that by an inclusion-exclusion principle sum, this implies that $H^{\succ[b]}_{G,X}(\Frob)=0$, where $H^{\succ[b]}_{G,X}(\Frob)$ is the sum of all $[b']$-dense terms, where $[b']\succ[b]$.
  
 We know that $H^{\succeq [b]}_{G,X}(\Frob)$ admits a factorisation
  \begin{equation*}
      H^{\succeq [b]}_{G,X}(\Frob)=H_{[b]}(\Frob)P(\Frob)
  \end{equation*}
  for some $P(\Frob)$, all of whose irreducible components are $[b']$-dense for $[b']\succeq [b]$. Moreover, $H_{[b]}(\Frob)$ is $\mu$-ordinary dense. Consider the difference
  \begin{equation*}
      H^{\succeq [b]}_{G,X}(\Frob)-H^{\succ[b]}_{G,X}(\Frob)=H_{[b]}(\Frob)P(\Frob)-H_{[b]}(\Frob)P'(\Frob)
  \end{equation*}
where $P'(\Frob)$ is the restriction to the $[b']$-dense irreducible components of $P(\Frob)$, for $[b']\succ [b]$. Then $P(\Frob)-P'(\Frob)\in \mathbb{Q}[p-\Isog\otimes\kappa]$ must be $[b]$-dense, by definition. Hence, we have
  \begin{equation*}
      H^{\succeq [b]}_{G,X}(\Frob)-H^{\succ[b]}_{G,X}(\Frob)=H_{[b]}(\Frob)\cdot C,
  \end{equation*}
  for some $[b]$-dense $C$. Applying Proposition \ref{prop:Hbkillsfactors} above, we can conclude that $H^{\succeq [b]}_{G,X}(\Frob)-H^{\succ[b]}_{G,X}(\Frob)=0$, and hence 
  \begin{equation*}
      H^{\succeq [b]}_{G,X}(\Frob)=0.
  \end{equation*}
\end{proof}

Theorem \ref{thm:keyvanishing} clearly implies Theorem \ref{thm:algcycles}, since we have the vanishing result for all unramified $[b]$ in $B(G,\upsilon)$.
\subsection{Passage to cohomology}
\label{subsubsection:7.2}
We record here the method to get from the congruence relation on the level of algebraic cycles (Theorem \ref{thm:algcycles}) to the congruence relation on \'{e}tale cohomology (Theorem \ref{mainthm}). 

We first note that by the main result of \cite{KMP19}, there exists a smooth projective toroidal compactification of $\sS_{K}(G,X)$, and hence by proper base change there is a canonical $\mathrm{Gal}_{E}$-equivariant isomorphism
\begin{equation}
\label{eqn:7.3}
    H^i_{\mathrm{\acute{e}t}}(\Sh_K(G,X)_{\bar{E}},\bbQ_l)\xlongrightarrow{\sim}H^i_{\mathrm{\acute{e}t}}(\sS_K(G,X)_{\bar{\kappa}},\bbQ_l).
\end{equation}
Moreover, the action of $\calH(G(\bbQ_p)//K_p,\bbQ)$ on the left naturally extends to an action on the right. Given any algebraic cycle $C$ in $\bbQ[p-\Isog\otimes\kappa]$, we see that the image of $C$ under the map
\begin{equation*}
    p-\Isog\otimes\kappa\xrightarrow{s,t} \sS_K(G,X)_\kappa \times \sS_K(G,X)_\kappa
\end{equation*}
defines a correspondence on $\sS_K(G,X)_\kappa$, with an induced action on cohomology. Since the congruence relation holds in $\bbQ[p-\Isog\otimes\kappa]$, it also holds in the ring of algebraic correspondences $\mathrm{Corr}(\sS_K(G,X)_\kappa,\sS_K(G,X)_\kappa)$. We thus have the relation on $H^i_{\mathrm{\acute{e}t}}(\sS_K(G,X)_{\bar{\kappa}},\bbQ_l)$, and thus, via the isomorphism above, the congruence relation on $H^i_{\mathrm{\acute{e}t}}(\Sh_K(G,X)_{\bar{E}},\bbQ_l)$. This in fact holds with $\bbQ_l$ replaced by any \'{e}tale sheaf $\mathcal{V}$ over $\sS_K(G,X)$, since we have a similar isomorphism as in \eqref{eqn:7.3} with $\bbQ_l$ replaced with $\mathcal{V}$, see \cite[Thm. 6.7]{LS18}. A similar argument holds for compactly supported cohomology $H^i_{\mathrm{\acute{e}t,c}}(\Sh_K(G,X)_{\bar{E}},\mathcal{V})$. 

In the case of intersection cohomology of the Baily-Borel compactification, we know, also from \cite{KMP19}, that there is a projective normal model $\sS^{\mathrm{min}}_K(G,X)$ over $O_{E,(v)}$ of the Baily-Borel compactification of $\Sh_K(G,X)$. Moreover, here we can apply \cite[Thm. 6.7]{LS18} as well, and we have an isomorphism of \'{e}tale intersection cohomology groups
\begin{equation*}
\mathbf{IH}^i(X_{\bar{E}},\overline{\bbQ}_l)\xrightarrow{\sim}\mathbf{IH}^i(X_{\bar{\kappa}},\overline{\bbQ}_l).
\end{equation*}
The same result also holds if we replace the intersection complex of $\Sh_K(G,X)$ with the sheaf $(j_{!,*}(\mathcal{V}[n]))[-n]$, where $j:\Sh_K(G,X)\hookrightarrow \Sh^{BB}_K(G,X)$, $n=\dim\Sh_K(G,X)$, and $\mathcal{V}$ is an automorphic $\lambda$-adic sheaf.

Since $\sS_K(G,X)_\kappa$ is open and dense in $\sS^{\mathrm{min}}_K(G,X)_\kappa$, we are reduced to showing that the congruence relation in $\mathrm{Corr}(\sS_K(G,X)_\kappa,\sS_K(G,X)_\kappa)$ implies, after taking closure of algebraic cycles, the congruence relation in $\mathrm{Corr}(\sS^{\mathrm{min}}_K(G,X)_\kappa,\sS^{\mathrm{min}}_K(G,X)_\kappa)$. To see this, let $C\subset \Sh_K(G,X)\times\Sh_K(G,X)$ be an algebraic correspondence corresponding to the action of some double coset $1_{K_pgK_p}\in \calH(G(\bbQ_p)//K_p,\bbQ)$, and $\mathscr{C}$ the closure of $C$ in $\sS^{\mathrm{min}}_K(G,X)\times \sS^{\mathrm{min}}_K(G,X)$. The boundary components $\sS^{\mathrm{min}}_K(G,X)\backslash \sS_K(G,X)$ consists of (flat integral models of) Hermitian symmetric domains, and over $E$, the induced left action of $g\in G(\mathbb{A}_f)$ on the boundary gives a finite correspondence on the boundary. Thus, if we let $\mathscr{C}'$ be the restriction of $\mathscr{C}$ to $\sS_K(G,X)\times\sS_K(G,X)$, we observe that $\mathscr{C}\backslash\mathscr{C'}$ is of dimension strictly less than $\dim \sS_K(G,X)$, since this is true over the generic fiber. Hence, we see that $\mathscr{C}'_\kappa$ is dense in $\mathscr{C}_\kappa$. Thus, taking closure allows us to conclude that we have the congruence relation in $\mathrm{Corr}(\sS^{\mathrm{min}}_K(G,X)_\kappa,\sS^{\mathrm{min}}_K(G,X)_\kappa)$, as desired.
\appendix
\section{Cohomological Correspondences}
\label{appendix1}
\subsection{}Let $X$ be a rigid analytic space over $E$, or a scheme over $k$. We denote by $D^b_c(X,\bbQ_l)$ the bounded constructible derived category of $l$-adic sheaves over $X$.

We recall the formalism of cohomological correspondences:
\begin{definition}
Let $(X_i,\calF_i)$ for $i = 1,2$ be two pairs, where $X_i$ are either both rigid analytic spaces over $E$, or both schemes over $k$, and $\calF_i \in D^b_c(X_i,\bbQ_l)$. A cohomological correspondence $(C,u) : (X_1,\calF_1) \rightarrow (X_2,\calF_2)$ is a rigid analytic space over $E$ (resp. scheme over $k$) $C$ equipped with morphisms $c_1: C\rightarrow X_1$, and $c_2: C\rightarrow X_2$, and a morphism in $D^b_{c}(C,\bbQ_l)$ 
\begin{equation*}
    u:c_1^*(\calF)\rightarrow c_2^!(\calF).
\end{equation*}
\end{definition}
Given the commutative diagram
\begin{equation*}
    \begin{tikzcd}
 X_1 \arrow[d,"f_1"] & C \arrow[l,"d_1"] \arrow[r,"d_2"] \arrow[d,"f"] & X_2 \arrow[d,"f_2"] \\
Y_1           & D \arrow[r,"c_1"] \arrow[l,"c_2"]           & Y_2          
\end{tikzcd}
\end{equation*}
where the left square is cartesian, we can define the $!$-pullback of cohomological correspondences from $D$ to $C$, where the induced map is
\begin{equation*}
    f^!(u):d_1^*f_1^!\calF_1\xrightarrow{BC^{*!}}f^!c_1^*\calF_1\xrightarrow{f^!u}f^!c_2^!\calF_2=d_2^!f_2^!\calF_2.
\end{equation*}
The base change map $BC^*!$ is defined in \cite[A.2]{XZ17}.
\begin{proposition}
\label{prop:shriekpullback}
Suppose the left square is Cartesian. Then the following diagram is commutative:
\begin{equation*}
    \begin{tikzcd}
H^*_c(X_1,f_1^!\calF) \arrow[r,"H(f^!(u))"] \arrow[d] & H^*_c(X_2,f_2^!\calF) \arrow[d] \\
H^*_c(Y_1,\calF) \arrow[r,"H(u)"]           & H^*_c(Y_2,\calF)          
\end{tikzcd}
\end{equation*}
where the vertical maps are the pushforward maps associated to $f_1,f_2$.
\end{proposition}
\begin{proof}
This result follows from the functoriality of the base-change map $BC^*!$ for cartesian diagrams.
\end{proof}
Similarly, given the commutative diagram
\begin{equation*}
    \begin{tikzcd}
 X_1 \arrow[d,"f_1"] & C \arrow[l,"d_1"] \arrow[r,"d_2"] \arrow[d,"f"] & X_2 \arrow[d,"f_2"] \\
Y_1           & D \arrow[r,"c_1"] \arrow[l,"c_2"]           & Y_2          
\end{tikzcd}
\end{equation*}
where the right square is cartesian, we can define the $*$-pullback of cohomological correspondences from $D$ to $C$, where the induced map is
\begin{equation*}
    f^*(u):d_1^*f_1^*\calF_1=f^*c_1^*\calF_1\xrightarrow{f^*u}f^*c_2^!\calF_2\xrightarrow{BC^{*!}}d_2^!f_2^!\calF_2
\end{equation*}
\begin{proposition}
\label{prop:starpullback}
Suppose the both squares are Cartesian, and the vertical maps $f_1,f_2$ are proper. Then the following diagram is commutative:
\begin{equation*}
    \begin{tikzcd}
H^*_c(X_1,f_1^*\calF) \arrow[r,"H(f^*(u))"]  & H^*_c(X_2,f_2^*\calF)  \\
H^*_c(Y_1,\calF) \arrow[r,"H(u)"]  \arrow[u]         & H^*_c(Y_2,\calF)   \arrow[u]       
\end{tikzcd}
\end{equation*}
where the vertical maps are the pullback maps associated to $f_1,f_2$.
\end{proposition}
\begin{proof}
This result follows from the functoriality of the base-change map $BC^*!$ for cartesian diagrams.
\end{proof}
The following proposition follows from the above two propositions. We denote by 
\begin{equation*}
    H_c^\bullet(X,\bbQ_l)=\sum_{i}(-1)^iH^i_c(X,\bbQ_l),
\end{equation*}
the alternating sum of cohomology groups.
\begin{proposition}
\label{prop:excision}
Let $U$ be an open subset of $X$, and $Z=X-U$, with $j:U\hookrightarrow X$ and $i:Z\hookrightarrow X$ the inclusion maps. Suppose we have a commutative diagram, where all squares are Cartesian
\begin{equation*}
    \begin{tikzcd}
U \arrow[d] & D \arrow[d] \arrow[l] \arrow[r] & U \arrow[d] \\
X           & C \arrow[r] \arrow[l]           & X           \\
Z \arrow[u] & E \arrow[r] \arrow[l] \arrow[u] & Z \arrow[u]
\end{tikzcd}
\end{equation*}
then the isomorphism
\begin{equation*}
    H_c^\bullet(X,\bbQ_l)\simeq H_c^\bullet(U,\bbQ_l)+H_c^\bullet(Z,\bbQ_l)
\end{equation*}
is equivariant for the action of the cohomological correspondence $u$, where the action on the right is given by $j^!(u)+i^*(u)$.
\end{proposition}
\section{Multiplication on $p-\Isog\otimes\kappa$}
\label{app:b}
For notational simplicity, for the entirety of this appendix we will denote $p-\Isog\otimes\kappa$ by $p-\Isog$.
We have the map
\begin{equation*}
    c:p-\Isog\times_\sS p-\Isog \rightarrow p-\Isog
\end{equation*}
given by composition. More precisely, given an $R$-point of $p-\Isog\times_\sS p-\Isog$ given as a pair $((y,f:\calA_y\rightarrow \calA_z),(x,g:\calA_x\rightarrow \calA_y))$, where $x,y\in \sS(R)$ and $f,g$ are quasi-isogenies, its image under $c$ is the point given by $(x,f\circ g)$. The map $c$ is proper.

We also have two proper maps $s,t$ (``source'' and ``target'' respectively)  
\begin{equation*}
    s,t: p-\Isog\rightarrow \sS.
\end{equation*}

Thus, we have the refined Gysin map $i^!:A_{2n}(p-\Isog\times_\kappa p-\Isog)\rightarrow A_{n}(p-\Isog\times_\sS p-\Isog)$ induced by the Cartesian square
\begin{equation*}
\begin{tikzcd}
p-\Isog\times_\sS p-\Isog\arrow[r] \arrow[d] & \sS\times_\kappa \sS\times_\kappa \sS \arrow[d, "i"] \\
p-\Isog\times_\kappa p-\Isog \arrow[r, "s\times t\times s\times t"]           &      \sS\times_\kappa \sS\times_\kappa \sS  \times_\kappa \sS.        
\end{tikzcd}
\end{equation*}

(Here, the map $i$ is the diagonal embedding on the second component of $\sS\times \sS\times \sS$.)

Given $C,D$ dimension $n$ irreducible components of $p-\Isog\otimes\kappa$, we define $C\cdot D$ to be 
\begin{equation*}
    C\cdot D=c_*(i^!(C\times_\kappa D)).
\end{equation*}

Note that if $C\times_\sS D$ is of dimension $n$, and the map $C\times_\sS D\rightarrow C\times_\kappa D$ is a closed regular embedding, then $i^!(C\times_\kappa D)=[C\times_\sS D]$. For example, if $D$ is a section of the map $s$ (for example, $D$ is $\mathrm{Frob}$, the Frobenius section), then since $C\times_\sS D\simeq C$ is irreducible of dimension $n$, and a closed regular embedding, $i^!(C\times_\kappa D)=[C\times_\sS D]$.

If we view $C,D$ as algebraic correspondences in $\sS\times_\kappa \sS$ through the pushforward via the map
\begin{equation*}
    s\times t:p-\Isog\rightarrow \sS\times_\kappa \sS,
\end{equation*}
then compatibility of the refined Gysin maps with proper pushforward together with the following commutative square 
\begin{equation*}
    \begin{tikzcd}
p-\Isog\times_\sS p-\Isog \arrow[d, "(s\times t)\times_{\sS} (s\times t)"'] \arrow[r, "c"] & p-\Isog \arrow[d, "{(s\times t)}"] \\
(\sS\times_\kappa \sS)\times_\sS (\sS\times_\kappa \sS) \arrow[r, "p"']               & \sS\times_\kappa \sS                
\end{tikzcd}
\end{equation*}
tells us that
\begin{equation*}
    (s\times t)_*(C\cdot D)=(s\times t)_*(C)\cdot (s\times t)_*(D),
\end{equation*}
where the product of correspondences is the one defined in \cite[Appendix A]{B2002}.
\bibliographystyle{alpha}
\bibliography{ref}

\begin{thebibliography}{GHKR06}

\bibitem[B\"02]{B2002}
Oliver B\"{u}ltel.
\newblock The congruence relation in the non-{PEL} case.
\newblock {\em J. Reine Angew. Math.}, 544:133--159, 2002.

\bibitem[Bor79]{borel1979}
Armand Borel.
\newblock Automorphic l-functions.
\newblock In {\em Automorphic forms, representations and L-functions (Proc.
  Sympos. Pure Math., Oregon State Univ., Corvallis, Ore., 1977), Part},
  volume~2, pages 27--61, 1979.

\bibitem[BR94]{BR1994}
Don Blasius and Jonathan~D. Rogawski.
\newblock Zeta functions of {S}himura varieties.
\newblock In {\em Motives ({S}eattle, {WA}, 1991)}, volume~55 of {\em Proc.
  Sympos. Pure Math.}, pages 525--571. Amer. Math. Soc., Providence, RI, 1994.

\bibitem[BW06]{BW2006}
Oliver B\"{u}ltel and Torsten Wedhorn.
\newblock Congruence relations for {S}himura varieties associated to some
  unitary groups.
\newblock {\em J. Inst. Math. Jussieu}, 5(2):229--261, 2006.

\bibitem[CKV15]{CKV}
Miaofen Chen, Mark Kisin, and Eva Viehmann.
\newblock Connected components of affine {D}eligne-{L}usztig varieties in mixed
  characteristic.
\newblock {\em Compos. Math.}, 151(9):1697--1762, 2015.

\bibitem[CTS79]{CS79}
J.-L. Colliot-Th\'{e}l\'{e}ne and J.-J. Sansuc.
\newblock Fibr\'{e}s quadratiques et composantes connexes r\'{e}elles.
\newblock {\em Math. Ann.}, 244(2):105--134, 1979.

\bibitem[Fal99]{Fal99}
Gerd Faltings.
\newblock Integral crystalline cohomology over very ramified valuation rings.
\newblock {\em J. Amer. Math. Soc.}, 12(1):117--144, 1999.

\bibitem[FC90]{FC1990}
Gerd Faltings and Ching-Li Chai.
\newblock {\em Degeneration of abelian varieties}, volume~22 of {\em Results in
  Mathematics and Related Areas (3)}.
\newblock Springer-Verlag, Berlin, 1990.

\bibitem[Fon90]{Fon90}
Jean-Marc Fontaine.
\newblock Repr\'{e}sentations {$p$}-adiques des corps locaux. {I}.
\newblock In {\em The {G}rothendieck {F}estschrift, {V}ol. {II}}, volume~87 of
  {\em Progr. Math.}, pages 249--309. Birkh\"{a}user Boston, Boston, MA, 1990.

\bibitem[GHKR06]{GHKR}
Ulrich G\"{o}rtz, Thomas~J. Haines, Robert~E. Kottwitz, and Daniel~C. Reuman.
\newblock Dimensions of some affine {D}eligne-{L}usztig varieties.
\newblock {\em Ann. Sci. \'{E}cole Norm. Sup. (4)}, 39(3):467--511, 2006.

\bibitem[GHN19]{GHN19}
Ulrich G\"{o}rtz, Xuhua He, and Sian Nie.
\newblock Fully {H}odge{\textendash}{N}ewton {D}ecomposable {S}himura
  {V}arieties.
\newblock {\em Peking Mathematical Journal}, 2(2):99--154, 2019.

\bibitem[Ham17]{Ham2017}
Paul Hamacher.
\newblock The almost product structure of {N}ewton strata in the deformation
  space of a {B}arsotti-{T}ate group with crystalline {T}ate tensors.
\newblock {\em Math. Z.}, 287(3-4):1255--1277, 2017.

\bibitem[HK19]{HK2019adic}
Paul Hamacher and Wansu Kim.
\newblock l-adic {\'e}tale cohomology of shimura varieties of hodge type with
  non-trivial coefficients.
\newblock {\em Mathematische Annalen}, 375(3-4):973--1044, 2019.

\bibitem[HV18]{HV2018}
Paul Hamacher and Eva Viehmann.
\newblock Irreducible components of minuscule affine {D}eligne-{L}usztig
  varieties.
\newblock {\em Algebra Number Theory}, 12(7):1611--1634, 2018.

\bibitem[Kim18]{kim_2018}
Wansu Kim.
\newblock Rapoport-{Z}ink {S}paces of {H}odge {T}ype.
\newblock {\em Forum of Mathematics, Sigma}, 6, 2018.

\bibitem[Kis10]{K2010}
Mark Kisin.
\newblock Integral models for {S}himura varieties of abelian type.
\newblock {\em J. Amer. Math. Soc.}, 23(4):967--1012, 2010.

\bibitem[Kis17]{K2017}
Mark Kisin.
\newblock Mod $p$ points on {S}himura varieties of abelian type.
\newblock {\em J. Amer. Math. Soc.}, 30(3):819--914, 2017.

\bibitem[KMPS]{KMPS}
Mark Kisin, Keerthi Madapusi-Pera, and Sug-Woo Shin.
\newblock Honda-{T}ate {T}heory.
\newblock \url{https://math.berkeley.edu/~swshin/HT.pdf}.

\bibitem[Kos14]{Kos2014}
Jean-Stefan Koskivirta.
\newblock Congruence relations for {S}himura varieties associated with
  {$GU(n-1,1)$}.
\newblock {\em Canad. J. Math.}, 66(6):1305--1326, 2014.

\bibitem[Kot85]{Kott85}
Robert~E. Kottwitz.
\newblock Isocrystals with additional structure.
\newblock {\em Compositio Math.}, 56(2):201--220, 1985.

\bibitem[Kot97]{Ko97}
Robert~E. Kottwitz.
\newblock Isocrystals with additional structure. {II}.
\newblock {\em Compositio Math.}, 109(3):255--339, 1997.

\bibitem[Lee18]{lee2018}
Dong~Uk Lee.
\newblock Nonemptiness of {N}ewton strata of {S}himura varieties of {H}odge
  type.
\newblock {\em Algebra Number Theory}, 12(2):259--283, 2018.

\bibitem[Li18]{Li18}
Hao Li.
\newblock Congruence relations of {GS}pin {S}himura varieties.
\newblock \url{arXiv:1812.11261}, 2018.

\bibitem[Lov17]{lov2017}
Tom Lovering.
\newblock Filtered {F}-crystals on {S}himura varieties of abelian type.
\newblock \url{arXiv:1702.06611}, 2017.

\bibitem[LR87]{LR}
R.~P. Langlands and M.~Rapoport.
\newblock Shimuravariet\"{a}ten und {G}erben.
\newblock {\em J. Reine Angew. Math.}, 378:113--220, 1987.

\bibitem[LS18]{LS18}
Kai-Wen Lan and Beno\^{\i}t Stroh.
\newblock Nearby cycles of automorphic \'{e}tale sheaves.
\newblock {\em Compos. Math.}, 154(1):80--119, 2018.

\bibitem[Man05]{mantovan2005cohomology}
Elena Mantovan.
\newblock On the cohomology of certain {PEL}-type shimura varieties.
\newblock {\em Duke Mathematical Journal}, 129(3):573--610, 2005.

\bibitem[Man08]{Man08}
Elena Mantovan.
\newblock On non-basic {R}apoport-{Z}ink spaces.
\newblock {\em Ann. Sci. \'{E}c. Norm. Sup\'{e}r. (4)}, 41(5):671--716, 2008.

\bibitem[Moo04]{M2004}
Ben Moonen.
\newblock Serre-{T}ate theory for moduli spaces of {PEL} type.
\newblock {\em Ann. Sci. \'{E}cole Norm. Sup. (4)}, 37(2):223--269, 2004.

\bibitem[MP19]{KMP19}
Keerthi Madapusi~Pera.
\newblock Toroidal compactifications of integral models of {S}himura varieties
  of {H}odge type.
\newblock {\em Ann. Sci. \'{E}c. Norm. Sup\'{e}r. (4)}, 52(2):393--514, 2019.

\bibitem[Nek18]{Nek2018}
Jan Nekov\'{a}\v{r}.
\newblock Eichler-{S}himura relations and semisimplicity of \'{e}tale
  cohomology of quaternionic shimura varieties.
\newblock {\em Ann. Sci. \'{E}cole Norm. Sup. (4)}, 51(5):1179--1252, 2018.

\bibitem[RV14]{RV2014}
Michael Rapoport and Eva Viehmann.
\newblock Towards a theory of local {S}himura varieties.
\newblock {\em M\"{u}nster J. Math.}, 7(1):273--326, 2014.

\bibitem[Sta97]{St1997}
Hellmuth Stamm.
\newblock On the reduction of the {H}ilbert-{B}lumenthal-moduli scheme with
  {$\Gamma_0(p)$}-level structure.
\newblock {\em Forum Math.}, 9(4):405--455, 1997.

\bibitem[SZ16]{SZ2016}
Ananth~N. Shankar and Rong Zhou.
\newblock Serre-{T}ate theory for {S}himura varieties of {H}odge type.
\newblock \url{arXiv:1612.06456}, 2016.

\bibitem[TX19]{TX19}
Yichao Tian and Liang Xiao.
\newblock Tate cycles on some quaternionic {S}himura varieties {${\rm
  mod}\,p$}.
\newblock {\em Duke Math. J.}, 168(9):1551--1639, 2019.

\bibitem[Wed00]{W2000}
Torsten Wedhorn.
\newblock Congruence relations on some {S}himura varieties.
\newblock {\em J. Reine Angew. Math.}, 524:43--71, 2000.

\bibitem[Win05]{Win05}
J.-P. Wintenberger.
\newblock Existence de {$F$}-cristaux avec structures suppl\'{e}mentaires.
\newblock {\em Adv. Math.}, 190(1):196--224, 2005.

\bibitem[Wor13]{Wor2013}
Daniel Wortmann.
\newblock The $\mu$-ordinary locus for {S}himura varieties of {H}odge type.
\newblock \url{arXiv:1310.6444}, 2013.

\bibitem[XZ17]{XZ17}
Liang Xiao and Xinwen Zhu.
\newblock Cycles on {S}himura varieties via geometric {S}atake.
\newblock \url{arXiv:1707.05700}, 2017.

\bibitem[Zha15]{Zha2015}
Chao Zhang.
\newblock Stratifications and foliations for good reductions of {S}himura
  varieties of hodge type.
\newblock \url{arXiv:1512.08102}, 2015.

\bibitem[Zhu17]{Zhu2014}
Xinwen Zhu.
\newblock Affine {G}rassmannians and the geometric {S}atake in mixed
  characteristic.
\newblock {\em Ann. of Math. (2)}, 185(2):403--492, 2017.

\bibitem[Zin01]{zink2001slope}
Thomas Zink.
\newblock On the slope filtration.
\newblock {\em Duke Mathematical Journal}, 109(1):79--95, 2001.

\end{thebibliography}

\end{document}